\documentclass[11pt]{amsart}
\usepackage{tikz,amsthm,amsmath,amstext,amssymb,amscd,epsfig,euscript, mathrsfs, dsfont,pspicture,multicol,graphpap,graphics,graphicx,times,enumerate,subfig,sidecap,wrapfig,color}

\usepackage{amsmath}
\usepackage{amssymb}
\usepackage{pgf}
\usepackage{stackrel}
\usepackage[colorlinks,citecolor=red]{hyperref}

\newcommand{\R}{{\mathbb R}}       
\newcommand{\Z}{{\mathbb Z}}       
\newcommand{\BB}{{\mathcal B}}
\newcommand{\DD}{{\mathcal D}}
\newcommand{\CC}{{\mathcal C}}
\newcommand{\FF}{{\mathcal F}}
\newcommand{\HH}{{\mathcal H}}

\newcommand{\MM}{{\mathcal M}}
\newcommand{\WW}{{\mathcal W}}

\newcommand{\AZ}{{\mathcal A}}

\newcommand{\EE}{{\mathcal E}}

\newcommand{\TT}{{\mathcal T}}

\newcommand{\om}{{\Omega}}
\newcommand{\hm}{{\omega}}

\newcommand{\diam}{\mathop{\rm diam}}
\newcommand{\dist}{{\rm dist}}
\newcommand{\ds}{\displaystyle }

\newcommand{\fiproof}{{\hspace*{\fill} $\square$ \vspace{2pt}}}

\newcommand{\rf}[1]{{(\ref{#1})}}

\newcommand{\supp}{\operatorname{supp}}

\newcommand{\vphi}{{\varphi}}
\newcommand{\ve}{{\varepsilon}}
\newcommand{\vv}{{\vspace{2mm}}}
\newcommand{\vvv}{\vspace{4mm}}
\newcommand{\wt}[1]{{\widetilde{#1}}}
\newcommand{\wh}[1]{{\widehat{#1}}}

\newcommand{\noi}{\noindent}
\newcommand{\lip}{{\rm Lip}}
\newcommand{\pom}{{\partial \Omega}}

\newcommand{\bad}{{\mathsf{Bad}}}

\newcommand{\HD}{{\mathsf{HD}}}
\newcommand{\LD}{{\mathsf{LD}}}
\newcommand{\sss}{{\mathsf{Stop}}}
\newcommand{\ttt}{{\mathsf{Top}}}
\newcommand{\tree}{{\mathsf{Tree}}}

\newcommand{\ls}{{\mathsf{LS}}}
\newcommand{\whsa}{{\mathsf{WHSA}}}
\newcommand{\wts}{{\mathsf{WTS}}}
\newcommand{\wtn}{{\mathsf{WTN}}}
\newcommand{\baup}{{\mathsf{BAUP}}}
\newcommand{\batpp}{{\mathsf{BATPP}}}
\newcommand{\case}{{\mathsf{Type}}}

\def\Xint#1{\mathchoice
{\XXint\displaystyle\textstyle{#1}}%
{\XXint\textstyle\scriptstyle{#1}}%
{\XXint\scriptstyle\scriptscriptstyle{#1}}%
{\XXint\scriptscriptstyle\scriptscriptstyle{#1}}%
\!\int}
\def\XXint#1#2#3{{\setbox0=\hbox{$#1{#2#3}{\int}$ }
\vcenter{\hbox{$#2#3$ }}\kern-.58\wd0}}

\def\avint{\Xint-}

\newtheorem{theorem}{Theorem}[section]
\newtheorem{lemma}[theorem]{Lemma}

\newtheorem{corollary}[theorem]{Corollary}

\newtheorem{coro}[theorem]{Corollary}
\newtheorem{proposition}[theorem]{Proposition}
\newtheorem{propo}[theorem]{Proposition}

\newtheorem*{lemma*}{Lemma}
\newtheorem*{theorem*}{Theorem}

\theoremstyle{definition}
\newtheorem{definition}[theorem]{Definition}

\theoremstyle{remark}
\newtheorem{rem}[theorem]{\bf Remark}

\numberwithin{equation}{section}

\newcommand{\RRem}{\begin{rem}}
\newcommand{\erem}{\end{rem}}

\def\d{\partial}

\newcommand{\mih}[1]{\marginpar{\color{red} \scriptsize \textbf{Mi:} #1}}

\makeatletter
\def\@tocline#1#2#3#4#5#6#7{\relax
  \ifnum #1>\c@tocdepth 
  \else
    \par \addpenalty\@secpenalty\addvspace{#2}%
    \begingroup \hyphenpenalty\@M
    \@ifempty{#4}{%
      \@tempdima\csname r@tocindent\number#1\endcsname\relax
    }{%
      \@tempdima#4\relax
    }%
    \parindent\z@ \leftskip#3\relax \advance\leftskip\@tempdima\relax
    \rightskip\@pnumwidth plus4em \parfillskip-\@pnumwidth
    #5\leavevmode\hskip-\@tempdima
      \ifcase #1
       \or\or \hskip 1em \or \hskip 2em \else \hskip 3em \fi%
      #6\nobreak\relax
    \dotfill\hbox to\@pnumwidth{\@tocpagenum{#7}}\par
    \nobreak
    \endgroup
  \fi}
\makeatother

\def\cB{{\mathscr{B}}}
\def\cC{{\mathscr{C}}}

\def\bR{{\mathbb{R}}}
\def\bS{{\mathbb{S}}}

\def\bN{{\mathbb{N}}}

\newcommand{\ps}[1]{\left( #1 \right)}

\newcommand{\isif}[1]{\left\{\begin{array}{cc} #1
\end{array}\right.}

\newcommand{\divv}{{\text{{\rm div}}}}

\def\Car{\textrm{Car}}
\def\gec{\gtrsim}
\def\lec{\lesssim}

\def\loc{\textrm{loc}}
\def\grad{\nabla}
\def\loc{\textrm{loc}}

\def\bR{\mathbb{R}}
\def\lec{\lesssim}
\def\Lip{\textrm{Lip}}

\begin{document}

\title[Uniform rectifiability and elliptic measure]{Uniform rectifiability, elliptic measure, square functions, and $\ve$-approximability via an ACF monotonicity formula}


\author[Azzam]{Jonas Azzam}

\address{Jonas Azzam\\
School of Mathematics \\ University of Edinburgh \\ JCMB, Kings Buildings \\
Mayfield Road, Edinburgh,
EH9 3JZ, Scotland.}
\email{j.azzam "at" ed.ac.uk}

\newcommand{\jonas}[1]{\marginpar{\color{magenta} \scriptsize \textbf{Jonas:} #1}}

\author[Garnett]{John Garnett}

\address{John Garnett\\
Department of Mathematics, 6363 Math Sciences Building \\University of California
at Los Angeles, Los Angeles, California 90095-1555.
}
\email{jbg@mat.ucla.edu}

\author[Mourgoglou]{Mihalis Mourgoglou}

\address{Mihalis Mourgoglou\\
BCAM - Basque Center for Applied Mathematics\\
Mazarredo, 14 E48009 Bilbao, Spain\\
Departamento de Matem\`aticas, Universidad del Pa\' is Vasco, Aptdo. 644, 48080 Bilbao, Spain and\\
Ikerbasque, Basque Foundation for Science, Bilbao, Spain.
}
\email{michail.mourgoglou@ehu.eus}

\author[Tolsa]{Xavier Tolsa}
\address{Xavier Tolsa
\\
ICREA, Passeig Lluís Companys 23 08010 Barcelona, Catalonia, and\\
Departament de Matem\`atiques and BGSMath
\\
Universitat Aut\`onoma de Barcelona
\\
Edifici C Facultat de Ci\`encies
\\
08193 Bellaterra (Barcelona), Catalonia
}
\email{xtolsa@mat.uab.cat}

\thanks{J.G. was supported by NSF Grant DMS 1217239. M.M. was supported  by the Basque Government through IKERBASQUE and the BERC 2014-2017 program, and by Spanish Ministry of Economy and Competitiveness MINECO: BCAM Severo Ochoa excellence accreditation SEV-2013-0323 and MTM2014-53850. X.T. was supported by the ERC grant 320501 of the European Research Council (FP7/2007-2013) and partially supported by MTM-2013-44304-P, MTM-2016-77635-P,  MDM-2014-044 (MICINN, Spain), 2014-SGR-75 (Catalonia), and by Marie Curie ITN MAnET (FP7-607647).
}

\begin{abstract}
Let $\Omega\subset\R^{n+1}$, $n\geq2$, be an open set with Ahlfors-David regular boundary that satisfies the corkscrew condition. We consider a uniformly elliptic operator $L$ in divergence form associated with a matrix $A$ with real, merely bounded and possibly non-symmetric coefficients, which are also locally Lipschitz and satisfy suitable Carleson type estimates.
In this paper we show that if $L^*$ is the operator in divergence form associated with the transpose matrix of $A$, then $\partial\Omega$ is uniformly $n$-rectifiable if and only if every bounded solution of $Lu=0$ and every bounded solution of $L^*v=0$ in $\Omega$ is
$\varepsilon$-approximmable if and only if every bounded solution of $Lu=0$ and every bounded solution of $L^*v=0$ in $\Omega$  satisfies a suitable square-function Carleson measure estimate. Moreover, we obtain two additional criteria for uniform rectifiability. One is given in terms of the so-called ``$S<N$'' estimates, and another in terms of a suitable corona decomposition involving $L$-harmonic and $L^*$-harmonic measures. We also prove that if $L$-harmonic measure and $L^*$-harmonic measure satisfy a weak $A_\infty$-type condition, then $\d \Omega$ is $n$-uniformly rectifiable. In the process we obtain a version of Alt-Caffarelli-Friedman monotonicity formula for a fairly wide class of elliptic operators which is of independent interest and plays a fundamental role in our arguments.

\end{abstract}

\maketitle

\tableofcontents

\section{Introduction}

Let  $\Omega \subset \R^{n+1}$ be open and $A=(a_{ij})_{1 \leq i,j \leq n+1}$ be a matrix with real measurable coefficients in $\Omega$. We say that $A$ is {\it uniformly elliptic} in $\Omega$ with constant $\Lambda \geq 1$ if it satisfies the following conditions:
\begin{align}\label{eqelliptic1}
&\Lambda^{-1}|\xi|^2\leq \langle A(x) \xi,\xi\rangle,\quad \mbox{ for all $\xi \in\R^{n+1}$ and a.e. $x\in\Omega$.}\\
&\langle A(x) \xi,\eta \rangle  \leq\Lambda |\xi| |\eta|, \quad \mbox{ for all $\xi, \eta \in\R^{n+1}$ and a.e. $x\in\Omega$.} \label{eqelliptic2}
\end{align}
Notice that the matrix $A$ is {\it possibly non-symmetric}. If 
$$L = -{\rm div} (A(\cdot)\nabla)$$
 is an elliptic operator of divergence form associated with a uniformly elliptic matrix $A$ as above in $\Omega$ we write $L \in \mathcal{L}(\Omega)$. We shall write $L \in \mathcal{L}_{\Car}(\Omega)$ if, in addition, $a_{ij} \in \Lip_{\loc}(\Omega)$ and the following Carleson condition holds:
\begin{equation}\label{eqelliptic3}
\sup_{\substack{x\in\partial\Omega\\r>0}} \frac1{r^n} \int_{B(x,r)\cap\Omega} 
\biggl(\,
\sup_{\substack{z_1,z_2\in B(y, M \delta_\Omega(y)) \cap \Omega\\ \delta_\Omega(z_k)\geq \frac{1}{4}\,\delta_\Omega(y)}}
\frac{|a_{ij}(z_1) - a_{ij}(z_2)|}{|z_1-z_2|}\biggr)\,dy \leq C,
\end{equation}
for $1\leq i,j \leq n+1$, where $\delta_\Omega(x):=\dist(x,\partial\Omega)$, $dy$ stands for the Lebesgue measure in $\R^{n+1}$,  and $M \geq 4$ is a large constant. It is clear that if $A^*$ is the transpose matrix of $A$ and $L^*$ the uniformly elliptic operator associated with $A^*$, then $L \in \mathcal{L}(\Omega) $ (resp. $\mathcal{L}_\Car(\Omega) $) implies $L^* \in \mathcal{L}(\Omega)$ (resp. $ \mathcal{L}_\Car(\Omega) $).  A measurable function $u: \Omega \to \R$ that satisfies the equation $Lu=0$ in the weak sense is called {\it $L$-harmonic}.


In this paper we  characterize  uniform $n$-rectifiability in $\R^{n+1}$, $n\geq 2$, in terms of approximability and in terms of 
Carleson measure estimates for square functions involving bounded $L$-harmonic and $L^*$-harmonic functions. 

To state our results in more detail, we need now to introduce some further definitions and notation.
A set $E\subset \R^d$ is called $n$-{\textit {rectifiable}} if there are Lipschitz maps
$f_i:\R^n\to\R^d$, $i=1,2,\ldots$, such that 
\begin{equation}\label{eq001}
\HH^n\biggl(E\setminus\bigcup_i f_i(\R^n)\biggr) = 0,
\end{equation}
where $\HH^n$ stands for the $n$-dimensional Hausdorff measure. 

A set $E\subset\R^{d}$ is called $n$-AD-{\textit {regular}} (or just AD-regular or Ahlfors-David regular) if there exists some
constant $C_{0}>0$ such that
$$C_0^{-1}r^n\leq \HH^n(B(x,r)\cap E)\leq C_0\,r^n\quad \mbox{ for all $x\in
E$ and $0<r\leq \diam(E)$.}$$
The set $E\subset\R^{d}$ is  uniformly  $n$-{\textit {rectifiable}} if it is 
$n$-AD-regular and
there exist constants $\theta, M >0$ such that for all $x \in E$ and all $0<r\leq \diam(E)$ 
there is a Lipschitz mapping $g$ from the ball $B_n(0,r)$ in $\R^{n}$ to $\R^d$ with $\text{Lip}(g) \leq M$ such that
$$
\HH^n (E\cap B(x,r)\cap g(B_{n}(0,r)))\geq \theta r^{n}.$$


The analogous notions for  measures are the following. 
A Radon measure $\mu$ on $\R^d$ is $n$-{\textit {rectifiable}} if it vanishes outside  an $n$-rectifiable
set $E\subset\R^d$ and if moreover $\mu$ is absolutely continuous with respect to $\HH^n|_E$.
On the other hand, $\mu$ is called $n$-AD-{\textit {regular}} if it is of the form $\mu=g\,\HH^n|_E$, where 
$E$ is $n$-AD-regular and $g:E\to (0,+\infty)$ satisfies $g(x)\approx1 $ for all $x\in E$, with the implicit constant independent of $x$.
If, moreover, $E$ is uniformly $n$-rectifiable, then $\mu$ is called uniformly $n$-{\textit {rectifiable}}.

 We say that an open set $\Omega\subset\R^{n+1}$ satisfies the {\textit {corkscrew condition}} if for every ball $B(x,r)$ with
 $x\in\partial\Omega$ and $0<r\leq\diam(\Omega)$ there exists another ball $B(x',r')\subset \Omega\cap B(x,r)$
 with radius $r'\approx r$, with the implicit constant independent of $x$ and $r$. 
Let us remark that we do not ask $\Omega$ to be connected. For example, if $E\subset\R^{n+1}$
is a closed $n$-AD-regular set, then it follows easily that $\R^{n+1}\setminus E$ satisfies the corkscrew condition.

We say that $\Omega\subset \mathbb{R}^{n+1}$ satisfies the {\it Harnack chain condition} if there is a uniform constant $C$ such that for every $\rho >0,\, \Lambda\geq 1$, and every pair of points
$x,y \in \Omega$ with $\min(\dist(x,\d\Omega),\,\dist(y,\d\Omega)) \geq\rho$ and $|x-y|<\Lambda\,\rho$, there is a chain of
open balls
$B_1,\dots,B_N \subset \Omega$, $N\leq C(\Lambda)$,
with $x\in B_1,\, y\in B_N,$ $B_k\cap B_{k+1}\neq \varnothing$
and $C^{-1}\diam (B_k) \leq \dist (B_k,\partial\Omega)\leq C\diam (B_k).$  The chain of balls is called
a {\it Harnack Chain}. Note that if such a chain exists, then
\[u(x)\approx_{N}u(y),\] for any positive $L$-harmonic function $u$.

Let $u$ be a  
bounded $L$-harmonic function on $\Omega$.  For $\ve > 0$ we say that
$u$ is $\ve$-\textit{approximable} 
if for every ball $B$ centered at the boundary $\d \Omega$ of radius $r(B)\in  (0, \diam(\Omega))$, there is $\varphi=\varphi_B \in W^{1,1}(B \cap \Omega)$ and $C>0$ such that 
\begin{equation}
\label{epapp}
\|u - \varphi\|_{L^{\infty}(B \cap\Omega)} < \varepsilon
\end{equation}
\and
\begin{equation}
\label{epCar}
\frac{1}{r(B)^n} \int_{B\cap \Omega} |\nabla \varphi(y)| \,dy \leq C(\|u\|_{L^\infty(\Omega)},\ve).
\end{equation}
An analogous notion of ``global'' $\ve$-approximability, where there the function $\vphi$ is taken independently of $B$
, was introduced by Varopoulos in \cite{Var}
in connection with corona problems. See \cite[Chapter VIII]{Garnett} for some applications and a
proof on the upper half plane and \cite{KKoPT}, \cite{KKiPT} \cite{HMM2} and \cite{Pipher} for surveys of more recent applications. A local version similar to the one above was introduced in \cite{HKMP} and it was used to prove the $A_\infty$ property of elliptic measure for non-symmetric time independent elliptic equations in the upper half-space.

We are now ready to state our main results.
\vv

\begin{theorem} \label{teo1}
Let $\Omega\subset\R^{n+1}$, $n\geq2$, be a domain with $n$-AD-regular boundary satisfying the corkscrew condition, and let $L \in \mathcal{L}_\Car(\Omega)$.  
Then the following conditions are equivalent:
\begin{itemize}

\item[(a)] Every bounded  $L$-harmonic and every bounded $L^*$-harmonic function in $\Omega$ is $\varepsilon$-approximable for all $\varepsilon  > 0.$ 
\vv
\item[(b)] There is $C>0$ such that  every bounded $L$-harmonic and every bounded $L^*$-harmonic function on $\Omega$ satisfies the following: if $B$ is a ball centered at $\partial\Omega$, then 
\begin{equation}\label{eqhip1}
\int_B |\nabla u(x)|^2\,\dist(x,\partial\Omega)\,dx\leq C\,\|u\|^2_{L^\infty(\Omega)}\,r(B)^n.
\end{equation}
\item[(c)]  $\partial \Omega$ is uniformly rectifiable.
\end{itemize}
\end{theorem}
\vv

Theorem \ref{teo1} has been proved recently in \cite{GMT}  in the special case when $-L$ is the Laplace operator. The ``one direction" of the theorem had already appeared in  \cite{HMM2}, where Hofmann, Martell and Mayboroda showed that the statements (a)\footnote{In fact, they proved the stronger ``global version'' of $\ve$-approximability, where $\vphi$ is independent of $B$.}
 and (b) hold for open subsets of $\R^{n+1}$ with uniformly  $n$-rectifiable boundary satisfying the corkscrew condition. In fact, they first showed it for the Laplace operator but almost the same proofs work also for uniformly elliptic operators in divergence form with coefficients satisfying \eqref{eqelliptic1}, \eqref{eqelliptic2} and \eqref{eqelliptic3}\footnote{This follows from the fact that for the class of operators considered in the theorem,  $L$-harmonic measure is $A_\infty$ in NTA domains.
 }. In the present paper we show that $(a)\Rightarrow(c)$ and $(b)\Rightarrow (c)$.

{
The simultaneous absence of Harnack chains and self-adjointeness of the matrix  in the assumptions of the theorem 
makes the problem very challenging (indeed, if one assumed either the uniformity of the domain or the elliptic matrix $A$ to be symmetric, then the proof of the aforementioned theorem would be concluded at the end of Section \ref{secur}). Analogous difficulties are not rare in boundary value problems with rough data for real, time-independent, uniformly elliptic operators in related contexts. 
Recall, for example, the beautiful proof in \cite{HKMP} of the solvability of the 
Dirichlet problem in $\R^{n+1}_+$ after the study of other particular cases of elliptic matrices in different previous works 
by several authors (e.g., symmetric, block, triangular).}

\begin{rem}\label{remb'}
In fact, we will show that the following statement, which is weaker than (b) in Theorem \ref{teo1}, suffices to prove that 
 $\partial \Omega$ is uniformly rectifiable:
\begin{itemize}
\item[(b$'$)] {\em There is $C>0$ such that  every bounded continuous $L$-harmonic and every bounded continuous $L^*$-harmonic function on $\Omega$ satisfies the following: if $B$ is a ball centered at $\partial\Omega$, then }
\begin{equation}\label{eqhip1'}
\int_B |\nabla u(x)|^2\,\dist(x,\partial\Omega)\,dx\leq C\,\|u\|^2_{L^\infty(\Omega)}\,r(B)^n.
\end{equation}
\end{itemize}
\end{rem}

To prove $(a)\Rightarrow(c)$ and $(b')\Rightarrow (c)$ in Theorem \ref{teo1} we first show both (a) or (b')  imply the existence of a corona decomposition in terms of the elliptic measures associated to $L$ and $L^*$. A
similar idea was already used in \cite{GMT} for the case of the Laplacian. One important difference
is that now we have to deal with two different elliptic measures instead of a single one.  

Next, we show that the existence of such corona decomposition implies the uniform rectifiability of
$\partial\Omega$. This is the most difficult (and newer) step of the arguments in the present paper.
In \cite{GMT} this was shown by using the connection between Riesz transforms  and harmonic measure, in combination with the rectifiability criterion from \cite{NToV}. 
However, in the more general situation of Theorem  \ref{teo1} the connection with Riesz transforms is not available and consequently we have to use different techniques. 

To show that the existence of the corona decomposition described above implies uniform recitifability first we apply some integration by parts techniques inspired by the nice ideas from \cite{HLMN}.
 In this work, Hofmann, Le, Martell and Nystr\"om (inspired in turn by some ideas from Lewis nad Vogel \cite{LV}) showed that the so called ``weak $A_\infty$ condition'' for the harmonic and $p$-harmonic measure implies the uniform $n$-rectifiability of $\partial\Omega$.
However, in our situation an important obstacle arises when the operator $L$ is not symmetric:
roughly speaking, the corona decomposition above ensures that in many scales and locations some level sets of the Green functions associated to $L$ and $L^*$ are ``large''. In the presence of Harnack chain conditions (i.e. when $\Omega$ is a uniform domain) both level sets intersect and indeed they are the same if $L=L^*$. This is essential for the implementation of the
aforementioned integration by parts. The problem is that in our general situation of Theorem \ref{teo1}
neither the Harnack chain assumption nor the symmetry of $L$ are present and thus we can not ensure the non-empty intersection of those level sets. 

So we need two distinguish two families of cubes of $\partial\Omega$, depending on wether the
associated level sets of the Green functions for $L$ and $L^*$ intersect or not.
Loosely speaking, we show that near the cubes of $\partial\Omega$ for which such intersection is non-empty
$\partial\Omega$ is  well approximated bilaterally by  two parallel planes ($\batpp$)\footnote{This condition should be compared to the ``bilateral approximation by a union of planes''
($\baup$) condition in \cite{DS2} and to the  ``weak half-space approximation'' ($\whsa$) in
\cite{HLMN}.}, while the cubes 
 for which such intersection is empty
satisfy
the property of being ``weak topologically satisfactory'' ($\wts$) (see Definition \ref{WTS}).
This property can be considered as a variant of the 
``weak topologically nice'' ($\wtn$) property introduced by David and Semmes in \cite{DS2}, which is necessary and sufficient for uniform rectifiability in the codimension $1$ case.

An essential tool  to derive that many cubes of $\partial \Omega$ are weak topologically satisfactory
is a suitable version of the Alt-Caffarelli-Friedman (ACF) monotonicity formula for elliptic operators which will be discussed below. The fact that the ACF formula is useful to prove quantitative connectivity (in relation with NTA domains) was first observed by Aguilera,
Caffarelli and Spruck \cite{ACS}. 
As far as we know, this is the first time such a formula is used in problems in connection with
uniform rectifiability in terms of  elliptic or harmonic measure\footnote{The ACF monotonicity formula 
has been used in works in connection with the two-phase problem for harmonic measure such as 
\cite{KPT}, \cite{AMT}, \cite{AMTV}, but not in connection with uniform rectifiability.}.
The last step of the proof of the implications  $(a)\Rightarrow(c)$ and $(b)\Rightarrow (c)$ of Theorem \ref{teo1} consists in obtaining a quite flexible criterion for uniform rectifiability involving
the two types of families $\wts$ and $\batpp$ of cubes, which is also of independent interest (see Proposition 
\ref{propowtf}).
\vv

As a corollary of Theorem \ref{teo1} we deduce another characterization of uniform rectifiability in terms
of a square function - nontangential maximal function estimate (of the type ``$S<N$") in the case $n\geq2$. To state this result we need some additional notation. Given $x\in\partial\Omega$, we consider the {\textit {cone}} with vertex $x \in \d \Omega$ and aperture $\alpha \geq 2$ given by
$$\Gamma(x,\alpha) = \{y\in\Omega:|x-y| < \alpha\,\dist(y,\partial\Omega)\}$$
and for a continuous function $u$ in $\Omega$, we define the {\textit {non-tangential maximal function}}
$$N_{*.\alpha} u(x) = \sup_{y\in\Gamma(x,\alpha)}|u(y)|.$$
For $u\in W^{1,2}_{loc}(\Omega)$ we also define the {\textit {square function}}
$$S_\alpha u(x) = \left(\int_{y\in\Gamma(x,\alpha)}|\nabla u(y)|^2\,\dist(y,\partial\Omega)^{1-n}\,dy\right)^{1/2}.$$
When $\alpha=2$ we just write $S u$ and $N_* u$. 
Then we have:

\begin{coro}\label{coro1}
Let $\Omega\subset\R^{n+1}$, $n\geq2$, be a domain with $n$-AD-regular boundary satisfying the corkscrew condition, and let $L \in \mathcal{L}_\Car (\Omega)$. Set  $\mu = \mathcal{H}^n_{|\d \Omega}$ and let $p,p^*\in [2, \infty)$. Suppose that there exists some constant $C>0$  such that 
\begin{equation}\label{eqhld2}
\|S u\|_{L^p(\mu)} \leq C\,\|N_{*} u\|_{L^p(\mu)}\quad \mbox{ for every $L$-harmonic function $u \in C_0(\overline \Omega)$}
\end{equation}
and
\begin{equation}\label{eqhld2*}
\|S u\|_{L^{p^*}(\mu)} \leq C\,\|N_{*} u\|_{L^{p^*}(\mu)}\quad \mbox{ for every $L^*$-harmonic function $u \in C_0(\overline \Omega)$.}
\end{equation}
Then $\partial\Omega$ is uniformly rectifiable.
\end{coro}

The proof of this corollary is the same as the one of \cite[Corollary 1.2]{GMT} using Lemma \ref{lem:u-decay-infinty} instead of \cite[Lemma 6.1]{GMT}. 

Remark also that if inequality \eqref{eqhld2} (or  \eqref{eqhld2*}) holds for some $\alpha > 2$ then it also holds for $a=2$ with constant depending on $\alpha$. This readily follows from the fact that if $\Omega$ is as above, then for $p\geq 2$, $\|N_{*} u\|_{L^p(\mu)} \lesssim_\alpha \|N_{*,\alpha} u\|_{L^p(\mu)}$ and  $\|S_\alpha u\|_{L^p(\mu)} \lesssim_\alpha \|Su\|_{L^p(\mu)}$. The first inequality 
is trivial, while the second one is just a rearrangement of sums over dyadic cubes in $\DD_\mu$ by an argument in the spirit of (but  simpler than) the one in Lemma \ref{lemcase0}. Thus, there is no loss of generality if we assume \eqref{eqhld2} only for $\alpha=2$.

If $p \in \Omega$, we denote by $\hm^{p}$ and $\hm^{p}_*$ the harmonic measures with pole at $p$ associated to $L$ and $L^*$ respectively.

As a by product of the techniques used to prove Theorem \ref{teo1} we also obtain the following: 

\begin{theorem} \label{teo2}
Let $\Omega\subset\R^{n+1}$, $n\geq2$, be a domain with $n$-AD-regular boundary satisfying the corkscrew condition, and let $L \in \mathcal{L}_\Car(\Omega)$. Let $c_0 \in (0,1)$ small enough depending only on $n$, the AD-regularity and the ellipticity constants. Suppose that there exist $\ve, \ve' \in (0,1)$ with $\ve'$ small enough, such that for every ball $B$ centered at $\d \Omega$ with $\diam(B) \leq \diam(\Omega)$ there exists a corkscrew point $x_B \in \frac{1}{2}B \cap\Omega$ with $\dist(x_B, \d \Omega)\geq c_0 r(B)$, so that the following holds: for any $E \subset B \cap \d \Omega$,
\begin{equation}\label{eq:Amu1}
\mbox{if}\quad\mu(E)\leq \ve\,\mu(B), \quad\text{ then }\quad \hm^{x_B}(E)\leq \ve'\,\hm^{x_B}(B).
\end{equation} 
If the same holds for $\hm^{x_B^*}_*$ with $\ve_*, \ve'_*, c_0^* \in (0,1)$ and corkscrew point $x_B^* \in \frac{1}{2}B \cap\Omega $, then $\partial \Omega$ is uniformly rectifiable.
\end{theorem}

\vv
In the case that $L=-\Delta$, the theorem above was proved first in \cite{HMU} assuming that $\Omega$ is a uniform domain. For general corkscrew domains with $n$-AD-regular boundaries and $L=-\Delta$, this was shown in \cite{HLMN} (and later in \cite{MTo} by different arguments). On the other hand, for non-symmetric elliptic operators in uniform domains, Theorem \ref{teo2} has also been 
proved in \cite{HMT}\footnote{In fact, the assumption in \cite{HMT} on the coefficients of the associated matrix is stated in terms of the gradient of the coefficients, although, in the presence of Harnack chains it is equivalent to \eqref{eqelliptic3}. Although, this is not the case in more general domains as the ones we investigate in the present paper. }.
The assumption that $\Omega$ is uniform, which guaranties a quantitative path connectedness of $\Omega$, is an essential ingredient of the arguments in \cite{HMT}.

\vv
As explained above, an important component of the proof of the Theorems \ref{teo1} and \ref{teo2} is an elliptic analogue of the Alt-Caffarelli-Friedman (ACF) {\it monotonicity formula}. There are some versions of this formula for parabolic (or even more general) operators, such as in \cite{Caffarelli-Kenig}, \cite{CJK}, or \cite{Mat-Pet}. However, for our purposes these variants of the ACF formula are not suitable. So in our paper
we prove another version of this formula for elliptic operators which is new, as far we know, and which may be of independent interest due to its possible applications to other free boundary problems. 

{
For the ACF formula we consider the more general operator
\begin{equation}\label{eqwtl}
\wt L u = - \divv A \nabla u + \vec b\cdot \nabla u + d\, u - \divv(\vec e \,u),
\end{equation}
where $A$ is a uniformly elliptic real measurable  matrix satisfying \rf{eqelliptic1} and \rf{eqelliptic2},  $\vec b =
\vec b(x) $ and $\vec e=\vec e(x)$ are vectors with real $L_{\loc}^\infty( \R^{n+1})$ coordinates and $d$ is a real $L_{\loc}^\infty( \R^{n+1})$ function. 
For $x\in \R^{n+1}$ and $r>0$ we consider $w(x,r)$ such that
\begin{align}\label{eqwxr1}
w(x,r)\geq \sup_{y\in B(x,r)} &|A(y)-A(x)| \\ 
&+\sup_{y\in B(x,r)}|y-x| \,(|\vec b(y)| + |\vec e(y)|)+ \sup_{y\in B(x,r)}|y-x|^2 \,|d(y)|,\notag
\end{align}
so that $w(x,\cdot)$ is right continuous for each fixed $x$. For instance, we can take
\begin{align*}
w(x,r)= \lim_{t\to r+} \Bigl(\,\,\sup_{y\in B(x,t)}&|A(y)-A(x)|\\
&+\sup_{y\in B(x,t)}|y-x| \,(|\vec b(y)| + |\vec e(y)|)+ \sup_{y\in B(x,t)}|y-x|^2 \,|d(y)|\Bigr).
\end{align*}

\vv

Our version of the ACF formula reads as follows.

\begin{theorem} \label{teoACF-elliptic}  Let $\wt L$ be as in \eqref{eqwtl} and 
suppose that 
$A(x)= Id$. Let $x \in  \R^{n+1}$ and $R>0$. Let $u_1,u_2\in
W^{1,2}(B(x,R))\cap C(B(x,R))$ be nonnegative $L$-subharmonic functions such that $u_1(x)=u_2(x)=0$ and $u_1\cdot u_2\equiv 0$. 
Set
\begin{equation}\label{eqACF2}
J(x,r) = \left(\frac{1}{r^{2}} \int_{B(x,r)} \frac{|\grad u_1(y)|^{2}}{|y-x|^{n-1}}dy\right)\cdot \left(\frac{1}{r^{2}} \int_{B(x,r)} \frac{|\grad u_2(y)|^{2}}{|y-x|^{n-1}}dy\right)
\end{equation}
and denote 
$$K_r = \max_{i=1,2} \,\left( 
{\ds \int_{B(x,r)}\frac{u_i(y)^2}{|y-x|^{n+1}}\,dy}\right)^{1/2} \left({\ds\int_{B(x,r)} \frac{|\grad u_i(y)|^{2}}{|y-x|^{n-1}}dy}\right)^{-1/2}.$$
Then $J(x,r)$ is an absolutely continuous function of $r\in (0,R)$ and
\begin{equation}\label{eqww2}
\frac{J'(x,r)}{J(x,r)}\geq - c\,\frac{(1+K_r)\,w(x,r)}{r}, \quad\mbox{ for a.e. $0<r<R$,}
\end{equation}
with $c$ depending $n$.
If, in addition, there exists a modulus of continuity $w_0: [0, \infty] \to [0,\infty]$ such that
\begin{equation}\label{eq:Dini cond}
C_{w_0}:=\int_0^{1/2}  \left( w_0(t) \,\log\frac{1}{t} \right)^2\,\frac{dt}{t} < \infty
\end{equation}
and
\begin{equation}\label{eqACF1}
u_i(y)\leq C_1 w_0\Big(\frac{|y-x|}r\Big)\,\|u\|_{\infty,B(x,r)},\quad \mbox{$i=1,2$,}
\end{equation}
for all $0<r\leq R$ and $y\in B(x,r)$, then 
\begin{equation}\label{eqACF0b}
K_r \lesssim 1+C_1\,C_{w_0}, \quad \mbox{ for all $0<r\leq R$},
\end{equation}
with the implicit constant depending on $n$.
\end{theorem}

The proof of Theorem \ref{teoACF-elliptic} is given in the Appendix \ref{secappendix}.
\begin{rem}
Notice that the condition \eqref{eq:Dini cond} is satisfied when we assume $w_0(t)=t^\alpha$, for some $\alpha \in (0,1)$, i.e., when $u$ is H\"older continuous up to the boundary.
\end{rem}}
\begin{rem}\label{rem:acf-dini}\label{rem100}
 Note that, under the assumptions of Theorem \ref{teoACF-elliptic}, if the condition  \rf{eqACF1} is fulfilled for all $r\in(r_1,r_2)$, 
from \rf{eqww2} and \rf{eqACF0b} we have
$$J(x,r_1)\leq J(x,r_2)\,e^{c\int_{r_1}^{r_2}\frac{w(x,r)}r\,dr}.$$
\end{rem}

\vvv

\noi{\bf Acknowledgements.} 
We would like to thank 
S. Hofmann and J. M. Martell for their useful comments on some arguments in the proof of Lemma \ref{lem:packing-case1}, and
S. Hofmann for letting us know about an updated version of \cite{HMM2}. Also,
we are grateful to S. Kim for some helpful discussions pertaining his work on Green functions and elliptic regularity theory.
Finally, we thank D. De Silva for providing us with the reference of the work by Aguilera, Caffarelli and Spruck \cite{ACS}.



\vv



\section{Preliminaries}

We will write $a\lesssim b$ if there is $C>0$ so that $a\leq Cb$ and $a\lesssim_{t} b$ if the constant $C$ depends on the parameter $t$. We write $a\approx b$ to mean $a\lesssim b\lesssim a$ and define $a\approx_{t}b$ similarly. Sometimes we will also use the notation $\avint_{F}$ for the average $|F|^{-1}\int_{F}$ over a set $F \subset \R^{n+1}$ with respect to the $(n+1)$-Lebesgue measure.
 
In the whole paper, $\Omega$ will be an open set in $\R^{n+1}$, with $n\geq2$.\vv

\subsection{The dyadic lattice $\DD_\mu$}\label{subsec:dyadic}

Given an $n$-AD-regular measure $\mu$ in $\R^{n+1}$ we consider 
the dyadic lattice of ``cubes'' built by David and Semmes in \cite[Chapter 3 of Part I]{DS2}. The properties satisfied by $\DD_\mu$ are the following. 
Assume first, for simplicity, that $\diam(\supp\mu)=\infty$). Then for each $j\in\Z$ there exists a family $\DD_{\mu,j}$ of Borel subsets of $\supp\mu$ (the dyadic cubes of the $j$-th generation) such that:
\begin{itemize}
\item[$(a)$] each $\DD_{\mu,j}$ is a partition of $\supp\mu$, i.e.\ $\supp\mu=\bigcup_{Q\in \DD_{\mu,j}} Q$ and $Q\cap Q'=\varnothing$ whenever $Q,Q'\in\DD_{\mu,j}$ and
$Q\neq Q'$;
\item[$(b)$] if $Q\in\DD_{\mu,j}$ and $Q'\in\DD_{\mu,k}$ with $k\leq j$, then either $Q\subset Q'$ or $Q\cap Q'=\varnothing$;
\item[$(c)$] for all $j\in\Z$ and $Q\in\DD_{\mu,j}$, we have $2^{-j}\lesssim\diam(Q)\leq2^{-j}$ and $\mu(Q)\approx 2^{-jn}$;
\item[$(d)$] there exists $C>0$ such that, for all $j\in\Z$, $Q\in\DD_{\mu,j}$, and $0<\tau<1$,
\begin{equation}\label{small boundary condition}
\begin{split}
\mu\big(\{x\in Q:\, &\dist(x,\supp\mu\setminus Q)\leq\tau2^{-j}\}\big)\\&+\mu\big(\{x\in \supp\mu\setminus Q:\, \dist(x,Q)\leq\tau2^{-j}\}\big)\leq C\tau^{1/C}2^{-jn}.
\end{split}
\end{equation}
This property is usually called the {\em small boundaries condition}.
From (\ref{small boundary condition}), it follows that there is a point $z_Q\in Q$ (the center of $Q$) such that $\dist(z_Q,\supp\mu\setminus Q)\gtrsim 2^{-j}$ (see \cite[Lemma 3.5 of Part I]{DS2}).
\end{itemize}
We set $\DD_\mu:=\bigcup_{j\in\Z}\DD_{\mu,j}$. 

In case that $\diam(\supp\mu)<\infty$, the families $\DD_{\mu,j}$ are only defined for $j\geq j_0$, with
$2^{-j_0}\approx \diam(\supp\mu)$, and the same properties above hold for $\DD_\mu:=\bigcup_{j\geq j_0}\DD_{\mu,j}$.

Given a cube $Q\in\DD_{\mu,j}$, we say that its side length is $2^{-j}$, and we denote it by $\ell(Q)$. Notice that $\diam(Q)\leq\ell(Q)$. 
We also denote 
\begin{equation}\label{defbq}
B_Q:=B(z_Q,c_1\ell(Q)),
\end{equation}
where $c_1>0$ is some fix constant so that $B_Q\cap\supp\mu\subset Q$, for all $Q\in\DD_\mu$.

For $\lambda>1$, we write
$$\lambda Q = \bigl\{x\in \supp\mu:\, \dist(x,Q)\leq (\lambda-1)\,\ell(Q)\bigr\}.$$

We denote $\delta_\Omega(x)=\dist(x,\partial\Omega)$.

\vv

\subsection{Sobolev spaces}

For an open set $\Omega \subset \R^{n+1}$,  we define $W^{1,2}(\Omega)$ to be the space of all weakly differentiable functions $u \in L^{2} (\Omega)$, whose weak derivatives belong to $L^2(\Omega)$. We also define $Y^{1,2}(\Omega)$ to be the space of all weakly differentiable functions $u \in L^{2^*} (\Omega)$, where $2^*= \frac{2(n+1)}{n-1}$, whose weak derivatives belong to $L^2(\om)$. We endow those spaces with the norms
\begin{align*}
\| u\|_{W^{1,2}(\om)}&=\| u\|_{L^{2}(\om)} +\| \nabla u\|_{L^{2}(\om)}\\
\| u\|_{Y^{1,2}(\om)} &=\| u\|_{L^{2^*}(\om)} +\| \nabla u\|_{L^{2}(\om)}.
\end{align*} 
Let $W^{1,2}_0(\om)$ and $Y^{1,2}_0(\om)$ be the closure of $C^\infty_0(\om)$ in $W^{1,2}(\om)$ and $Y^{1,2}(\om)$ respectively and note that by Sobolev's inequality, 
$$\| u \|_{L^{2^*}(\om)} \leq C \| \nabla u \|_{L^{2}(\om)},$$
for all $u \in Y_0^{1,2}(\om)$,  thus, $W_0^{1,2}(\om) \subset Y_0^{1,2}(\om)$. If the $(n+1)$-dimensional Lebesgue measure of $\om$ is finite then $W_0^{1,2}(\om) = Y_0^{1,2}(\om)$. The aforementioned Sobolev spaces are in fact Hilbert spaces with inner product 
$$\langle u, v \rangle:=\int_\om \nabla u \cdot \nabla v.$$

\vv

\subsection{A regularity result}

For an open set $\Omega \subset \R^{n+1}$, we denote 
$$ \| f \|_{C^\alpha(\Omega)}:= \sup_{\Omega}| f | + \sup_{x\neq y: x, y \in \Omega} \frac{| f (x)- f(y)|}{|x-y|^\alpha} $$
and we define the space of inhomogeneous $\alpha$-H\"older functions by 
$$C^{\alpha}(\Omega):=\{f:\Omega \to \R: \|f \|_{C^\alpha(\Omega)}<\infty\}.$$

\begin{lemma}\label{lem:grad-u-regul}
Let us assume that  $u\in W^{1,2}(B(0,1))$ is a weak solution of $Lu=\divv F$. If $a_{ij} \in C^\alpha(B(0,1))$, for $1\leq i,j\leq n+1$, and $F \in C^\alpha(B(0,1);\R^{n+1})$  then for any $\delta \in (0,1)$, it holds that
$$
\| \nabla u \|_{C^\alpha(B(0, 1-\delta))}  \lesssim_{n, \delta} \| \nabla u\|_{L^2(B(0,1))} +  \| F \|_{C^\alpha(B(0,1))},
$$
with the implicit constant depending also on the $\alpha$-H\"older norm of the coefficients $a_{ij}$.
\end{lemma}

\begin{proof}
The lemma follows from Theorem 5.19 in \cite{GiaMa}.
\end{proof}

\vv
\subsection{About elliptic measure}

Let $A=(a_{ij}(x))_{1\leq j \leq n+1}$ be a matrix that satisfies \eqref{eqelliptic1} and \eqref{eqelliptic2} without any additional assumption (e.g. locally Lipschitz or symmetric).
We consider the second order elliptic operator $L = -\divv A \nabla$ and we say that a function $u \in W^{1,2}_{loc}(\om)$ is a {\it weak solution} of the equation $L u=0$ in $\Omega$ (or just {\it $L$-harmonic}) if 
\begin{equation}\label{eq:solution}
\int A \nabla u \nabla \Phi=0, \;\; \mbox{ 
for all $\Phi \in C^\infty_0(\om)$}.
\end{equation} 
We also say that  $u \in W^{1,2}_{loc}(\om)$ is a {\it supersolution} (resp. \textit{subsolution}) for $ L$ in  $\Omega$ or just {\it $L$-superharmonic} (resp. {\it $L$-subharmonic}) if $\int A \nabla u \nabla \Phi \geq 0$   (resp. $\int A \nabla u \nabla \Phi \leq 0$) for all non-negative $\Phi \in C^\infty_0(\om)$.

\vv
Following  \cite[Section 9]{HKM}, from now on we make the convention that if $\om$ is unbounded, then the point at infinity always belongs to its boundary. So, all the topological notions are understood with respect to the compactified space $\overline{\R}^{n+1}=\R^{n+1} \cup \{\infty\}$. Moreover, the functions $f \in C(E)$, for $E \subset \overline{\R}^{n+1}$ are assumed to be continuous and real-valued. Therefore, all functions in $C(\d \om)$ are bounded even if $\Omega$ is unbounded. 

\vv
If  $ \Phi \in W^{1,2}(\Omega)\cap C(\overline \Omega)$ with $\phi=\Phi|_{\d \om}$, then one can construct a unique {\it variational solution} for the $L$-Dirichlet problem with data $\phi$. Indeed, by Lax-Milgram theorem  in $W_0^{1,2}(\Omega)$, there exists a unique $v \in W_0^{1,2}(\om)$ such that $\int_\om A \nabla v \cdot \nabla \Psi = - \int_\om A \nabla \Phi \cdot \nabla \Psi$, for every $\Psi \in W_0^{1,2}(\Omega)$. Therefore, if we set $u= v+\Phi$, it is clear that $u \in  W^{1,2}(\Omega)$, $Lu=0$, and $u|_{\d \om}=\phi$ in the Sobolev sense. To prove uniqueness, one should exploit the ellipticity condition as well as the fact that the difference of two variational solutions with the same data is a solution for $L$ which lies in $W^{1,2}_0(\om)$. 

We say that a point $x_0 \in \d \om\setminus \{\infty\}$  is {\it Sobolev $L$-regular} if, for each function $\Phi \in W^{1,2}(\Omega)\cap C(\overline \Omega)$, the $L$-harmonic function $h$ in $\om$ with $h-\Phi \in W^{1,2}_0(\om)$ satisfies
$$\lim_{x \to x_0} h(x)=\Phi(x_0) .$$

\vv
 \begin{theorem}[Theorem 6.27 in \cite{HKM}] \label{teoreg1}If for $x_0 \in \partial \Omega\setminus \{\infty\}$ it holds that
  $$\int_0^1 \frac{\textup{cap}(B(x_0,r) \cap \Omega^c, B(x_0,2r))}{\textup{cap}(B(x_0,r),B(x_0,2r))}\, \frac{dr}{r}=+\infty,$$
 then $x_0$ is Sobolev $L$-regular. Here ${\rm cap}(\cdot, \cdot)$ stands for the variational $2$--capacity of the condenser $(\cdot, \cdot)$ (see e.g. \cite[p. 27]{HKM}).
 \end{theorem}

We say that a point $x_0 \in \d \om$  is {\it Wiener regular} if, for each function $f \in C(\d \om; \R)$, the $L$-harmonic function $H_f$ constructed by the Perron's method satisfies 
$$\lim_{x \to x_0} H_f(x)=f(x_0) .$$
See \cite[Chapter 9]{HKM}. 
 
 \vv
 \begin{lemma}[Theorem 9.20 in \cite{HKM}]
 Suppose that $x_0 \in \d \om \setminus \{\infty\}$. If $x_0$ is Sobolev $L$-regular then it is also  Wiener regular.
 \end{lemma}
 
Note that some of the aforementioned results form \cite{HKM} are only stated for $\Omega$ bounded. Although, a careful inspection of their proofs shows that, with our above construction of variational solutions,  they extend to the case that $\Omega$ is unbounded. Moreover,  $\infty$ is a Wiener regular point for each unbounded $\Omega \subset \R^{n+1}$, if $n \geq 2$ (see Theorem 9.22 in \cite{HKM}).
 
We say that $\Omega$ is Sobolev $L$-regular (resp.\ Wiener regular) if all the points in $\partial\Omega \setminus \{\infty\}$ are Sobolev $L$-regular 
(resp.\ Wiener regular).

As a consequence of Theorem \ref{teoreg1} and the preceding lemma, if $\partial\Omega$ is $n$-AD-regular, then it is also
Sobolev $L$-regular and Wiener regular.

Let $\Omega \subset \R^{n+1}$ be Wiener regular and $x\in\Omega$. If $f \in C(\partial \Omega)$, then the map $f \mapsto \overline{H}_f(x)$ is a bounded linear functional on $C(\partial \Omega)$. Therefore, by Riesz representation theorem and the maximum principle, there exists a probability measure $\hm^{x}$ on $\partial \Omega$ (associated to $L$ and the point $x \in \Omega$) defined on Borel subsets of $\d\Omega$ so that
$$ \overline{H}_f(x) =\int_{\partial \Omega} f \, d\hm^{x}, \;\; \mbox{for all $x \in \Omega$.}$$ 
We call $\omega^x$ the elliptic measure associated to $L$ and $x$.\vv

\vv
{
The following lemma is  standard and its proof is omitted. 
\begin{lemma}\label{lem:Pullback}
Let $\Omega \subset \R^{n+1}$ be an open set, and assume that $A$ is a uniformly elliptic matrix with real entries and $\Phi:\R^{n+1} \to \R^{n+1}$ is a bi-Lipschitz map. If we set
$$\tilde A:=  |\det D_{\Phi}|\, D_{\Phi^{-1}} (A\!\circ\!\Phi) D^T_{\Phi^{-1}},$$ 
then $\tilde A$ is a uniformly elliptic matrix and $u$ is a weak solution of $L_{A} u =0$ in $\Phi(\Omega)$ if and only if $\tilde{u}=u\circ \Phi$ is a weak solution of $L_{\tilde A} \tilde u =0$ in $\Omega$.

Assuming $\Omega$ is a Wiener regular domain, we have that for any set $E \subset \Phi(\d \Omega)=\d \Phi(\Omega)$ and $x\in \Omega$,
\begin{equation}\label{e:pushf}
 \omega_{\Phi(\Omega)}^{L_{A},\Phi(x)}(E)=\omega_{{\Omega}}^{L_{\widetilde A},x}(\Phi^{-1}(E)).
\end{equation}
\end{lemma}

As a corollary of this result we have the following.
\begin{corollary}\label{cor:A(x0)=id}
Let $\Omega \subset \R^{n+1}$ be an open set, and assume that $A$ is a uniformly elliptic matrix with real entries.  Let $A_s= (A + A^*)/2$  be the symmetric part of $A$ and for a fixed point $y_0 \in \om$ define $S= \sqrt{A_s(y_0)}$. If  
\[
\tilde{A}(\cdot) = S^{-1} (A\circ S)(\cdot) S^{-1},
\]
then $\tilde{A}$ is uniformly elliptic,  $\tilde A_s(z_0) = Id$ if $z_0 = S^{-1}y_0$ and $u$ is a weak solution of $L_{A} u =0$ in $\Omega$ if and only if $\tilde{u}=u\circ S$ is a weak solution of $L_{\tilde A} \tilde u =0$ in $S^{-1}(\Omega)$ . 

Assuming $\Omega$ is a Wiener regular domain, we have that for any set $E \subset \d \Omega$ and $x\in \Omega$,
\begin{equation}\label{e:pushf}
 \omega_{\Omega}^{L_{A},x}(E)=\omega_{S^{-1}(\Omega)}^{L_{\widetilde A},S^{-1}x}(S^{-1}(E)).
\end{equation}
\end{corollary}

\vv

Here we used that $A_s$ is a symmetric and uniformly elliptic matrix and thus $A_s(y_0)$ has a unique square root $S$ which is also symmetric and uniformly elliptic with real entries. 

Let us recall some simple facts from linear algebra which help us understand how the geometry of $\Omega$ is affected by the linear transformation above. Note that $S$ is orthogonally diagonalizable since it is symmetric, which means that it represents a linear transformation with scaling in mutually perpendicular directions. Hence $S^{-1}$ is a special bi-Lipschitz change of variables that takes balls to ellipsoids, where eigenvectors determine directions of semi-axes, eigenvalues determine lengths of semi-axes and its maximum eccentricity is given by $\sqrt{(\lambda_{\max} / {\lambda_{\min}})}$ (where $\lambda_{\max}$  are $\lambda_{\min}$ 
are the maximal and minimal eigenvalues of $S^{-1}$),
which is in turn bounded below by $\sqrt{\Lambda}^{-1}$ and above by $\sqrt{\Lambda}$.  

In particular, $S^{-1}(\d \Omega)=\d (S^{-1}(\Omega))$, $\Lambda^{-1/2} \leq \|S^{-1} \|\leq   \Lambda^{1/2}$, i.e., $S^{-1}$ distorts distances by at most a constant depending on ellipticity. It is not hard to see that the collection $\wt \DD_{\mu}:= \{ S^{-1}(Q) \}_{Q \in \DD_\mu}$  forms a dyadic grid on $S^{-1}(\d\om)$ 
as described in Subsection \ref{subsec:dyadic}, where the involved constants depend on the ones in $\DD_\mu$ and ellipticity. Finally, there is a one-to-one correspondence  between the connected components of $\{x \in \om:  u > \alpha\}$ and the connected components of  $\{x \in S^{-1}(\om): \wt u > \alpha\}$ via $S^{-1}$. 
 }

\vvv

\begin{lemma}\label{lemgreen*}
Let $\Omega\subset\R^{n+1}$ be an open, connected set so that $\partial\Omega$ is Sobolev $L$-regular. 
There exists a Green function $G:\Omega\times \Omega\setminus\{(x,y):x=y\}\to \R$ associated with $L$ which
satisfies the following.
For $0<a<1$, there are are positive constants $C$ and $c$ depending on $a$, $n$ and $\Lambda$ such that for all $x,y\in\Omega$ with $x\neq y$, it holds:
$$0\leq G(x,y)\leq C\,|x-y|^{1-n}$$

\vspace{-4mm}
$$G(x,y)\geq c\,|x-y|^{1-n} \quad \mbox{ if $|x-y|\leq a\,\delta_\Omega(x)$,}$$

\vspace{-3mm}
$$G(x,\cdot)\in C(\overline\Omega \setminus \{x\}) \cap W^{1,2}_{loc}(\Omega \setminus \{x\}) 
\quad \mbox{ and } \quad G(x,\cdot)|_{\partial\Omega} \equiv 0,$$

\vspace{-3mm}
$$G(x,y) = G^*(y,x),$$
where $G^*$ is the Green function associated with the operator $L^* = -\divv A^*\nabla$, and for every $\vphi\in C_c^\infty(\R^{n+1})$,
\begin{equation}\label{eqgreen*23}
\int_{\partial\Omega} \vphi\,d\omega^x - \vphi(x) = - \int_\Omega A^*(y)\nabla_yG(x,y)\cdot \nabla\vphi(y)\,dy,
\quad\mbox{ for a.e. $x\in\Omega$.}
\end{equation}
\end{lemma}

In the statement in \rf{eqgreen*23}, one should understand that the integral on right hand side is absolutely convergent for a.e.\ $x\in\Omega$.

\begin{proof}
{ 
In \cite{HK}, Hofmann and Kim showed that
there exists a  function $G(\cdot, \cdot)$, continuous on $\om \times \om \setminus \{(x,y) \in \om \times \om: x=y\}$ and in $W^{1,2}_{loc}(\Omega \setminus \{x\}) $, so that $G(x, \cdot)$ is locally integrable for every $x\in \om$ with the following properties (among others):
\begin{enumerate}
\item For any $\eta \in C^\infty_0(\om)$ with $\eta \equiv 1$ in $B(y,r)$, for $r<\dist(y, \d \om)$,  it holds that $(1-\eta) G(\cdot,y) \in Y^{1,2}_0(\om)$.
\item For any $\Psi \in C^\infty_0(\om)$.
\begin{align}
\int_{\Omega} A(y) \nabla_y G(y,x) \nabla \Psi(y)\, dy =\Psi(x).\label{eq:Green-dirac}
\end{align}
\item If  $p\in [1,\frac{ n+1}{n-1})$ then 
\begin{equation}\label{eqgpo0}
\|G(x,\cdot)\|_{L_p(B(x,r))} + r\,\|\nabla_yG(x,\cdot)\|_{L_p(B(x,r))} \leq C\,r^{2-n+n/p},
\end{equation}
for all $0<r<\delta_\Omega(x)$. The same estimates hold for $G(\cdot,x)$.

\item $G(x,y)=G^*(y,x)$, where $G^*$ stands for the Green function associated with $L^*=- \divv A^*\nabla$ and $A^*$ for the transpose matrix of $A$ (which is also uniformly elliptic with the same ellipticity constants).
\end{enumerate}
Also $G$ is unique in the class of functions for which the potential 
$$u(x)=\int_\Omega G(y,x)\,f(y)\,dy,\,\,\textup{for}\,\,f \in L^\infty_c(\Omega)$$
belongs to $Y_0^{1,2}(\Omega)$ and satisfies $ L^*u=f$.}

Kang and Kim
further proved in \cite{KK} that 
\begin{equation}\label{eq:Y12Green}
\|G(\cdot, y)\|_{Y^{1,2}(\om \setminus B(y,r))} \lesssim r^{-\frac{n-1}{2}}, \quad \mbox{ for all $y\in\Omega$,}
\end{equation}
and
$$|G(x,y)|\leq C\,|x-y|^{1-n},\quad \mbox{ for all $x,y\in\Omega$, $x\neq y$.}$$
See Theorem 3.6 and Corollary 4.1 in \cite{KK}. { Moreover, we can show that $G \geq 0$ and  
$$G(x,y)\geq c\,|x-y|^{1-n}, \quad \mbox{ if $|x-y|\leq a\,\delta_\Omega(x)$,}$$ 
if we argue as Gr\"uter and Widman did in \cite[Theorem 1.1]{GW} but with the function $G$ above (in fact, the Green function of Gr\"uter and Widman coincides with the one of Hofmann and Kim in bounded domains). }

So it just remains to show the identity \rf{eqgreen*23}. Although this is rather standard, we will show the details.
If $u$ is the variational solution with data $\vphi|_{\d \om}$, then $u \in C(\overline \om)\cap W^{1,2}(\Omega)$ and $u - \vphi \in W_0^{1,2}(\om)$. By using the estimates above for the Green function and Fubini, one can show that the integral
\begin{equation}\label{eqfubini*33}
\int_\om A^*(y) \nabla_y G(x,y) \nabla u(y)\, dy
\end{equation}
is absolutely convergent for a.e.\ $x\in\Omega$. Indeed, given any ball $B=B(x_0,r)$ such that $10B\subset \Omega$, we have
\begin{align}\label{eqtrw17}
\int_B\int_\om |A^*(y) \nabla_y G(x,y) \nabla u(y)|\, dy\,dx & \lesssim
\int_{B}\int_{|x-y|\leq 2 r(B)}|\nabla_y G(x,y)|\, | \nabla u(y)|\,dy\,dx \\
&\quad +   \int_{B}\int_{|x-y|> 2 r(B)}|\nabla_y G(x,y)|\, | \nabla u(y)|\,dy\,dx.\nonumber
\end{align}
By \rf{eqgpo0} and Tonelli's theorem, the first integral on the right hand side is bounded by
$$\int_{y\in 3B}\int_{x\in B}|\nabla_y G(x,y) |\,dx\,|\nabla u(y)|dy \lesssim c(r)\,\|\nabla u\|_{L^1(B)} \leq c'(r) 
\|\nabla u\|_{L^2(B)},$$
while for the second one, by Cauchy-Schwarz and \rf{eq:Y12Green} we have
$$\int_{|x-y|> 2 r(B)}|\nabla_y G(x,y)|\, |  \nabla u(y)|\,dy \leq \|G(x,\cdot)\|_{Y^{1,2}(\Omega\setminus 2B)}\|\nabla u\|_{L^2(\Omega)}\leq c(r)\,\|\nabla u\|_{L^2(\Omega)}.
$$
So the last integral in \rf{eqtrw17} is also finite, which shows that indeed the integral in \rf{eqfubini*33} is absolutely convergent for a.e.\ $x\in B$, and thus for a.e.\ $x\in\Omega$.

We claim now that for a.e.\ $x\in\Omega$, we have
\begin{equation}\label{eqclaim388}
(u-\vphi)(x)= \int_\om A^*(y) \nabla_y G(x,y) \nabla(u- \vphi)(y)\, dy.
\end{equation}
To show this, since $u-\vphi\in W^{1,2}_0(\Omega)$, there exists a sequence of functions $\psi_k\in C_0^\infty(\Omega)$ which converge to $u-\vphi$ in $W^{1,2}(\Omega)$.
Then note that the same estimates we used to show that the integral in \rf{eqfubini*33} is convergent, when applied to 
$u-\vphi-\psi_k$ give that
$$\int_B\int_\om |A^*(y) \nabla_y G(x,y) \nabla (u  -\vphi - \psi_k)
(y)|\, dy\,dx\leq c(r)\,\|u  -\vphi - \psi_k\|_{W^{1,2}(\Omega)} \to 0$$
as $k\to\infty$.
By applying \rf{eq:Green-dirac} to $\psi_k$, we deduce that
$$\int_B \Bigl|\psi_k(x) - \int_\om A^*(y) \nabla_y G(x,y) \nabla( u  -\vphi)(y)\, dy\Bigr|\,dx\to 0
\quad\mbox{ as $k\to\infty$.}$$
Since $\psi_k$ converges to $u-\vphi$ in $L^1(B)$, we get
$$\int_B \Bigl|(u-\vphi)(x) - \int_\om A^*(y) \nabla_y G(x,y) \nabla( u  -\vphi)(y)\, dy\Bigr|\,dx =0,$$
which proves our claim \rf{eqclaim388}.

We will show now that 
$$\int_\om A^*(y) \nabla_y G(x,y) \nabla u(y)\, dy=0$$
for a.e.\ $x\in\Omega$ such that the integral on left hand side is absolutely convergent.
To this end, fix $\ve >0$ small enough, so that $\ve \ll \delta_{\om}(x)$ and let  
$$\eta_\ve(y):= \eta(|x-y| / \ve),$$
where $\eta \in C^\infty_0(\R)$ is a non-negative function so that $\eta \equiv 0$ in $B(0,1)$ and $\eta\equiv 1$ in $\R^{n+1} \setminus B(0,2)$. Therefore, by dominated convergence, it is enough to prove that for a.e. $x \in \om$,
$$\lim_{\ve \to 0}I_\ve := \lim_{\ve \to 0}  \int_\om A^*(y) \nabla_y G(x,y) \nabla u(y) \eta_\ve(y)\, dy=0.$$  
Note that 
\begin{align*}
I_\ve &= \int_\om A^*\nabla [\eta_\ve G(x,\cdot)]\cdot \nabla u  - \int_\om A^*\nabla u \cdot \nabla \eta_\ve \,  G(x,\cdot)\\ 
&=:I^1_\ve-I^2_\ve,
\end{align*}
and $I^1_\ve=0$, since $u \in W^{1,2}(\Omega)$ is a solution for $L u=0$ in $\om$ and $\eta_\ve G(x,\cdot) \in Y_0^{1,2}(\om)$ (recall that $C^\infty_0(\om)$ is dense in $Y_0^{1,2}(\om)$). It remains to show that $I_\ve^2 \to 0$ as $\ve \to 0$. Indeed, by the fact $\nabla \eta_\ve$ is supported in the annulus $B(x,2\ve) \setminus B(x, \ve)$, the bound $\|\nabla \eta_\ve \|_\infty\lesssim \ve^{-1}$, the pointwise bounds of Green function and Cauchy-Schwarz, we have that 
\begin{align*}
|I^2_\ve| &\lesssim \|A\|_{\infty} \ve^{-n} \int_{B(x,2\ve) \cap \om} |\nabla u|   \lesssim \ve\, \mathcal{M}(\nabla u \chi_{\om})(x),
\end{align*}
where $\mathcal{M}$ stands for the Hardy-Littlewood maximal function. Since, 
$$\|\mathcal{M}(\nabla u \chi_{\om})\|_{L^2} \lesssim \|\nabla u\|_{L^2(\om)} \leq \| u\|_{W^{1,2}(\om)} <\infty,$$
then, for a.e. $x\in \om$, $ \mathcal{M}(\nabla u \chi_{\om})(x)<\infty$, and thus, $I_\ve^2 \to 0$ as $\ve \to 0$ for a.e. $x\in \om$. Hence we have shown that
\begin{equation}\label{eqffpp1}
(u-\vphi)(x)= - \int_\om A^*(y) \nabla_y G(x,y) \nabla\vphi(y)\, dy, \quad\mbox{ for a.e. $x\in\Omega$.}
\end{equation}

To prove \rf{eqgreen*23} we will show now that $u(x)= \int_{\d \om} \vphi \,d\omega^x$ for all $x\in\Omega$. First we need to check that $u(x)\to 0$ as $x\to\infty$. To this end, suppose that $\supp\vphi\subset B(0,R)$, and 
without loss of generality assume that  $0\in \d \om$. 
Then if $x \in \R^{n+1} \setminus B(0, 4R)$, in view of \eqref{eqffpp1} and, Cauchy-Schwarz,  Caccioppoli's inequality for $L$-subharmonic functions and the pointwise bounds for Green function, for a.e. $x\in  \overline\Omega\setminus B(0,4R)$ we have that
\begin{align*}
u(x) &\lesssim \|A\|_{L^\infty(\om)}\, \|\nabla \vphi\|_{L^\infty(\om)}\, \int_{B(0,R)} |\nabla_y G(y,x)|\, dy\\
&\lesssim \|\nabla \vphi\|_{L^\infty(\om)}\,R^{n-1}\left(\,\avint_{B(0,2R)} G(y,x)^2\, dy\right)^{1/2}\\
& \lesssim \|\nabla \vphi\|_{L^\infty(\om)}\,\frac{R^{n-1}}{|x|^{n-1}}.
 \end{align*}
Since $u$ is continuous in $\overline\Omega$, the above estimate holds true for every $x\in  \overline\Omega\setminus B(0,4R)$, and thus  $u(x)\to 0$ as $x\to\infty$, as wished.

 Finally, since  $u \in C(\overline\om)$ and for every $\xi \in \d \om \setminus\{\infty\}$ it holds that $\ u(x) \to \vphi(\xi)$ continuously and $u$ vanishes at $\infty$, by the maximum principle, $u=H_{\vphi|_{\d \om}}$. Thus, $u(x)= \int_{\d \om} \vphi \,d\omega^x$ for all $x\in\Omega$.
\end{proof}

\vv

In case $\Omega=\R^{n+1}$, the result above can be rephrased appropriately by replacing the Green function $G(\cdot,\cdot)$ by the
 fundamental solution $\EE_L(\cdot,\cdot)$. This satisfies

\begin{equation}\label{eq:ptws-EE}
 C_\ve^{-1}\,|x-y|^{1-n} \leq \EE_L(x,y)\leq C_\ve\, |x-y|^{1-n}, \quad\textup{for}\,\,x \neq y, 
\end{equation}

\vspace{-3mm}
$$
\EE_L(x,\cdot)\in C(\R^{n+1} \setminus \{x\}) \cap W^{1,2}_{loc}(\R^{n+1} \setminus \{x\}),
$$
\vspace{-3mm}
$$\EE_L(x,y) = \EE_{L^*}(y,x),$$
and
$$ \vphi(x) =  \int A\nabla_y\EE_L(y,x)\cdot \nabla\vphi(y)\,dy.$$
\vv

The following result, sometimes known as ``Bourgain's estimate", also holds. For a proof see e.g. Lemma 11.21 in \cite{HKM}.

\begin{lemma}\label{l:bourgain}
Let $\Omega\subsetneq \bR^{n+1}$ be open with $n$-AD-regular boundary,  $x\in \d\Omega$, and $0<r\leq\diam(\partial\Omega)/2$.  Then 
\begin{equation}\label{e:bourgain}
\hm^y(B(x,2r))\geq c >0, \;\; \mbox{ for all }y\in \Omega\cap B(x,r)
\end{equation}
where $c$ depends on $n$, the ellipticity constant $\Lambda$ and the $n$-AD-regularity constant of $\partial\Omega$.
\end{lemma}
\vv

The next lemma is deduced from the preceding one by standard arguments involving the pointwise bounds for Green function and Lemma \ref{l:bourgain} and maximum principle.

\begin{lemma}\label{l:w>G}
Let $\Omega\subset\R^{n+1}$ be open with $n$-AD-regular boundary.
Let $B=B(x_0,r)$ be a closed ball with $x_0\in\pom$ and $0<r<\diam(\pom)$. Then, 
\begin{equation}
 \omega^{x}(4B)\gtrsim  r^{n-1}\, G(x,y),\quad\mbox{
 for all $x\in \Omega\backslash  2B$ and $y\in B\cap\Omega$,}
 \end{equation}
 with the implicit constant depending on $n$, the ellipticity constant $\Lambda$ and the $n$-AD-regularity constant of $\partial\Omega$.
\end{lemma}

For a proof in the case of $L=\Delta$, see \cite[Lemma 3.3]{AHM3TV}.

\vv

The next result is also standard and follows from Lemma \ref{l:bourgain}. For a proof see e.g. Lemma 2.3 in \cite{AM15}.

\begin{lemma}\label{lemholder**1}
\label{l:holder}
Let $\Omega\subsetneq\R^{n+1}$ be open with $n$-AD-regular boundary and let $x\in \d\Omega$. Then there is $\alpha>0$ so that for all $0<r<\diam(\Omega)$,
\begin{equation}\label{e:wholder}
 \omega^{y}({B}(x,r)^{c})\lesssim \ps{\frac{|x-y|}{r}}^{\alpha},\quad \mbox{ for all } y\in \Omega\cap B(x,r),
 \end{equation}
where $\alpha$ and the implicit constant depend on $n$, the ellipticity constant $\Lambda$ and the $n$-AD-regularity constant of $\partial\Omega$.
\end{lemma}

\vv
From the preceding lemma, the maximum principle, and standard Moser estimates for subsolutions of $L$, one obtains the following auxiliary result which will be necessary below.

\begin{lemma}\label{lem333}
Let $\Omega\subsetneq\R^{n+1}$ be open with $n$-AD-regular boundary. Let $x\in\partial\Omega$ and $0<r<\diam(\Omega)$.
Let $u$ be a non-negative $L$-harmonic function in $B(x,4r)\cap \Omega$ and continuous in $B(x,4r)\cap \overline\Omega$
so that $u\equiv 0$ in $\partial\Omega\cap B(x,4r)$. Then extending $u$ by $0$ in $B(x,4r)\setminus \overline\Omega$,
there exists a constant $\alpha>0$ such that
$$u(y)\leq C\,\left(\frac{\delta_\Omega(y)}r\right)^\alpha \!\sup_{B(x,2r)}u
\leq C\,\left(\frac{\delta_\Omega(y)}r\right)^\alpha \;\avint_{B(x,4r)}u,
\quad \mbox{for all $y\in B(x,r)$,}$$
where $C$ and $\alpha$ depend on $n$, the ellipticity constant $\Lambda$ and the AD-regularity of $\partial \Omega$.
In particular, $u$ is $\alpha$-H\"older continuous in $B(x,r)$.
\end{lemma}
{

We shall now prove a lemma concerning the rate of decay at infinity of a bounded $L$-harmonic function vanishing outside a ball centered at the boundary. In the case that the elliptic matrix is symmetric, one may use a generalised Kelvin transform (see Section 3 in \cite{SW}) and argue as in Lemma 6.1 \cite{GMT}. In the non-symmetric case though this method does not work and we follow an alternative path. In particular, we follow mutatis mutandi the steps in Lemma 4.9 in \cite{HKMP} until the statement of the Main Claim. The main obstacle in following  the proof of the corresponding claim in \cite{HKMP} is that our domain does not satisfy the Harnack chain condition. Thus, we are forced to give a different and a bit more complicated argument.

\begin{lemma}\label{lem:u-decay-infinty}
Let $\Omega\subset\R^{n+1}$, $n\geq 2$, be a domain with $n$-AD-regular boundary.
Let $u$ be a bounded,  $L$-harmonic function in $\Omega$, vanishing at $\infty$, and let $B$ be a ball centered at
$\partial\Omega$. Suppose that $u$ vanishes continuously in $\partial\Omega\setminus B$. Then, there is a constant
$\alpha>0$ such that
\begin{equation}\label{eqdjl125}
|u(x)| \lesssim \frac{r(B)^{n-1+\alpha}}{\bigl(r(B) + \dist(x,B)\bigr)^{n-1+\alpha}}\,\|u\|_{L^\infty(\Omega \cap (3B \setminus 2B))}.
\end{equation}
Both $\alpha$ and the constant implicit in the above estimate depend only on $n$, the AD-regularity constant of $\partial \Omega$ and the ellipticity constant  $\Lambda$.
\end{lemma}

\begin{proof}
We begin with a few reductions. After possible rotation, translation and dilation we may assume that $B$ is the unit ball $B(0,1)$ (since our operators are invariant under such transformations). We may further renormalize $u$ so that $u \leq 1$ in $\Omega \cap (3B \setminus 2B)$ and assume without loss of generality that $u \geq 0$. Indeed, let $\Omega_1= \Omega \setminus 2B$ and $f=u|_{\d \Omega_1}$. We may write $f = f^+ - f^-$, where $f^+$ and $f^-$ is the positive and  negative part of $f$ respectively and notice that
$$\max(f^+, f^-) = |f| \leq  \|u\|_{L^\infty(\Omega \cap (3B \setminus 2B))} \leq 1,$$ 
since $f$ vanishes on $\d \Omega_1\setminus \overline {2B}$. Then, we construct $L$-harmonic functions $u^+$ and $u^-$ continuous up to the boundary of $\d \Omega_1$ with $u^\pm|_{\d \Omega_1}= f^\pm$ that vanish at infinity. Therefore, by maximum principle, $u = u^+- u ^-$ and
$$\|u^\pm\|_{L^\infty(\Omega_1)}\leq 1.$$
Thus, we may treat each $u^\pm$ separately and assume without loss of generality that $u \geq 0$ with $\|u\|_{L^\infty(\Omega_1)}\leq 1$.

Now, let $\EE_0(\cdot):=\EE(0,\cdot)$ be the fundamental solution for $L$ with pole at $0$ and since $\EE_0(x) \approx |x|^{1-n}$ we may choose a constant $
\kappa_0>1$ so that $u(x) \leq v(x):= \kappa_0\, \EE_0(x)$ for any $x \in \Omega \cap (3B \setminus 2B)$. Since both $u$ and $v$ vanish at infinity, by maximum principle we have that $u(x) \leq v(x)$ for any $x \in \Omega \setminus 2B$. 

\vv

It is rather easy to see that our lemma follows from the following claim and an iteration argument and the details are left to the interested reader. \vv

\noi{\bf Main Claim:} If $u \leq v$ in $\Omega \setminus 2^kB$, for some $k \geq 1$, then there exists $\delta \in (0,1)$ such that $u \leq (1-\delta) v$ in $\Omega \setminus 2^{k+1}B$, where $\delta$ depends only on the dimension, the AD-regularity and ellipticity constants.

\vv


Indeed, let us denote for simplicity $B_\ell:= 2^{\ell}B$ for $\ell \geq 1$. If $\d B_{k+1}\cap \Omega=\varnothing$, then $u\equiv 0$ in $\Omega\backslash B_{k+1}$ and we are done, so assume $\d B_{k+1}\cap \Omega\neq \varnothing$. Define the domain $ \Omega_k:=\Omega  \setminus \overline B_k$ and notice that its boundary is AD-regular, and thus, Wiener regular. Since $u$ is continuous in $\Omega \setminus 2B$, for fixed  $x \in \d B_{k+1}\cap \Omega$, we can write
$$ 
u(x) = \int_{\d\Omega_k} u(y) \, d \omega^x_{ \Omega_k}= \int_{\d B_{k}}  u(y) \, d \omega^x_{\Omega_k} \leq \int_{\d B_{k}}  v(y) \, d \omega^x_{ \Omega_k},
$$
where the second equality follows from the fact that $u$ vanishes on $\d \Omega \setminus 2B$ and the last one from the assumption that $u \leq v$ in $\Omega \setminus B_k$. Therefore,  since $v$ is $L$-harmonic in $\Omega_k$, we have that
\begin{align*}
u(x) & \leq  \int_{\d \Omega_k}  v(y) \, d \omega^x_{\Omega_k} - \int_{\d \Omega \setminus  B_k}  v(y) \, d \omega^x_{\Omega_k} = v(x) - \int_{\d \Omega \setminus  B_k}  v(y) \, d \omega^x_{\Omega_k}.
\end{align*}

It suffices to show that there exists $\delta \in (0,1)$ such that
\begin{equation}\label{eq:udecay-mainest}
 \int_{\d \Omega \setminus  B_k}  v(y) \, d \omega^x_{\Omega_k} \geq \delta\, v(x).
 \end{equation}
 
By the AD-regularity of $\mu$ (with constant $C_0$) and the fact that $0\in\supp\mu$ (which can be assumed for free), we know that
$\mu(C_0^2B_{k+1}\setminus B_{k})\geq (2^n-1) \, C_0 \,2^{kn}$.  Consider a grid of closed cubes in $\R^{n+1}$ with diameter $\frac1{100}2^{-k}$.
Denote by $I$ the subfamily of cubes from this grid that intersect the closure of $C_0^2B_{k+1}\setminus B_{k}$, so that
$$\overline{C_0^2B_{k+1}\setminus B_{k}}\subset \bigcup_{Q\in I} Q.$$
Denote by $I_0$ the subfamily of the cubes $Q\in I$ such that $3Q\cap \partial\Omega\neq \varnothing$. We know that $I_0\neq\varnothing$ because  $\mu(C_0^2B_{k+1}\setminus B_{k+1})>0$.
By \eqref{e:bourgain} we obtain
\begin{equation}\label{eqclaim481}
\omega_{\Omega_k}^x(10Q\cap \partial\Omega) =  \omega_{\Omega_k}^x(10Q\cap \partial\Omega_k) \geq c >0,\quad
\mbox{ for all $x\in 3Q\cap\Omega_k$, $Q\in I_0$,}
\end{equation}
where we took also into account that $10Q\cap \partial B_k=\varnothing$.

We claim that  
\begin{equation}\label{eqclaim482}
\omega_{\Omega_k}^x(10Q\cap \partial\Omega) \geq c'>0\quad
\mbox{ for all $x\in Q\cap\Omega_k$, $Q\in I$,}
\end{equation}
which in particular shows that $\omega_{\Omega_k}^x(10Q\cap\partial\Omega)\geq\delta$ (with $\delta=c'$) for all 
$x\in\Omega_k\cap \partial B_{k+1}$, as wished.
The  estimate \rf{eqclaim482} is obvious if $Q\in I_0$, by \rf{eqclaim481}. In case $Q\in I\setminus I_0$, there is a chain of cubes
$$Q=Q_0,Q_1,Q_2,\ldots Q_{m-1},Q_m$$
so that $Q_i$ and $Q_{i+1}$ are neighbors (i.e., they intersect but they have disjoint interiors), $Q_i\in I\setminus I_0$ for 
$0\leq i\leq m-1$, and $Q_m\in I_0$, with $m\leq\# I$ (so $m$ is bounded  by some constant depending only on $C_0$). Since $Q_{m-1}\subset 3Q_m$, in light of \rf{eqclaim481}, it is clear that \rf{eqclaim482} also holds for $x\in Q_{m-1}\cap \Omega_k$.  Now the cubes $Q_0,Q_1,Q_2,\ldots Q_{m-1}$ form a Harnack chain 
in $\Omega_k$ (because $3Q_i\cap
\partial\Omega_k=\varnothing$ for $0\leq i\leq m-1$), and so deduce $\omega_{\Omega_k}^x(10Q\cap \partial\Omega) \gtrsim 1$ for all $x\in Q_0$.

Since $ x \in \d B_{k+1}$, then for every $y \in C_0^2B_{k+1}\setminus B_{k}$, we have that $v(y) \leq (C_0\,C_\ve)^{-2} v(x)$ and moreover, there exists a cube $Q_x \in I$ such that $x \in Q_x$. Thus, by \eqref{eqclaim482} we obtain
$$
 \int_{\d \Omega \setminus  B_k}  v(y) \, d \omega^x_{\Omega_k}  \geq  \int_{10Q_x \cap \d \Omega}  v(y) \, d \omega^x_{\Omega_k} \geq (C_0\,C_\ve)^{-2} \,v(x) \, \omega^x_{\Omega_k}(10Q_x ) \gtrsim v(x).
 $$

We  have shown that for any $x \in \d B_{k+1}$ it holds  $u(x) \leq (1- \delta) v(x)$. By maximum principle the same is true in  $\Omega \setminus B_{k+1}$, which  concludes the proof of Main Claim.

\end{proof}
}
\vv


\section{The corona decomposition for elliptic measure}\label{secorona}

From now on  till the end of the paper, we will assume that $\Omega\subset\R^{n+1}$ is an open set with $n$-AD-regular boundary satisfying the corkscrew condition.
We denote $\mu=\HH^n|_{\partial\Omega}$, and we consider the associated David-Semmes lattice $\DD_\mu$.
 We also denote by $\omega^p$ and $\omega_*^p$ the $L$-harmonic and $L^*$-harmonic measures in $\Omega$ with pole at $p\in\Omega$, respectively. 

A {\it corona decomposition} of $\mu$ is a partition of $\DD_\mu$ into trees.
A family $\TT$ of cubes from $\DD_\mu$ is a {\it tree} if it
verifies the following properties:
\begin{enumerate}
\item $\TT$ has a maximal element (with respect to inclusion) $Q(\TT)$ which contains all the other
elements of $\TT$ as subsets of $\R^{n+1}$. The cube $Q(\TT)$ is the ``root'' of $\TT$.

\item If $Q,Q'$ belong to $\TT$ and $Q\subset Q'$, then any $\mu$-cube $P\in\DD^\mu$ such that $Q\subset P\subset
Q'$ also belongs to $\TT$.
\item If $Q\in\TT$, then either all the children belong to $\TT$ or none of them do.
\end{enumerate}
If $R=Q(\TT)$, we also write $\TT=\tree(R)$.

By arguments quite similar to the ones in \cite{GMT}, we will prove the following:

\begin{propo}\label{propo1}
Let $\Omega\subset\R^{n+1}$, $n\geq2$, be an open set with $n$-AD-regular boundary satisfying the corkscrew condition and let  $L \in \mathcal L(\Omega)$. 
Suppose that one of the following assumptions holds:
\begin{itemize}
\item[(a)] Every bounded $L$-harmonic function in $\Omega$ is $\varepsilon$-approximable for all $\varepsilon  > 0,$ or 
\vv
\item[(b)] There is $C>0$ such that if $u$ is a bounded $L$-harmonic function on $\Omega$ and $B$ is a ball centered at $\partial\Omega$,
\begin{equation}\label{eqhip10}
\int_B |\nabla u(x)|^2\,\delta_\Omega(x)\,dx\leq C\,\|u\|^2_{L^\infty(\Omega)}\,r(B)^n.
\end{equation}
\end{itemize}
Then $\mu$ admits a corona decomposition
$\DD_\mu=\bigcup_{R\in \ttt} \tree(R)$
so that the family $\ttt$ is a Carleson family, that is,
\begin{equation}\label{eqpack1}
\sum_{R\subset S: R\in\ttt}\mu(R)\leq C\,\mu(S)\quad \mbox{for all $S\in\DD_\mu$},
\end{equation}
and for each $R\in\ttt$ there exists a corkscrew point $p_R\in\Omega$ with 
$$c^{-1}\ell(R)\leq\dist(p_R,R)\leq \dist(p_R,\partial\Omega)\leq c\,\ell(R)$$
so that
$$
\omega^{p_R}(3Q)\approx \frac{\mu(Q)}{\mu(R)}\quad\mbox{ for all $Q\in\tree(R)$,}$$
with the implicit constant uniform on $Q$ and $R$.
\end{propo}
\vv

Of course, the analogous result is valid replacing the  operator $L$ by $L^*$ and $\omega$ by $\omega_*$.
Note that for the validity of the proposition  we do not impose any additional regularity assumption on the coefficients of the matrix $A$ (e.g. to be locally Lipschitz).
Most of this section is devoted to the proof of this proposition, which is one of the main steps for the proof of Theorem
\ref{teo1}. We will follow quite closely the arguments from \cite{GMT}.

\vv
\subsection{The approximation lemma}

The next result is an immediate consequence of Lemma \ref{lemholder**1}.

\begin{lemma}\label{lempq0}
There are constants $0<\alpha<1$ and $c_2>0$, depending only on $n$ and the AD-regularity constant of $\mu$ such that the
following holds. For any $0<\ve<1/2$ and any $Q\in\DD_\mu$, we have
$$\omega^x(Q)\geq \omega^x(\tfrac34B_Q)\geq1-c_2\ve^\alpha\quad\mbox{ if $x\in \frac12B_Q$\; and\; $\dist(x,\partial\Omega)\leq \ve\,\ell(Q)$.}$$
\end{lemma}

\vv

For any $Q\in\DD_\mu$, we consider a corkscrew point $p_Q\in \frac12B_Q\cap\Omega$, with 
\begin{equation}\label{eq1}
\delta_\Omega(p_Q)\approx\ve\,\ell(Q),\quad \;\ve\ll1,
\end{equation}
so that
$\omega^{p_Q}(Q)\geq \omega^x(\tfrac34B_Q)\geq 1-C\,\ve^\alpha$. The corkscrew condition of $\Omega$ ensures the existence of such point $p_Q$.
We denote by $y_Q$ a point in $\partial\Omega$ such that
\begin{equation}\label{eq-1}
\delta_\Omega(p_Q) = |y_Q-p_Q|,
\end{equation}
and we assume that $p_Q$ has been chosen so that 
\begin{equation}\label{eq0}
B(y_Q,|y_Q-p_Q|)\subset \tfrac34B_Q.
\end{equation}
For a small constant $0<\tau<1/2$ to be fixed below, we also denote 
$$V_Q= B\big(p_Q,(1-\tau)\delta_\Omega(p_Q)\bigr),$$ so that $V_Q\subset\Omega$. Notice that
$$r(V_Q)\approx \ve \,\ell(Q).$$
\vv

The next lemma is quite similar to Lemma 3.2 in \cite{GMT}. 

\begin{lemma}\label{lemapprox}
Suppose that the constant $\ve$ in \rf{eq1} is small enough and $\tau$ is also small enough.
Let $Q\in\DD_\mu$ and let $E_Q\subset Q$ be such that
$$\omega^{p_Q}(E_Q) \geq (1-\ve)\,\omega^{p_Q}(Q).$$
Then there exists a non-negative $L$-harmonic function $u_Q$ on $\Omega$ and a Borel function $f_Q$ with 
$$u_Q(x) = \int_{E_Q} f_Q\,d\omega^x,\qquad f_Q\leq c\,\chi_{E_Q},$$
so that
\begin{equation}\label{eqkey30}
\int_{V_Q} |\nabla u_Q(x)|^2\,\delta_\Omega(x)\,dx\gtrsim_\tau r(V_Q)^n\approx_\ve \ell(Q)^n.
\end{equation}
\end{lemma}

\begin{proof}
Let $y_Q\in\partial\Omega$ be the point defined in \rf{eq-1}. 
Recall that $\EE_L(x,y)$ is the fundamental solution of the elliptic operator $L$ and satisfies
\begin{equation}\label{eq**1}
0<\EE_L(x,y)\approx \frac1{|x-y|^{n-1}}.
\end{equation}
Consider the function
\begin{equation}\label{deggg*}
g_Q(x) := \int_{B(y_Q,\tau \,r(V_Q))} \frac1{\tau \,r(V_Q)} \,\EE_L(x,y)\,d\mu(y).
\end{equation}
From the AD-regularity of $\mu$ and \rf{eq**1} we deduce that
 $\|g_Q\|_\infty\lesssim 1$. Let $s_Q$ be the closest point to $y_Q$ in $\partial B(p_Q,(1-2\tau)\delta_\Omega(p_Q))$. Obviously, 
 $$B(s_Q,\tau\delta_\Omega(p_Q))\subset  B(p_Q,(1-\tau)\delta_\Omega(p_Q))=:V_Q.$$
From \rf{eq**1} we deduce that 
$$g_Q(x)\approx \tau^n,\quad \mbox{for all $x\in B(s_Q,\tau\delta_\Omega(p_Q))$,}$$
while
$$g_Q(y)\approx 1,\quad \mbox{for all $y\in B(p_Q,\tfrac12\delta_\Omega(p_Q))$.}$$
Hence, if $\tau$ is small enough (depending only on $n$ and the implicit constant in \rf{eq**1}), 
we deduce that
\begin{equation}\label{eqdif**1}
|g_Q(x)- g_Q(y)|\gtrsim 1\qquad \mbox{for all $x\in B(s_Q,\tau\delta_\Omega(p_Q))$ 
and all $y\in B(p_Q,\tfrac12\delta_\Omega(p_Q))$.}
\end{equation}

Now we define $f_Q:=\chi_{E_Q}\,g_Q$ and
$$u_Q(x) := \int f_Q\, d\omega^x = \int_{E_Q} g_Q\,d\omega^x.$$
Since $g_Q$ is $L$-harmonic in $\Omega$ and continuous in $\R^{n+1}$ because of the local $\mu$-integrability of $\EE_L(x,y)$, we have that, for all $x\in\Omega$,
$$g_Q(x) = \int g_Q\,d\omega^x,$$
and then,
\begin{equation}\label{eqggg*0}
\bigl|g_Q(x) - u_Q(x)\bigr| = \left|\int_{\partial\Omega\setminus E_Q} g_Q\,d\omega^x\right|
\leq \|g_Q\|_\infty\,\omega^x(\partial\Omega\setminus E_Q)\lesssim 
\,\omega^x(\partial\Omega\setminus E_Q).
\end{equation}
By \rf{eq1} and  the assumption in the lemma,
$$
\omega^{p_Q}(\partial\Omega\setminus E_Q) = 
\omega^{p_Q}(\partial\Omega\setminus Q) + \omega^{p_Q}(Q\setminus E_Q)
\leq C\ve^\alpha + \ve\lesssim \ve^\alpha,
$$
and then 
 by a Harnack chain argument
it follows that 
\begin{equation}\label{eqggg*1}
\omega^{x}(\partial\Omega\setminus E_Q)\lesssim_\tau\ve^\alpha\quad\mbox{ for all $x\in V_Q$.}
\end{equation}
Therefore, 
$$
\bigl|g_Q(x) - u_Q(x)\bigr|\lesssim_\tau \ve^\alpha\quad\mbox{ for all $x\in V_Q$.}
$$
Assuming $\ve$ small enough, from this estimate and \rf{eqdif**1} we infer that
$$
|u_Q(x)- u_Q(y)|\gtrsim 1\qquad \mbox{for all $x\in B(s_Q,\tau\delta_\Omega(p_Q))$ 
and all $y\in B(p_Q,\tfrac12\delta_\Omega(p_Q))$.}
$$
Now, using Poincar\'e's inequality we derive 
$$\avint_{V_Q}|\nabla u_Q|^2\,dx \gtrsim_\tau \frac1{r(V_Q)^2},$$
which implies the estimate \rf{eqkey30}.
\end{proof}

\vv
From now on we fix $\tau>0$ small enough so that the preceding lemma holds,
and we will drop the dependence on $\tau$ when writing the estimates above.

For the record, note that we have shown above
that
\begin{equation}\label{eq**4}
|u_Q(x)- u_Q(y)|\gtrsim 1\qquad \mbox{for all $x\in V_Q^1$ 
and all $y\in V_Q^2$,}
\end{equation}
where $V_Q^1:=B(s_Q,\tau\delta_\Omega(p_Q))$ and $V_Q^2:=B(p_Q,\tfrac12\delta_\Omega(p_Q))$, with
$V_Q^1\cup V_Q^2\subset V_Q$.
\vv


\subsection{The stopping cubes and the key lemma}\label{secdens1}

Next we construct some stopping cubes analogous to the ones in Section 4 of \cite{GMT}.
Let $0<\delta\ll1$ and $A\gg1$ be some fixed constants. For a fixed a cube $R\in\DD_\mu$,
let $Q\in\DD_\mu$, $Q\subset R$.
We say that $Q\in\HD(R)$ (high density) if $Q$ is a maximal cube satisfying
$$\frac{\omega^{p_R}(2Q)}{\mu(2Q)}\geq A \,\frac{\omega^{p_R}(2R)}{\mu(2R)}.$$
We say that $Q\in\LD(R)$ (low density) if $Q$ is a maximal cube satisfying
$$\frac{\omega^{p_R}(Q)}{\mu(Q)}\leq \delta \,\frac{\omega^{p_R}(R)}{\mu(R)}$$
(notice that $\omega^{p_R}(R)\approx \omega^{p_R}(2R)\approx 1$). Observe that the definition of the family $\HD(R)$ involves the density of $2Q$, while the one of $\LD(R)$ involves the density of $Q$.

We denote
$$B_H(R)=\bigcup_{Q\in \HD(R)} Q \quad\mbox{ and }\quad B_L(R)=\bigcup_{Q\in \LD(R)} Q.$$
\vv

The same arguments as in Lemmas 4.1 and 4.2 of \cite{GMT} yield the following:

\begin{lemma}\label{lemhdld}
We have
$$\mu(B_H(R)) \lesssim \frac1A\,\mu(R)$$
and
$$\omega^{p_R}(B_L(R)) \leq \delta\,\omega^{p_R}(R).$$
\end{lemma}
\vv

Next we denote $\LD^0(R)=\{R\}$, $\LD^1(R)=\LD(R)$, and inductively, for $k\geq1$, 
$$\LD^{k+1}(R) = \bigcup_{Q\in\LD^k(R)} \LD(Q),$$
and the subset of $R$ given by
$$B_L^k(R)=\bigcup_{Q\in\LD^k(R)} Q.$$
Notice that the stopping conditions to define the family of low density cubes $\LD^k(R)$
involve the $L$-harmonic measure $\omega^{p_Q}$ for a suitable $Q\in\LD^{k-1}(R)$, instead of $\omega^{p_R}$.

\vv

\begin{lemma}[Key Lemma]\label{lemclau}
Suppose that either the assumptions (a) or (b) in Theorem \ref{teo1}, or (b') in Remark \ref{remb'} hold. Suppose also that the constant $\ve>0$ in \rf{eq1} is chosen small enough
and that $0 <\delta\leq\ve$. Then for any $m\ge1$ we have
\begin{equation}\label{eqg**p45}
\sum_{k=1}^m \,\sum_{Q\in\LD^k(R)}\mu(Q)\lesssim_\ve\mu(R)
\end{equation}
and
\begin{equation}\label{eqg**p46}
\mu(B_L^m(R))\lesssim_\ve\frac1m\,\mu(R).
\end{equation}
\end{lemma}

\begin{proof}
For $Q\subset\DD_\mu$, $Q\subset R$, we denote 
$$E_Q = Q\setminus B_L(Q).$$
By Lemma \ref{lempq0} and Lemma \ref{lemhdld} applied to $Q$,
\begin{equation}\label{eqhm11}
\omega^{p_Q}(E_Q)  = \omega^{p_Q}(Q) - \omega^{p_Q}(B_L(Q)) \geq (1-\delta)\,\omega^{p_Q}(Q)\geq (1-\delta)(1-c\,\ve^\alpha)\geq 1-c'\ve^\alpha.
\end{equation}
Hence, by Lemma \ref{lemapprox}, if $\ve$ is small enough and $\delta\leq\ve$, there exists a function $u_Q$ 
 on $\Omega$ and a non-negative Borel function $f_Q$ with $f_Q\leq c\,\chi_{E_Q}$ 
 such that
$$u_Q(x) = \int_{E_Q} f_Q\,d\omega^x,$$
satisfying
$$\int_{V_Q} |\nabla u_Q(x)|^2\,\dist(x,\partial\Omega)\,dx\geq c\,r(V_Q)^n.$$

Now let $\AZ$ denote the set of sequences $\{(a_{Q}):Q\in \bigcup_{k=1}^m \LD_k(R),\, a_{Q}=\pm 1\}$, and consider a probability measure $\lambda$ on $\AZ$ that assigns equal probability to $1$ and $-1$. For $a\in \AZ$, set
\[
u_{a}(x)= \sum_{k=1}^m \,\sum_{Q\in\LD^k(R)} a_{Q}\,u_Q.\]
Notice that the set $E_Q$ is contained out of the low density cubes from $\LD(Q)$. Hence, by construction, it turns out that
the sets $E_Q$, for $Q\in\LD^k(R)$, $k\geq1$, are pairwise disjoint.
This implies that the functions
$u_a$
are uniformly bounded by some fixed constant on $\Omega$. Indeed, by the definitions of the functions $u_a$ and $u_Q$,
\begin{equation}\label{eqhm12}
|u_a(x)| \leq  \int \sum_{k=1}^m \,\sum_{Q\in\LD^k(R)}  |a_Q|\,f_Q\,\chi_{E_Q}\,d\omega^x
\leq c\,\sum_{k=1}^m \,\sum_{Q\in\LD^k(R)}\hm^{x}(E_Q)\leq  c.
\end{equation}

\vv
$\bullet$ Suppose first that the assumption (b) in Theorem \ref{teo1} holds.
Let $B(R)$ be some big ball concentric with $R$, with radius comparable to $\ell(R)$, which contains
the sets $V_Q$, $Q\in\LD^k(R)$, $k=1,\ldots,m$.
Since these sets have bounded overlap, by (c) and orthogonality  we get 
\begin{align*}
\mu(R)
& \approx \ell(R)^{n}
 \gtrsim\int \ell(R)^{n} \|u_{a}\|_{\infty}^2\,d\lambda(a) \\
&  \gtrsim \iint_{B(R)} |\grad u_{a}(x)|^{2} \,\delta_\Omega(x)\,dx\, d\lambda(a)\\
& = \int_{B(R)}\int  \Bigl|\sum_{Q}a_Q \grad u_Q(x) \Bigr|^{2}\,  d\lambda(a)\,\delta_\Omega(x)\,dx \\
& = \int_{B(R)} \sum_{Q} |\grad u_Q(x)|^{2} \,\delta_\Omega(x)\,dx\\
& \geq \sum_{Q} \int_{V_Q}  |\nabla u_Q(x)|^{2}\,\delta_\Omega(x)\, dx
\gtrsim \sum_{Q} \ell(Q)^{n} \approx \sum_{Q} \mu(Q),
\end{align*}
where the sums above run over $Q\in \bigcup_k \LD_k(R)$.
This yields the first assertion of the lemma in this case.

\vv
$\bullet$ Suppose now that the hypothesis (a) in Theorem \ref{teo1} holds, i.e., that for all $\varepsilon_0  > 0$ 
every bounded $L$-harmonic function on $\Omega$ is $\varepsilon_0$-approximable. 
So, for some $\ve_0>0$ small enough to be chosen below, and $u$ and $a\in\AZ$ as above, let 
 $\varphi_a =\varphi_{a,\, R}\in W^{1,1}(B(R) \cap \Omega)$ such that 
$\|u_a - \varphi_a\|_{L^{\infty}(B(R) \cap \Omega)} < \varepsilon_0$ and 
\begin{equation}\label{eqcc23}
\int_{B(R) \cap \Omega} |\nabla \varphi_a(y)|\, dy \leq C\,\mu(R),
\end{equation}
where $B(R)$ is as above too.
Recall that from \rf{eq**4} we know that
$$|u_Q(x)- u_Q(y)|\gtrsim 1,\quad \mbox{for all $x\in V_Q^1$ 
and all $y \in V_Q^2$.}$$
Hence we deduce that
$$|m_{V_Q^1} u_Q - m_{V_Q^2} u_Q|\gtrsim 1,$$
for all $Q\in \bigcup_{k=1}^m \LD_k(R)$.
By Kintchine's inequality, we have
\begin{align*}
1\lesssim |m_{V_Q^1} u_Q - m_{V_Q^2} u_Q|&\leq
\Bigl(\sum_P|m_{V_Q^1} u_P - m_{V_Q^2} u_P|^2\Bigr)^{1/2}\\
& \approx \int_\AZ \Bigl|\sum_P a_P\bigl(m_{V_Q^1} u_P - m_{V_Q^2} u_P\bigr)\Bigr|\,d\lambda(a)\\
&= \int_\AZ \Bigl|m_{V_Q^1} u_a - m_{V_Q^2} u_a\Bigr|\,d\lambda(a)\\
& \leq \int_\AZ \Bigl|m_{V_Q^1} \vphi_a - m_{V_Q^2} \vphi_a\Bigr|\,d\lambda(a) + 2\ve_0
,
\end{align*}
where the sums run over $P\in \bigcup_{k=1}^m \LD_k(R)$.
Hence, if $\ve_0$ is small enough we obtain
$$1\lesssim \int_\AZ \Bigl|m_{V_Q^1} \vphi_a - m_{V_Q^2} \vphi_a\Bigr|\,d\lambda(a)$$
for each $Q$.
Thus, integrating on $V_Q$ and summing over $Q$, by Poincar\'e's inequality and the assumption (a) we obtain
\begin{align*}
\sum_Q \mu(Q) & \lesssim \sum_Q\ell(Q)^n \lesssim \sum_Q\int_{V_Q} \int_\AZ \frac1{\ell(Q)}\,\Bigl|m_{V_Q^1} \vphi_a - m_{V_Q^2} \vphi_a\Bigr|\,d\lambda(a)\,
dx\\
&
\lesssim \sum_Q \int_\AZ\int_{V_Q} \bigl|\nabla \vphi_a \bigr|\,
dx\,d\lambda(a)
\leq\int_\AZ\int_{B(R)} \bigl|\nabla \vphi_a \bigr|\,
dx\,d\lambda(a)
\stackrel{\eqref{eqcc23}}{\lesssim} \mu(R).
\end{align*}
 This  completes the proof of the first assertion of the lemma.
\vvv

The second estimate in the lemma follows from the fact that if $Q\in B_L^m(R)$, then $x$ belongs to $m$ different cubes $Q\in\LD^k(R)$, $k=1,\ldots,m$. So
$$\sum_{k=1}^m \,\sum_{Q\in\LD^k(R)}\chi_Q(x) = m,$$
and by Chebyshev,
$$\mu(B_L^m(R))\leq \frac1m\,\sum_{k=1}^m \,\sum_{Q\in\LD^k(R)}\mu(Q) \lesssim_\ve \frac1m\,\mu(R).$$

\vv
 $\bullet$ To prove that the lemma holds assuming (b') in Remark \ref{remb'} it is enough to notice that, by a small modification of the arguments, one can assume the functions
$u_a$ above to be continuous. For further details, see Remark 3.8 in \cite{GMT}. 
\end{proof}
\vv

\subsection{The end of the proof of Proposition \ref{propo1}}\label{subsec:endofpropo1}

To complete the proof of this proposition,  we will define the family $\ttt\subset\DD_\mu$ exactly as in 
Subsection 3.4 of \cite{GMT}, and then we will prove that it satisfies the packing condition 
\rf{eqpack1}, arguing as in
 \cite{GMT}.
For the convenience of the reader we will repeat here the definition of $\ttt$.
Given a cube $R\in\DD_\mu$ we let 
$$\sss(R):=\{ S \in \HD(R)\cup\LD(R):  \nexists \,\,\wt S \in \HD(R)\cup\LD(R) \,\, \textup{such that}\,\,  S \subsetneq \wt S\}.$$
Note that $\sss(R)$ is a family of pairwise disjoint cubes. We set 
$$\tree(R):=\{ Q \in \DD_\mu (R): \nexists \,\, S \in \sss(R) \,\, \textup{such that}\,\, Q\subset S\}.$$ 
 In particular, note that $\sss(R)\not\subset\tree(R)$.
Then, arguing as in Lemma 3.6 of \cite{GMT}, it follows easily that
$$
\omega^{p_R}(3Q)\approx \frac{\mu(Q)}{\mu(R)}\quad\mbox{ for all $Q\in\tree(R)$,}$$
with the implicit constant uniform on $Q$ and $R$.


Suppose first that $\partial\Omega$ is bounded and fix a cube $R_0\in\DD_\mu$ so that $\partial\Omega = R_0$  and we define the family of the top cubes in $R_0$ as follows:
first we define the families $\ttt_k$ for $k\geq0$ inductively. We set
$$\ttt_0=\{R_0\}.$$
Assuming that $\ttt_k$ has been defined, we set
$$\ttt_{k+1} = \bigcup_{R\in\ttt_k}\sss(R),$$
and then, 
\begin{equation}\label{eqtopr0}
\ttt=\ttt(R_0)=\bigcup_{k\geq0}\ttt_k.
\end{equation}
Notice that
$$\DD_\mu= \bigcup_{R\in\ttt}\tree(R),$$
and this union is disjoint.


Now, by the same arguments of Lemma 3.9 from \cite{GMT}, using Lemmas \ref{lemhdld} and \ref{lemclau}, we derive
\begin{lemma}\label{lemcarleson}
Under the assumptions of Lemma \ref{lemclau}, there exists a constant $C$ such that
for any $Q_0\in\DD_\mu(R_0)$,
\begin{equation}\label{eqpack00}
\sum_{R\in\ttt:R\subset Q_0} \mu(R)\leq C\,\mu(Q_0).
\end{equation}
\end{lemma}

\vv
Arguing as in the proof of the proof
of Proposition 3.1 in \cite{GMT}, it follows that the family $\ttt$ satisfies the properties required
in the corona decomposition in Proposition \ref{propo1}.  

\vv In the case when $\partial\Omega$ is not bounded we apply a technique described in p.\ 38 of \cite{DS1}: we consider a family of cubes $\{R_j\}_{j\in J}\in\DD_\mu$ which are pairwise disjoint, whose union is all of $\supp\mu$, and which have the property 
that for each $k$ there at most $C$ cubes from $\DD_{\mu,k}$ not contained in any cube $R_j$. For each
$R_j$ we construct a family $\ttt(R_j)$ analogous to the above $\ttt(R_0)$ in connection with $R_0$,  as in \rf{eqtopr0}.
Then we set 
$$\ttt=\bigcup_{j\in J} \ttt(R_j) \cup \BB,$$
where $\BB\subset\DD_\mu$ is the family of cubes which are not contained in any cube $R_j$, $j\in J$. 
One can easily check that the family $\ttt$ satisfies all the properties from Proposition \ref{propo1}. See p.\ 38 of \cite{DS1}
for the construction of the family $\{R_j\}$ and additional details.

\vv
\subsection{Corona decomposition if elliptic measure satisfies a weak-$A_\infty$ condition}

Here we prove the direct analogue of Proposition \ref{propo1}.
\begin{propo}\label{propo2}
Let $\Omega\subset\R^{n+1}$, $n\geq2$, be an open set with $n$-AD-regular boundary satisfying the corkscrew condition and  $L \in \mathcal{L}(\Omega)$. 
Let $c_0 \in (0,1)$ small enough depending only on $n$, the AD-regularity and the ellipticity constants. Suppose that there exist $\ve \in (0,1)$ and $\ve' \in (0,c)$, where $c<1$ is the constant in \eqref{e:bourgain}, such that for every ball $B$ centered at $\d \Omega$ with $\diam(B) \leq \diam(\Omega)$ there exists a corkscrew point $x_B \in \frac{1}{2}B \cap\Omega$ with $\dist(x_B, \d \Omega)\geq c_0 r(B)$, so that for any subset $E \subset B \cap \d \Omega$, \eqref{eq:Amu1} holds.
Then $\mu$ admits a corona decomposition
$\DD_\mu=\bigcup_{R\in \ttt} \tree(R)$
so that the family $\ttt$ is a Carleson family, that is,
\begin{equation}\label{eqpack2}
\sum_{R\subset S: R\in\ttt}\mu(R)\leq C\,\mu(S)\quad \mbox{for all $S\in\DD_\mu$},
\end{equation}
and for each $R\in\ttt$ there exists a corkscrew point $p_R\in\Omega$ with 
$$c^{-1}\ell(R)\leq\dist(p_R,R)\leq \dist(p_R,\partial\Omega)\leq c\,\ell(R)$$
so that
$$
\omega^{p_R}(3Q)\approx \frac{\mu(Q)}{\mu(R)}\quad\mbox{ for all $Q\in\tree(R)$,}$$
with the implicit constant uniform on $Q$ and $R$.
\end{propo}

\begin{proof}
Let $p_S$ be a Corkscrew point associated with the ball $B_S\subset S \in \DD_\mu$ and let $\wt B_S$ be a ball containing $S$ with radius comparable to $\ell(S)$.  
It is rather easy to check that \eqref{eq:Amu1} holds (with different constants) if instead of balls we have cubes in $\DD_\mu$. Indeed, let $E \subset S \in \DD_\mu$ such that $\omega^{p_S}(E) > \eta \,\omega^{p_S}(S) \geq c\,\eta\,\omega^{p_S}(\wt B_S)$, where we used \eqref{e:bourgain}. Since $ E \subset S \subset \wt B_S$, if we choose $\eta$ so that $c\,\eta \geq \ve'$ (here we use $\ve'/c <1$), then by \eqref{eq:Amu1} we have that $\mu(E) >  \ve\,\mu( \wt B_S) \geq \ve\,\mu(S)$. 
We have shown that  there exist $\eta,\ve \in (0,1)$ that  so for any $E \subset S \in \DD_\mu$,
\begin{equation*}
\mbox{if} \quad  \mu(E)\leq  \ve  \,\mu(S)  \quad\text{ then }\quad \hm^{x_{S}}(E)\leq   \, \eta \,\hm^{x_S}(S).
\end{equation*} 
and by taking $F= S\setminus E$,
\begin{equation}\label{eq:Amu1-dyadic}
\mbox{if} \quad  \hm^{x_{S}}(F)\leq   \,(1-\eta)\,\hm^{x_S}(S)  \quad\text{ then }\quad \mu(F)\leq  (1- \ve) \,\mu(S),\quad
\mbox{ for all $F\subset S$.}
\end{equation}

Let us fix $R	 \in \DD_\mu$ and let $p_R$ be a Corkscrew point associated with the ball $B_R \subset R$. We 
set $\delta=A^{-1}$ and we define $\HD(R)$, $\LD(R)$, 
$B_H(R)$, and $B_L(R)$ and in Subsection \ref{secdens1}.
Then it holds
\begin{align}
\mu(B_H(R))  &\leq C_n\,C_0^2\, A^{-1} \mu(R),\label{eq:bad2-mu}\\
\hm^{p_R}(B_L(R) )   &\leq  A^{-1} \hm^{p_R}(R),\label{eq:bad1-hm} 
\end{align}
where $C_n>0$ is a dimensional constant and $C_0$ is the constant in the AD-regularity condition. If we choose $A$ big enough so that $ A^{-1}< 1-\eta $, by \eqref{eq:Amu1-dyadic} we have that
$$\mu(B_L(R)) \leq  (1-\ve) \,\mu(R).$$
If we further assume that $C_n\,C_0^2\, A^{-1}  < \ve/2$, we obtain
\begin{equation}\label{eq:mu(stop)}
 \mu(B_H(R) \cup B_L(R)) < (1- \ve/2) \mu(R)=:\tau \mu(R).
 \end{equation}
 
Given a cube $R\in\DD_\mu$ we define the families $\sss(R)$ and $\tree(R)$  as in Subsection \ref{subsec:endofpropo1}, as well as the family $\ttt$. 
By construction, it follows that
$$
\omega^{p_R}(3Q)\approx \frac{\mu(Q)}{\mu(R)}\quad\mbox{ for all $Q\in\tree(R)$,}$$Note that $\sss(R)$ is a family of pairwise disjoint cubes and in view of \eqref{eq:mu(stop)}, it holds
\begin{equation}\label{eq:mu(st-bis)}
\sum_{S \in \sss(R)} \mu(S) \leq  \mu(B_H(R) \cup B_L(R)) <\tau \mu(R).
 \end{equation} 

From the preceding estimate the packing condition \rf{eqpack2} follows easily.
Indeed, for $R\in\ttt$, set $\sss^0(R)=\{R\}$ and, for $k\geq1$,
$$\sss^k(R) = \bigcup_{Q\in\sss^{k-1}(R)}\sss(Q),$$
so that 
$$\ttt\cap \DD_\mu(R) = \bigcup_{k\geq0}\sss(R).$$
Then we have
\begin{align*}
\sum_{Q \in \ttt:Q\subset R} \mu(Q) &= \sum_{k=0}^\infty \sum_{Q \in \sss^{k}(R)} \mu(Q) = \sum_{k=0}^\infty \sum_{S \in \sss^{k-1}(R)} \sum_{Q \in \sss(S)}\mu(Q) \\
&\leq  \sum_{k=0}^\infty  \sum_{S\in \sss^{k-1}(R)} \tau\,\mu(S) \leq \dots \leq \sum_{k=0}^\infty \tau^{k} \mu(R) \lesssim \mu(R).
\end{align*}
Thus, for a general cube $S\in\DD_\mu$, denoting by $\MM(S)$ the family of maximal cubes from $\ttt\cap\DD_\mu(S)$, we
have
$$
\sum_{Q \in \ttt:Q\subset S} \mu(Q)= \sum_{R \in \MM(S)} \sum_{Q\in\ttt:Q\subset R} \mu(Q) \lesssim \sum_{R \in \MM(S)} \mu(R)
\leq \mu(S),
$$  
which proves the packing condition \eqref{eqpack2}. 
 \end{proof}

\vv

\subsection{The mixed corona decomposition}

Consider the corona decompositions for $\omega$ and $\omega_*$ described in Proposition \ref{propo1}, with the associated pfamilies $\ttt$ and $\ttt_*$, which induce the partitions
$$\DD_\mu = \bigcup_{R\in\ttt}\tree(R),\qquad \DD_\mu = \bigcup_{R'\in\ttt_*}\tree_*(R').$$
Thus,
$$\DD_\mu = \bigcup_{R\in\ttt}\, \bigcup_{R'\in\ttt_*}\tree(R)\cap \tree_*(R').$$

 It is immediate to check
that if $\TT,\TT'\subset \DD_\mu$ are two non-disjoint trees (i.e., they contain some common cube from $\DD_\mu$), then
the family $\TT\cap\TT'$ is also a tree. That is, $\TT\cap\TT'$ satisfies the properties (1), (2), (3) stated at the beginning of
Section \ref{secorona}. Therefore, for $R\in\ttt$ and $R'\in\ttt_*$, $$\wt \tree(R,R'):= \tree(R)\cap \tree_*(R')$$
is a tree (unless this is empty).
 In this  case its root $Q$ coincides  either with $R$ or $R'$. Abusing notation we also write 
$\wt\tree(Q):=\wt \tree(R,R')$. So we have
$$\DD_\mu = \bigcup_{Q\in\wt\ttt} \wt\tree(Q), \quad \mbox{ where $\wt\ttt=\ttt\cup\ttt_*$.}
$$
By Proposition \ref{propo1} we have that $\ttt\cup\ttt_*$ satisfies the packing condition
$$
\sum_{R\subset S: R\in\wt\ttt}\mu(R)\leq 
\sum_{R\subset S: R\in\ttt}\mu(R) + \sum_{R\subset S: R\in\ttt_*}\mu(R)\leq C\,\mu(S),$$
for every $S\in\DD_\mu$.
So the following holds:

\begin{propo}\label{propo**}

Under the assumptions 
either of Theorem \ref{teo1} or Theorem \ref{teo2}, $\mu$ admits a corona decomposition
$\DD_\mu=\bigcup_{R\in \wt\ttt} \wt\tree(R)$
so that the family $\wt\ttt$ is a Carleson family, that is,
\begin{equation}\label{eqpack1**}
\sum_{R\subset S: R\in\wt\ttt}\mu(R)\leq C\,\mu(S),\quad \mbox{for all $S\in\DD_\mu$},
\end{equation}
and for each $R\in\wt \ttt$, there exist cubes $R_1,R_2\supset R$ and corkscrew points $p_{R_1},p_{R_2}\in\Omega$ with 
$$c^{-1}\ell(R_i)\leq\dist(p_{R_i},R_i)\leq \delta_\Omega(p_{R_i})\leq c\,\ell(R_i)\qquad \mbox{ for $i=1,2$,}$$
so that
$$
\omega^{p_{R_1}}(3Q)\approx \frac{\mu(Q)}{\mu(R_1)}\quad \mbox{and}\quad\omega_*^{p_{R_2}}(3Q)\approx \frac{\mu(Q)}{\mu(R_2)},$$
for all $Q\in\wt\tree(R)$,
with the implicit constants uniform on $Q$ and $R$.
\end{propo}

The proof follows by applying Proposition \ref{propo1} or \ref{propo2} both to $\omega$ and $\omega_*$ and considering the ``mixed corona decomposition" described above. 



\vv


\section{From the corona decomposition for elliptic measure to uniform rectifiability}\label{secur}

Our next objective consists in showing that if $\mu$ admits a corona decomposition involving elliptic measure 
such as the one in Proposition \ref{propo**}, then $\mu$ is uniformly rectifiable. Once we show this, those propositions, Theorems \ref{teo1} and \ref{teo2} will readily follow. The first idea consists in adapting the arguments by Hofmann, Le, Martell and Nystr\"om in \cite{HLMN},  which in turn are based on some of the techniques from \cite{LV}, to the case of elliptic measure. 
However, we will find when the matrix $A$ is non-symmetric, some big obstacles appear, and some additional connectivity
arguments will be required.

From now on we will assume that the matrix $A$ satisfies \rf{eqelliptic1}, \rf{eqelliptic2}, and \rf{eqelliptic3}.

\vv

\subsection{The conditions $\whsa$, $\batpp$, and $\wts$}

Following \cite{HLMN}, given $K_0\gg1$ and $0<\ve<K_0^{-6}$,
we say that a cube $Q\in\DD_\mu$ satisfies the ``$\ve$-weak half space approximation property'', and we write $Q\in \whsa_\ve$, if there is a half-space $H=H(Q)$ and a hyperplane $P=P(Q)=\partial
H$ so that
\begin{itemize}
\item $\dist(z,\supp\mu)\leq \ve\,\ell(Q)$ for every $z\in P\cap B(x_Q,\ve^{-2}\ell(Q))$,
\item $\dist(Q,P)\leq K_0^{3/2}\,\ell(Q)$, and
\item $H\cap B(x_Q,\ve^{-2}\ell(Q))\cap \supp\mu =\varnothing$.
\end{itemize}

As shown in \cite[Proposition 1.17]{HLMN},  $\mu$ is uniformly rectifiable if and only if, for every $\ve>0$, the family of cubes $Q\in\DD_\mu$ that do not satisfy the
$\whsa_\ve$ property is a Carleson family. That is, for each $S\in\DD_\mu$,
$$\sum_{Q\subset S:\,Q\not\in\whsa_\ve} \mu(Q)\leq C\,\mu(S).$$
This criterion is used then in \cite{HLMN} to show that the so-called weak $A_\infty$ property of harmonic
measure implies uniform rectifiability.

For our purposes, the preceding criterion is not suitable, and we will need a somewhat more powerful condition.
First we introduce the $\batpp_\ve$ cubes.
We say that a cube $Q\in\DD_\mu$ satisfies the condition of ``bilateral $\ve$-approximation by two parallel planes'', and we write $Q\in\batpp_\ve$,
if there are two parallel $n$-planes $P_1,P_2\subset\R^{n+1}$ such that
\begin{equation}
\label{batpp1}
\dist(z,P_1\cup P_2)\leq \ve\,\ell(Q)\mbox{ for every }z\in \supp\mu\cap B(x_Q,10\ell(Q)),\mbox{ and }
\end{equation}
\begin{equation}
\label{batpp2}
\dist(z,\supp\mu)\leq \ve\,\ell(Q)\mbox{ for every }z\in (P_1\cup P_2)\cap B(x_Q,10\ell(Q)).
\end{equation}
This condition is a variant of the $\mathsf{BAUP}$ condition from \cite[Chapter II.3]{DS2}, and
by Proposition II.3.18 in this reference, it turns out that $\mu$ is uniformly rectifiable if and only if,
given any $\ve>0$,
for each $S\in\DD_\mu$,
$$\sum_{Q\subset S:\,Q\not\in\batpp_\ve} \mu(Q)\leq C(\ve)\,\mu(S).$$

\begin{definition}\label{WTS}
We say that a cube $Q\in\DD_\mu$ is {\it weak topologically satisfactory} or $\wts$ if there are constants $A_{0},\alpha,p,t,\tau>0$ and connected sets $U_{1}(Q),U_{2}(Q),U_{1}'(Q),U_{2}'(Q)\subset A_0 B_{Q}$ so that
\begin{enumerate}
\item $\{x\in 10B_{Q}: \dist(x,E)>\tau\ell(Q)\}\subset U_{1}(Q)\cup U_{2}(Q)$,
\item $U_i(Q)\subset U_i'(Q)$ and $U_{1}'(Q)\cap U_{2}'(Q)=\varnothing$. 
\item For $i=1,2$ and all $x\in 10B_{Q}\cap E$ and $t\ell(Q)<r<10\ell(Q)$, there is a corkscrew ball $B(y,p\ell(Q))\subset U_{i}(Q)\cap B(x,r)$. 
\item For $x,y\in U_{i}(Q)$, there is a curve $\Gamma\subset U_{i}$ containing $x,y$ so that $\dist(\Gamma,E)\gec \alpha \ell(Q)$. 
\end{enumerate}
If we want to make explicit the dependence of $\wts$ on $A_{0},\alpha,p,t,\tau$, we will write instead $\wts(A_0,\alpha,p,t,\tau)$.
The property $\wts$ is a variant of the so-called {\it weak topologically nice} or $\wtn$ condition from \cite{DS2}. 

Given $a_0\geq1$, we say that the {\em compatibility condition} holds for some family $\FF\subset \wts$ if for all $P,Q\in \FF$ such that $2^{-a_{0}}\ell(Q)\leq  \ell(P)\leq \ell(Q)$, it holds that  $U_{i}(P)\cap 10B_{Q}\subset U_{i}'(Q)$.
\end{definition}

In Section \ref{secuniform} we will prove the following result.

\def\H{\HH}
\begin{propo}\label{propowtf}
Let $\mu$ be an $n$-AD-regular measure in $\R^{n+1}$. Let $\FF\subset \wts$ be some family of cubes satisfying
the compatibility condition, and suppose that for every $S\in \DD_{\mu}$
\begin{equation}\label{Bsum}
\sum_{Q\subset S \atop Q\not\in \batpp_{\ve} \cup \FF} \mu(Q) \lec \mu(S) .
\end{equation}
For appropriate choices of constants in the definitions of $\batpp_{\ve}$, $\wts$, and the compatibility condition, $\mu$ is uniformly rectifiable. 
\end{propo}

This will be the criterion we will use to prove that if $\mu$ admits a corona decomposition involving elliptic measure 
such as the one in Proposition \ref{propo1}, then $\mu$ is uniformly rectifiable. 
Our next objective consists in showing show that the existence of such corona decomposition implies that the
assumptions in Proposition \ref{propowtf} are fulfilled, and afterwards we will turn to the proof of the proposition.


\subsection{Preliminary notation and auxiliary lemmas}\label{sub95}

From now on we assume that $\mu$ admits a corona decomposition involving elliptic measure 
such as the one in Proposition \ref{propo1}.

We denote by $\WW=\WW(\Omega)$ a collection of closed dyadic Whitney cubes of $\Omega$, so that the cubes in $\WW$
cover $\Omega$, have non-overlapping interiors, and satisfy
$$8\diam(I)\leq \dist(8I,\partial\Omega)\leq \dist(I,\partial\Omega)\leq 80\,\diam(I),\quad\mbox{ for all $I\in\WW$,}$$
and
$$\diam(I_1)\approx\diam(I_2),\quad\mbox{ for $I_1,I_2\in\WW$ such that $I_1\cap I_2\neq\varnothing$.}$$
For some positive constants $k_0\ll1$ and $K_0\gg1$ to be fixed below and $Q\in\DD_\mu$, we set
$$\WW_Q = \{I\in\WW:\,k_0\ell(Q)\leq\ell(I)\leq K_0\,\ell(Q)\;\mbox{ and }\;\dist(I,Q)\leq K_0\,\ell(Q)\}.$$
We then define the Whitney region $U_Q$ associate with $Q$ as follows: for a small positive parameter $\tau \ll 1$, let 
$$I^*:=(1+\tau)I$$
be the enlarged cubes that preserve the properties of Whitney cubes. Namely,
$$\diam(I) \approx\diam(I^*)\approx \dist(I^*,\partial\Omega)\approx \dist(I,\partial\Omega).$$
Then, define the Whitney regions with respect to $Q$ by
$$U_Q = \bigcup_{I\in \WW_Q}I^*.$$
These Whitney regions may be non-connected, but the total number of connected components is at most $\#\WW_Q$, and so this is bounded above by some constant
depending on $K_0$. We denote the connected components of $U_Q$ by $\{U_Q^i\}_i$.

We consider the $L$-Green  function $G$, and some fixed cubes $R\in\wt\ttt$,  and $R_i\supset R$, $i=1,2$, given in Proposition \ref{propo**}. For such fixed cubes, we denote
$$u = \mu(R_1)\,G(p_{R_1}, \cdot),\quad u _*:=\mu(R_2)\,G(\cdot, p_{R_2}),\quad \nu = \mu(R_1)\,\omega^{p_{R_1}},\quad \nu_* = \mu(R_2)\,\omega_{*}^{p_{R_2}}.$$
Note that
$$\nu(3Q)\approx \nu_*(3Q)\approx \mu(Q), \quad\mbox{ for all $Q\in\wt\tree(R)$,}$$
by the properties of the corona decomposition.
\vv

\begin{lemma}\label{lem:failure}
There is some constant $C_2$ depending on the parameters of the corona decomposition in Proposition \ref{propo1} such that the following holds.
For each $Q\in\wt\tree(R)$ with $\ell(Q)\leq C_2^{-1}\ell(R)$ there exists $\wt Y_Q\in \Omega$ with 
$$C_2^{-1}\,\ell(Q)\leq\delta_\Omega(\wt Y_Q)\leq |\wt Y_Q-x_Q|\leq C_2\, \ell(Q)$$ such that
\begin{equation}\label{eq-*1}
\frac{\nu(3Q)}{\mu(Q)}\approx |\nabla u(\wt Y_Q)| \approx \frac{u(\wt Y_Q)}{\delta_\Omega(\wt Y_Q)}\approx1,
\end{equation}
with the implicit constant depending on the parameters of the corona decomposition in Propostion \ref{propo1}. The same estimates hold for $\nu_*$ and $u_*$, for a possibly different point $\wt Y^*_Q\in \Omega$, satisfying also $C_2^{-1}\,\ell(Q)\leq\delta_\Omega(\wt Y_Q^*)\leq |\wt Y_Q^*-x_Q|\leq C_2\, \ell(Q)$.
\end{lemma}

To prove this one uses Lemma \ref{lem:grad-u-regul} and Harnack's inequality to obtain $ |\nabla u(\wt Y_Q)| \lesssim \frac{u(\wt Y_Q)}{\delta_\Omega(\wt Y_Q)}$. Apart from this, the proof of this lemma is basically the same as the one of Lemma 4.24 and (5.7) from \cite{HLMN}, and so we skip it.
\vv

\begin{lemma}\label{lem29} 
Let $0<\tau_0\leq1$. Given $Q\in\wt\tree(R)$ with $\ell(Q)\leq c_3\,\tau_0\,\ell(R)$,
suppose that
$$20B_Q\cap \{x:u(x)>\tau_0\,\ell(Q)\}\cap \{x:u_*(x)>\tau_0\,\ell(Q)\}\neq\varnothing.$$
If $k_0$ is chosen small enough (depending on $\tau_0$), and $K_0$ big enough, then there exists some point $Y_Q\in U_Q$ such that
$$\min\{ u(Y_Q),\, u_*(Y_Q)\}\gtrsim_{\tau_0}\,\ell(Q)\quad \text{ and } \quad|\nabla u_*(Y_Q)|\gtrsim_{\tau_0} 1.$$
\end{lemma}

\begin{proof}
Let $Z_Q\in 20 B_Q$ be such that $\min\{ u(Z_Q),\, u_*(Z_Q)\}> \tau_0\,\ell(Q)$.
Consider the open set
$\{x\in 20B_Q:\,u_*(x)>\tau_0\,\ell(Q)/2\}$, and denote by $V$ the connected component that contains $Z_Q$.
Denote by $L$ the segment $[Z_Q,x_Q]$ (recall that $x_Q$ is the center of $Q$), and let $Z'$ be the closest point in $\partial V$
which lies on $L$. Since $u_*(Z')=\tau_0\,\ell(Q)/2$, by the mean value theorem there exists some point $Y_Q\in (Z_Q,Z')$
such that
$$|\nabla u_*(Y_Q)|\geq \frac{|u_*(Z_Q)-u_*(Z')|}{|Z_Q-Z'|} \gtrsim \frac{\tau_0\,\ell(Q)}{\ell(Q)} 
= \tau_0.$$
Since $Y_Q\in V$, we also have $u_*(Y_Q)>\tau_0\,\ell(Q)/2$. 

It remains to show that  $u(Y_Q)\gtrsim_{\tau_0}\ell(Q)$ and that $Y_Q\in U_Q$. To this end,
note that by Lemma \ref{l:w>G}, $u_*(x)\lesssim\ell(Q)$ for all $x\in 40B_Q$, and hence by Lemma \ref{lem333} we have
$$u_*(x)\lesssim \left(\frac{\delta_\Omega(x)}{\ell(Q)}\right)^\alpha\,\ell(Q)
\quad\mbox{ for all $x\in 20B_Q$.}$$
Therefore,
\begin{equation}\label{eqvv0}
\dist(V,\supp\mu)\gtrsim_{\tau_0}\ell(Q).
\end{equation}
This shows that $V\subset U_Q$ if $k_0$ is chosen small enough, depending on $\tau_0$, and $K_0$ some absolute constant big enough. 
Further, by \rf{eqvv0} again and since $V$ is connected, there is Harnack chain of balls 
with radius comparable to $\dist(V,\supp\mu)$ which connects $Z_Q$ and and $Y_Q$, with the number of balls bounded 
by some constant depending on  $\tau_0$. This implies that  $u(Y_Q)\gtrsim_{\tau_0}\,\ell(Q)$.
\end{proof}
\vv

\begin{rem}
For the arguments below, it is important to note that we allow $k_0$ to depend on $\tau_0$, while $K_0$
should be consider as an absolute constant independent of $\tau_0$.
\end{rem}

We need now some additional notation. For a small parameter $\ve>0$, $Q\in\DD_\mu$, and $X,Y\in\Omega$, we write $X\approx_{\ve,Q}\!Y$ if $X$ and $Y$ may be connected by a chain of at most $\ve^{-1}$ balls $B(Z_k,\delta_\Omega(Z_k)/2)$, with $\ve^3\ell(Q)\leq\delta_\Omega(Z_K)\leq \ve^{-3}\ell(Q)$.
Then we set
$$\wt U_Q^i=\bigl\{X\in \Omega:X\approx_{\ve,Q} Y,\;\mbox{ for some $Y\in U_Q^i$}\bigr\}.$$
Note here that $\wt U_Q^i$ are enlarged versions of $ U_Q^i$ and it may happen that $\wt U_Q^i \cap \wt U_Q^j \neq \varnothing$, for $i \neq j$. Although, the overlap is uniformly controlled.  
 
For $X\in\Omega$, we set
\begin{equation}\label{eqmmm3}
B_X = \overline B\bigl(X,(1-\ve^{2M/\alpha})\,\delta_\Omega(X)\bigr),
\end{equation}
where $M\gg1$ is some constant big enough and $0<\alpha<1$ is as in Lemma \ref{lem333}.

For $Q\in\DD_\mu$ we consider the augmented regions $U_Q^{j,*}$ and $ U_Q^*$ defined as follows. We set
\begin{equation}\label{defuq1}
\WW_Q^{j,*} = \bigl\{I\in\WW: I^* \cap B_Y\neq\varnothing\;\mbox{ for some $Y\in\bigcup_{X\in\wt U_Q^j} B_X$}\bigr\},
\end{equation}
and then,  if $I^{**}:=(1+ 2\tau)I$, where $\tau\in (0,1)$ is the parameter in the definition of $I^*$, we let 
\begin{equation}\label{defuq2}
U_Q^{j,*} = \bigcup_{I\in \WW_Q^{j,*}} I^{**}\quad\mbox{and}\quad U_Q^{*}= \bigcup_j U_Q^{j,*}.
\end{equation}
Note that, by \rf{eq-*1} and a Harnack chain argument, there exist a connected component $U^{i}_Q$ such that
\begin{equation}
u(Y)\approx_{\ve,M,K_0} \delta_\Omega(Y),\quad\mbox{ for all $Y\in U_Q^{i,*}$.}
\end{equation}
The same inequalities hold in the situation of Lemma \ref{lem29} if $U_Q^i$ is the connected component of
$U_Q$ containing the point $Y_Q$ in that lemma.
Further, the upper estimates
$$u(Y)\lesssim_{\ve,M,K_0} \delta_\Omega(Y), \quad\; |\nabla u(Y)|\lesssim_{\ve,M,K_0} 1$$
hold for all $Y\in U^*_Q$, because $\nu(3Q)\lesssim \mu(Q)$ (recall the definition of $\nu$) and  Lemma \ref{l:w>G}. The same estimates are  true for $u _*$  since $\nu_*(3Q)\approx \mu(Q)$ for all $Q \not \in \HD^*(R)$. 

\vv

\subsection{Types of cubes}\label{subtype}

We need now to distinguish different types of cubes. To this end we consider constants $A_0,M\gg1$ (recall $A_0$ appears in
Definition \ref{WTS} and $M$ in \rf{eqmmm3}) and $0<\kappa_0,\theta_0,\tau_0\ll1$ to be fixed later. Further we assume that $K_0$ is big enough so that Lemma \ref{lem29} holds. We set:
\vv

\begin{itemize}
\item {\bf Type 0:} $Q\in\wt\tree(R)$, which satisfies at least one the following conditions:
\vv

\begin{enumerate}
\item  there exists some cube $P\in\DD_\mu\setminus \wt\tree(R)$ such that $\kappa_0\,\ell(Q)\leq \ell(P)\leq\kappa_0^{-1}\ell(Q)$
and $\dist(P,Q)\leq\kappa_0^{-1}\ell(Q)$; or
\vv

\item  $\sup \bigl\{ \frac{|A(y_1)-A(y_2)|}{ |y_1-y_2|}: y_1,y_2\in \kappa_0^{-1} B_Q,\, \delta_\Omega(y_i)\geq \kappa_0\,\ell(Q)
\bigr\}
>\theta_0\, \ell(Q)^{-1}$.
\end{enumerate}
\vv
\noi We write $Q\in \case(0,R)$ if (1) or (2) holds.
\vv

\item {\bf Type 1:} $Q\in\wt\tree(R)$, with $Q\not\in\case(0,R)$, such that
$$20B_Q\cap \{x:u(x)>\tau_0\,\ell(Q)\}\cap \{x:u_*(x)>\tau_0\,\ell(Q)\}\neq\varnothing$$
and, if $U_Q^i$ is the connected component of
$U_Q$ containing the point $Y_Q$ in Lemma \ref{lem29},
\begin{equation}\label{eqcase1}
\sup_{X\in \wt U^{i}_Q} \sup_{z\in B_X} |\nabla u _*(z) - \nabla u _*(Y_Q)| > \ve_0^M.
\end{equation}
 We write $Q\in \case(1,R)$ in this case.

\vv
\item {\bf Type 2:} $Q\in\wt\tree(R)$, with $Q\not\in\case(0,R)$, such that
$$20B_Q\cap \{x:u(x)>\tau_0\,\ell(Q)\}\cap \{x:u_*(x)>\tau_0\,\ell(Q)\}\neq\varnothing$$
and, if $U_Q^i$ is the connected component of
$U_Q$ containing the point $Y_Q$ in Lemma \ref{lem29},
\begin{equation}\label{eqcase2}
\sup_{X\in \wt U^{i}_Q} \sup_{z\in B_X} |\nabla u _*(z) - \nabla u _*(Y_Q)| \leq \ve_0^M.
\end{equation}
 We write $Q\in \case(2,R)$.
 \vv
 
\item {\bf Type 3:} $Q\in\wt\tree(R)$, with $Q\not\in\case(0,R)$, such that
$$20B_Q\cap \{x:u(x)>\tau_0\,\ell(Q)\}\cap \{x:u_*(x)>\tau_0\,\ell(Q)\}=\varnothing.$$
 We write $Q\in \case(3,R)$.
 
\end{itemize}

\vv
Regarding the choice of the constants above, we remark that $k_0$ depends on $\tau_0$, which will be chosen in Section \ref{secuniform}. On the other hand, recall that we ask $\ve\leq K_0^{-6}$. Further, we may choose $M=100$, say.
Further, we remark that $\theta_0$ may depend on all other constants and, in particular, we will ask $\theta_0\ll \ve_0,\tau_0,\kappa_0$.

\vv

We will now handle each type of cube in separate sections below.


\subsection{About the cubes of $\case(0,R)$}

Let us first introduce some notation and record a lemma that we will use to prove the packing condition of the $\case(0,R)$ cubes.

\begin{definition}
Given a (large) constant $M>0$ we say that two cubes $Q_1, Q_2 \in \DD_\mu$ are $M$-close if
\begin{itemize}
\item $M^{-1} \diam Q_1 \leq \diam Q_2 \leq M \diam Q_1$, and 
\item $\dist(Q_1, Q_2) \leq M (\diam Q_1 + \diam Q_2)$.
\end{itemize}
\end{definition}

The following is Lemma 3.27, page 59 in \cite{DS2}.
\begin{lemma}\label{lem:A-close-pack}
Let $\mathcal B \subset \DD$ be a family of cubes that satisfies a Carleson packing condition. Then the family
\begin{equation}\label{eq:A-close-pack}
\mathcal B_M:= \{Q \in \DD_\mu: Q \,\,\textup{is}\,\, M\textup{-close to some}\,\,Q'  \in \mathcal B\}
\end{equation}
also satisfies a Carleson packing condition with Carleson constant depending on $M$.
\end{lemma}

\begin{lemma}\label{lemcase0}
For $S\in\wt\tree(R)$, $R\in\ttt$, we have
\begin{equation}\label{eqpack0}
\sum_{\substack{Q\subset S:\\Q\in\case(0,R)}} \mu(Q)\leq C(\kappa_0,\theta_0)\,\mu(S).
\end{equation}
\end{lemma}

\begin{proof}
Denote by $\CC_1(R)$ the family of cubes $Q\in\wt\tree(R)$ for which there exists some cube $P\in\DD_\mu\setminus \wt\tree(R)$ such that $\kappa_0\,\ell(Q)\leq \ell(P)\leq\kappa_0^{-1}\ell(Q)$
and $\dist(P,Q)\leq\kappa_0^{-1}\ell(Q)$; and for $A>0$ we let 
$\CC_2(M, R)$ be the family of cubes $Q\in\wt\tree(R)$  such that
$$
\epsilon_M(Q):=\sup \bigl\{|A(y_1)-A(y_2)| \cdot |y_1 -y_2|^{-1}: y_1,y_2\in M B_Q,\, \delta_\Omega(y_i)\geq M^{-1} \,\ell(Q) \bigr\}>\frac{\theta_0}{\ell(Q)}.
$$

By \cite[(3.28) on page 60]{DS2} it follows that
$$\sum_{\substack{Q\subset S:\\Q\in\CC_1(R)}} \mu(Q)\leq C(\kappa_0)\,\mu(S).$$

Next, we turn our attention to the family $\CC_2(\kappa_0^{-1}, R)$. Let $p_Q \in \Omega$ be a corkscrew point for the ball $B_Q$ so that $\dist(p_Q. \d \Omega) \geq 2 c\, \ell(Q)$ and recall that a cone $\Gamma(x,\alpha)$ with vertex $x \in \d  \Omega$ and aperture $\alpha \geq 2$ is given by 
$$\Gamma(x,\alpha) = \big\{ y \in \Omega : |x-y|\leq \alpha\,  \dist(y, \d \Omega)\big\}.$$
Note that for every $x \in Q$ and $y \in B(p_Q, c\, \ell(Q))$, we have that 
$$
|x-y| \leq 2\ell(Q) \leq 2c^{-1}\dist(y, \d \Omega),
$$
which implies that $B(p_Q, c\, \ell(Q)) \subset \Gamma(x, 2c^{-1})$. 

Set $\Gamma(x):=\Gamma(x, 2c^{-1})$ and let the respective truncated cones from above and below be
\begin{align*}
\Gamma^t(x):= \big\{ y \in \Gamma(x) :\dist(y, \d \Omega) < t\big\} \quad\textup{and} \quad\Gamma_s(x):= \big\{ y \in \Gamma(x) :\dist(y, \d \Omega) \geq s\big\}.
\end{align*}

For a point $x \in \d \Omega$ and a dyadic cube $Q \in \DD_\mu$ so that  $x \in Q$, we define the Whitney-type region relative to $Q$ and $x$ to be
$$\Gamma(Q, x):= \Gamma^{\ell(Q)}(x) \cap \Gamma_{2^{-1}c\,\ell(Q)}(x),$$

Since $B(p_Q, c\, \ell(Q)) \subset \Gamma_{2c^{-1}}(x)$  and $\dist(y, \d \Omega) \geq c \ell(Q)$ for every $y \in B(p_Q, c\, \ell(Q))$, we have that 
\begin{equation}\label{eq:coneQ-size}
|\Gamma(Q, x)| \approx |B(p_Q, c\, \ell(Q))| \approx_{c} \ell(Q)^{n+1}.
\end{equation}
Moreover it is not hard to see that if  $S \in \DD_\mu$, $x\in S$ and $y \in \Omega$, then
\begin{equation}\label{eq:cone-bdd-overlap}
\sum_{Q \in \DD_\mu: Q\subset S} \chi_{\Gamma(Q, x)}(y) \lesssim_{c} \chi_{\Gamma^{\ell(S)}(x)}(y),
\end{equation}
and also that if $x\in Q$ and $y \in \Gamma_Q(x)$,
\begin{equation}\label{eq:coeff-discr-to-cont}
\epsilon_4(Q) \leq 
\sup_{\substack{z_1,z_2\in B(y, 12c^{-1}\delta_\Omega(y)) \cap \Omega\\ \delta_\Omega(z_k)\geq \frac{1}{4}\,\delta_\Omega(y)}}
\frac{|a_{ij}(z_1) - a_{ij}(z_2)|}{|z_1-z_2|}.
\end{equation}

Then, for $S \in \wt\tree(R)$, we have that
\begin{align*}
&\sum_{\substack{Q\subset S\\ Q \in \CC_2(4, R)}} \mu(Q) \leq \theta_0^{-1} \sum_{\substack{Q\subset S\\ Q \in \CC_2(4, R)}}  \epsilon_4(Q)\,\ell(Q)\,\mu(Q)\\
 & \lesssim \int_S \sum_{Q \in \DD_\mu: Q\subset S} \chi_{Q}(x)\, \epsilon_4(Q)\, \ell(Q) \, |\Gamma(Q, x)|^{-1}\int_{\Gamma(Q, x)} \,dy\,d\mu(x)\\
 &\stackrel[\eqref{eq:coneQ-size}]{\eqref{eq:coeff-discr-to-cont}}{\lesssim} \int_S \sum_{Q \in \DD_\mu: Q\subset S} \chi_{Q}(x)\!\int_{\Gamma(Q, x)} \,\,\sup_{\substack{z_1,z_2\in B(y,12c^{-1}\delta_\Omega(y)) \cap \Omega\\ \delta_\Omega(z_k)\geq \frac{1}{4}\,\delta_\Omega(y)}}\!
\frac{|a_{ij}(z_1) - a_{ij}(z_2)|}{|z_1-z_2|} \,\frac{dy}{\delta_\Omega(y)^{n}}\,d\mu(x)\\
&\stackrel{\eqref{eq:cone-bdd-overlap}}{\lesssim} \int_S \,\int_{\Gamma^{\ell(S)}(x)}\,\,\sup_{\substack{z_1,z_2\in B(y, 12c^{-1}\delta_\Omega(y)) \cap \Omega\\ \delta_\Omega(z_k)\geq \frac{1}{4}\,\delta_\Omega(y)}}
\frac{|a_{ij}(z_1) - a_{ij}(z_2)|}{|z_1-z_2|} \,\frac{dy}{\delta_\Omega(y)^{n}}\,d\mu(x)\\
&\lesssim \int_{4 c^{-1}B_S \cap \Omega} \,\,\sup_{\substack{z_1,z_2\in B(y,12c^{-1}\delta_\Omega(y)) \cap \Omega\\ \delta_\Omega(z_k)\geq \frac{1}{4}\,\delta_\Omega(y)}}
\frac{|a_{ij}(z_1) - a_{ij}(z_2)|}{|z_1-z_2|}\,dy \stackrel{\eqref{eqelliptic3}}{\lesssim} \ell(S)^n \approx \mu(S),
\end{align*}
where in the  third to the last inequality we used Fubini  and AD-regularity of $\mu$.

It is easy to check now that $\CC_2(\kappa_0^{-1}, R) \subset \CC_2(4, R)_{\kappa_0^{-1}}$, and thus, by Lemma \ref{lem:A-close-pack}, we infer that
$$\sum_{Q \in  \CC_2(\kappa_0^{-1}, R): Q \subset S} \mu(Q) \lesssim_{\theta_0, \kappa_0} \mu(S).$$ 
\end{proof}

\vv

\subsection{About the cubes of $\case(1,R)$}

 \begin{lemma}\label{lem:packing-case1}
Under the previous assumptions and if $U^{i^*}_Q=U^i_Q$ (i.e. the component of $U_{Q}$ containing $Y_{Q}^{*}$ is the same as $U^{i}_{Q}$), the following packing condition holds:
\begin{equation}\label{eqpack2}
\sum_{\substack{Q\subset S:\\Q\in\case(1,R)}} \mu(Q)\leq C(\ve_0,M,K_0)\,\mu(S),
\quad \mbox{ for all $R\in\ttt$ and $S\in\wt\tree(R)$,}
\end{equation}
\end{lemma}


To prove this lemma we need some additional auxiliary results.


\vv

\begin{lemma}\label{lem:grad-Moser}
Let us assume that  $u\in W^{1,2}(B(0,1))$ is a weak solution of $Lu=0$. If $a_{ij} \in \lip(B(0,1))$, for $1\leq i,j\leq n+1$, then for any $\vec \beta \in \R^{n+1}$ and any $\delta \in (0,1)$, it holds that
\begin{equation}\label{eq:grad-Moser}
\| \nabla u - \vec \beta \|_{C^\alpha(B(0, 1-\delta))}  \lesssim_{n, \delta} \| \nabla u - \vec \beta \|_{L^2(B(0,1))} +  \|(A -A(0))\vec \beta \|_{C^\alpha(B(0,1))}.
\end{equation}
\end{lemma}

\begin{proof}
We first define $ w(x)= u(x)-u(0)- \vec \beta\, x^\top$ and note that 
$$Lw= -\divv A \nabla u+ \divv A \vec \beta= \divv A \vec \beta= \divv (A -A(0))\vec \beta.$$
We apply Lemma \ref{lem:grad-u-regul} with $v=w$ and $F=(A -A(0))\vec \beta$ and obtain
\begin{equation*}
\| \nabla u - \vec \beta \|_{C^\alpha(B(0, 1-\delta))}  \lesssim \| \nabla u - \vec \beta \|_{L^2(B(0,1))} + \|(A -A(0))\vec \beta \|_{C^\alpha(B(0,1))},
\end{equation*}
since $\nabla w = \nabla u - \vec \beta $. 
\end{proof}

\vv

{
The rescaled and translated version of the lemma above reads as follows.
\begin{corollary}\label{cor:rescaledscauder}
Let us assume that  $u\in W^{1,2}(B(x_0,r))$ is a weak solution of $Lu=0$. If $a_{ij} \in \lip(B(x_0,r))$, for $1\leq i,j\leq n+1$, then for any $\vec \beta \in \R^{n+1}$ and any $\delta \in (0,1)$, if $B_{\delta r}:=B ( x_0 , \delta r )$ and $B_r:= B ( x_0 ,  r )$, it holds that
\begin{align*}
&\sup_{x \in B_{\delta r}  } | \nabla u - \vec \beta | + \sup_{x \neq y: \,x, y \,\in  B_{\delta r} } r^\alpha \, \frac{| \nabla u(x) - \nabla u(y)|}{|x-y|^\alpha}  \\
& \lesssim_{n, \delta} \left( \avint_{B_{ r} } | \nabla u -  \vec \beta|^2 \right)^{1/2}+ |\vec \beta| \left( \sup_{x \in B_{r} } |A(x) -A(x_0)|  +  \sup_{x \neq y: x,y \in  B_{r}  } r^\alpha \frac{|A(x) -A(y)|}{|x-y|^\alpha} \right).
\end{align*}
\end{corollary}
}
\vv

\begin{lemma}\label{lem2der}
If $u \in W^{1,2}(B(0,1)) $ is a weak solution of $Lu=0$ in $B(0,1)$, $A \in \lip(B(0,1))$, and $0<\delta<1$, then $u \in W^{2,2}(B(0,1-\delta))$ and 
$$\|\nabla^2u\|_{L^{2}(B(0,1-\delta))} \lesssim _\delta \|\nabla u\|_{L^{2}(B(0,1))} +\|u\|_{L^{2}(B(0,1))},$$
with the implicit constant depending also on the Lipschitz norm of the coefficients $a_{ij}$.
\end{lemma}

\begin{proof}
This is an easy consequence of Theorem 8.8 in \cite{GiTr}.
\end{proof}
\vv

\subsection{Proof of Lemma \ref{lem:packing-case1}}
{
We will follow the scheme of the proof of Lemma 5.8 in \cite{HLMN}. First, 
we denote $\wt\tree(S)=\{Q\in\wt\tree(R):Q\subset S\}$, and we define the ``sawtooth regions" associated to the family $\wt\tree(S)$ by 
$$\Omega^{j,*}_{\wt\tree(S)}:= \textup{int}\Big(  \mathop{\bigcup_{Q \in \wt\tree(S)}}_{\ell(Q) \leq \ve_0^{10} \ell(R)} U_Q^{j,*} \Big) \quad \textup{and}\quad \Omega^{*}_{\wt\tree(S)}:= \bigcup_{j} \Omega^{j,*}_{\wt\tree(S)}.$$
Fix a cube $Q \in \case(1,R)\cap\wt\tree(S) =:\case(1,R,S)$ and let $z\in B_X$ and  $X\in\wt U_Q^{i}$ supremise \eqref{eqcase1}. If $\wt B_z$ stands for the dilation of the ball $B_z$ of radius $(1-\ve_0^{4M/\alpha})\delta_\Omega(z)$, then 
\begin{align*}
 \sup_{x \in \wt B_{z} } |A(x) -A(z)|  +  \sup_{x \neq y: x,y \in \wt B_{z}  } r( \wt B_{z} )^\alpha \frac{|A(x) -A(y)|}{|x-y|^\alpha} \lesssim \theta_0,
 \end{align*}
where we used that $Q \not \in \case(0,R)$ and $ (1-\ve_0^{4M/\alpha})\delta_\Omega(z) \lesssim \ell(Q)$. The same estimate holds for $B_{Y_Q}$ instead of $B_z$. Then for $\vec \beta=\vec \beta(Q)$ to be chosen momentarily, by Corollary  \ref{cor:rescaledscauder}, we have that
\begin{align*}
\ve_0^{M} &\lesssim C(\ve_0)\,\left(\frac{1}{\ell(Q)^{n+1}} \int_{\wt B_z \cup \wt B_{Y_Q}} |\nabla u _*(y) - \vec \beta|^2\, dy\right)^{1/2} +C(\ve_0)\, \theta_0\, |\vec \beta| \\
 &\leq C(\ve_0)\,\left(\frac{1}{\ell(Q)^{n+1}} \int_{U^{{i},*}_Q} |\nabla u _*(y) - \vec \beta|^2\, dy\right)^{1/2} + C(\ve_0)\, \theta_0 \,|\vec \beta|.
\end{align*}
Therefore, if we  choose 
$$\vec \beta  =\frac{1}{ |U^{{i},*}_Q|} \int_{U^{{i},*}_Q}\nabla u _*,$$
by Poincar\' e inequality and the fact that $u(y)\approx \ell(Q)$ and $|\nabla u _*(y)| \lesssim 1$ for all $y\in U^{{i},*}_Q$,
\begin{align*}
\ve_0^{2M} \lesssim C'(\ve_0)\,\ell(Q)^{-n}  \int_{U^{{i},*}_Q} |\nabla^2 u _*(y)|^2\, u(y) \,dy + C'(\ve_0)\,\theta_0.
\end{align*}
If we further pick $\theta_0$ small enough depending on $\ve_0$,} it is enough to show
$$
\sum_{Q \in \case(1,R,S)} \mu(Q) \lesssim  \sum_{Q \in \case(1,R,S)}  \int_{U^{*}_Q} |\nabla^2 u _*(y)|^2\, u(y) \,dy \lesssim \mu(S).
$$
In fact, since the augmented Whitney regions $U^{*}_Q$ have bounded overlap, we have that
$$\sum_{Q \in \case(1,R,S)}   \int_{U^{*}_Q} |\nabla^2 u _*(y)|^2\, u(y) \,dy \lesssim \int_{\Omega^{*}_{\wt\tree(S)}} |\nabla^2 u _*(y)|^2\, u(y) \,dy,$$ 
and thus, it suffices to apply the following result:




\begin{proposition}\label{propo4.8}
Under the above assumptions and notation, we have
\begin{equation}\label{eqnoprov}
\int_{\Omega^{*}_{\wt\tree(S)}} |\nabla^2 u _*(y)|^2\, u(y) \,dy\lesssim \mu(S).
\end{equation}
\end{proposition}

This proposition is due to Hofmann, Martell and Toro and is already proven in \cite{HMT}. However we prefer to give here an alternative argument, for the reader's convenience.

\begin{proof}
We introduce a partition of unity $\{\eta_Q\}_Q$ on $\Omega$ so that the following hold:
\begin{itemize}
\item $\sum_{Q\in\DD_\mu} \eta_Q (y)=1$, where $y \in \Omega$,
\item $\supp \eta_Q \subset U_Q^{**}$, where $U_Q^{**}$ is the neighborhood of $U_Q^*$ given by
$$U_Q^{**}= \bigcup_{I\in \bigcup_ j \WW_Q^{j,*}} (1+3\tau)I$$
(compare to the definition of $U_Q^*$ in \rf{defuq2}).

\item $\eta_Q \in C^\infty_c(\R^{n+1})$, with $0 \leq \eta_Q \leq 1$, $\eta_Q \geq c$ on $U^*_Q$ and $\|\nabla \eta_Q\|_\infty \lesssim \ell(Q)^{-1}$.
\end{itemize}

Note now that if $U_Q^{***}= \bigcup_{I\in \bigcup_ j \WW_Q^{j,*}} (1+4\tau)I$ and $U_Q^{****}= \bigcup_{I\in \bigcup_ j \WW_Q^{j,*}} (1+5\tau)I$, then by Harnack's inequality we have that 
\begin{equation}\label{eq:bound-u-3*}
\max(u(y),u _*(y))\lesssim \delta(y)\approx \ell(Q), \,\,\textup{ for all}\,\, y \in U_Q^{****}.
\end{equation}
Furthermore, if $U_Q^{j,***}$ is a connected component of $U_Q^{***}$, by Lemma \ref{lem:grad-u-regul}, Caccioppoli's inequality, Harnack's inequality and \eqref{eq:bound-u-3*},
\begin{align*}
|\nabla u(y)|^2 & \lesssim \ell(Q)^{-n+1}\, \| u\|^2_{L^2(U_Q^{j,****})} \approx  \frac{u(y)}{\delta_\Omega(y)}\lesssim 1, \,\,\textup{ for all}\,\, y \in U_Q^{j,***}.
\end{align*}
Notice that, arguing as above, one can prove that $|\nabla u _*(y)| \lesssim 1$, for all $ y \in U_Q^{j,***}$. Therefore, 
\begin{equation}\label{eq:bound-gradu-i2*}
\max\Big(\sup_{U_Q^{***}}|\nabla u(y)|,\sup_{U_Q^{***}}|\nabla u _*(y)|\Big) \lesssim 1.
\end{equation}
 Moreover, since $L \d_k u _*= \divv [\partial_k A \nabla u _*]$ for any $k=1,2,\dots, n+1$, by Caccioppoli's inequality for inhomogeneous elliptic equations (see \cite[Theorem 4.4]{GiaMa}, for example) and \eqref{eq:bound-gradu-i2*}, 
\begin{align}\label{eq:grad2ubound}
 \int_{U_Q^{**}}| \nabla \partial_k u _*|^2 &\lesssim \frac{1}{\ell(Q)^2} \int_{U_Q^{***}}| \partial_k u _*|^2 +\int_{U_Q^{***}}| (\partial_k A) \nabla u _*|^2 \notag\\
 & \lesssim (\ell(Q)^{-2}+   \sup_{ U_Q^{***}} |\nabla A|^2 )\,\ell(Q)^{n+1} \lesssim \ell(Q)^{n-1},
 \end{align}
where in the last inequality we used that, by the Carleson condition \rf{eqelliptic3}, $ \sup_{ U_Q^{***}} |\nabla A| \lesssim \ell(Q)^{-1}$.

Given large number $N \gg \ve_0^{-10}$, let us set 
$$\Lambda(N)=\{ Q \in \DD_\mu: U_Q^{**} \cap \Omega^*_{\wt\tree(S)} \neq \varnothing \,\, \textup{and} \,\, \ell(Q) \geq N^{-1} \ell(S)\}.$$
By the properties of $\eta_Q$, the positivity of $u$ and the fact that $u _* \in W^{2,2}(U_Q^{**})$ for any $Q \in \wt\tree(S)$, it holds that
$$
\int_{\Omega^*_{\wt\tree(S)}} |\nabla^2 u _*(y)|^2\, u(y) \,dy \lesssim \lim_{N \to \infty} \sum_{Q \in \Lambda(N)} \int |\nabla^2 u _*(y)|^2\, u(y)\, \eta_Q(y) \,dy.
$$
Thus, it is enough to show that for each $1 \leq i \leq n+1$,
\begin{equation}\label{eq:main-packing-L}
\sum_{Q \in \Lambda(N)}  \int |\nabla \partial_i u _*|^2\, u\, \eta_Q \lesssim \mu(S),
\end{equation}
where the implicit constant is independent of $N$. Without loss of generality we assume that $i=1$. Then, by ellipticity,
\begin{align}\label{eqj1j2}
\Lambda^{-1} \,\sum_{Q \in \Lambda(N)}&\int |\nabla \partial_1 u _*|^2\, u\, \eta_Q \leq  \sum_{Q \in \Lambda(N)}\int (A\nabla \partial_1 u _*\cdot \nabla \partial_1 u _*)\, u \,\eta_Q\\
&=\sum_{Q \in \Lambda(N)}\int A\nabla \partial_1 u _*\cdot \nabla ( u \,\eta_Q\,\partial_1 u _*)- \sum_{Q \in \Lambda(N)}\int A\nabla \partial_1 u _*\cdot \nabla(u \,\eta_Q) \partial_1 u _*\nonumber \\
 &=: \sum_{Q \in \Lambda(N)}J_1(Q) - \sum_{Q \in \Lambda(N)}J_2(Q).\nonumber
\end{align}

Regarding the first sum on the right hand side above, we have
\begin{align*}
J_1(Q)&=\int \partial_1(A \nabla u _*) \cdot \nabla ( u \,\eta_Q\,\partial_1 u _*) - \int (\partial_1A) \nabla u _* \cdot \nabla ( u \,\eta_Q\,\partial_1 u _*)\\
&=0- \int (\partial_1A) \nabla u _* \cdot [\nabla u \,(\eta_Q\,\partial_1 u _*)+ \nabla \eta_Q\, ( u \,\partial_1 u _*)+\nabla \partial_1 u _*\, (\eta_Q\, u)]\\
&=:-\int (\partial_1A) \nabla u _* \cdot F_Q.
\end{align*}
Here we used that $u _*$ is $L$-harmonic in $U_Q^{**}$ and the fact that  $u \,\eta_Q\,\partial_1 u _* \in W^{1,2}_0(U^{**}_Q)$. Indeed, if we approximate $u \,\eta_Q\,\partial_1 u _*$ by a sequence of functions $\phi_j \in C^\infty_c(U^{**}_Q)$ in $W^{1,2}$-norm, then, since $A \nabla u _* \in W^{1,2}(U^{**}_Q)$, we have that
$$ \int \partial_1(A \nabla u _*) \cdot \nabla ( u \,\eta_Q\,\partial_1 u _*)=\lim_j\int \partial_1(A \nabla u _*) \cdot \nabla \phi_j = -\lim_j \int A \nabla u _* \cdot \nabla \partial_1\phi_j=0.$$


\vv

Notice that by \eqref{eq:bound-u-3*}, \eqref{eq:bound-gradu-i2*} and \eqref{eq:grad2ubound}, we obtain
\begin{align*}
\sum_{Q \in \Lambda(N)} |J_{1}(Q)| \leq \sum_{Q \in \Lambda(N)}\int_{U^{i,**}_Q}  |\nabla A|  |F_Q| \lesssim \sum_{Q \in \Lambda(N)} ( \ell(Q)\, \sup_{ U_Q^{**}} |\nabla A|)\, \mu(Q) \lesssim \mu(S),
\end{align*}
where in the last inequality we used \eqref{eqelliptic3}.

Let us turn our attention to the second sum on the right hand side of \rf{eqj1j2}.
We claim that 
$$2 J_2(Q) =\int A\nabla[(\partial_1 u_*)^2]\cdot \nabla\eta_Q\, u - \int A^*\nabla u\cdot \nabla\eta_Q\, (\partial_1 u _*)^2.$$
Indeed, the right hand side of the last equality is equal to
\begin{align*}
2 J_2(Q)&=\int A\nabla[(\partial_1 u _*)^2]\cdot \nabla(\eta_Q\, u)\\
&= \int A\nabla[(\partial_1 u _*)^2]\cdot \nabla\eta_Q\, u + \int A\nabla[(\partial_1 u _*)^2]\cdot \nabla u \,\eta_Q \\
&= \int A\nabla[(\partial_1 u _*)^2]\cdot \nabla\eta_Q\, u + \int A\nabla[(\partial_1 u _*)^2 \eta_Q]\cdot \nabla u -\int A\nabla \eta_Q \cdot \nabla u\,(\partial_1 u _*)^2\\
&= \int A\nabla[(\partial_1 u_*)^2]\cdot \nabla\eta_Q\, u - \int A^*\nabla u\cdot \nabla\eta_Q\, (\partial_1 u _*)^2.
\end{align*}
where in the last equality we used that $u$ is $L^*$-harmonic in $U_Q^{**}$ and the fact that  $\eta_Q\, (\partial_1 u _*)^2 \in W^{1,2}_0(U^{**}_Q)$.

Our last goal is to show  that
\begin{align}\label{eq:packingJ2}
\biggl|\sum_{Q \in \Lambda(N)} J_{2}(Q)\biggr| \lesssim \mu(S).
\end{align}
By our claim, we have
$$2\! \sum_{Q \in \Lambda(N)} \!\!J_2(Q) =\int\! A\nabla[(\partial_1 u_*)^2]\cdot \nabla\Bigl(\sum_{Q \in \Lambda(N)}\!\eta_Q\Bigr)\, u - \int A^*\nabla u\cdot \nabla\Bigl(\sum_{Q \in \Lambda(N)}\!\eta_Q\Bigr)\, (\partial_1 u_*)^2.$$
Set $\Lambda_1(N)=\Lambda_{11}(N) \cup \Lambda_{12}(N) $, where
\begin{align*}
\Lambda_{11}(N) := \{ Q \in \Lambda(N): U^{**}_Q & \cap U^{**}_{Q'} \neq \varnothing\\
&
\textup{for some}\,\,Q'\in\DD_\mu\setminus \wt\tree(S)
\,\textup{such that} \,\, \ell(Q') \geq N^{-1} \ell(S) \}.
\end{align*}
and 
$$\Lambda_{12}(N) := \{ Q \in \Lambda(N): U^{**}_Q \cap U^{**}_{Q'} \neq \varnothing\,\,\textup{for some}\,\,Q'\,\textup{such that} \,\, \ell(Q') < N^{-1} \ell(S) \}.$$
Note that
$$\biggl|\nabla\Bigl(\sum_{Q\in\Lambda(N)}\eta_Q\Bigr)\biggr| \leq \sum_{P\in\Lambda(N)}\chi_{U_P^{**}}\,
\biggl|\nabla\Bigl(\sum_{Q\in\Lambda(N)}\eta_Q\Bigr)\biggr|
= \sum_{P\in\Lambda_1(N)} \chi_{U_P^{**}}\biggl|\nabla\Bigl(\sum_{Q\in\Lambda(N)}\eta_Q\Bigr)\biggr|
$$
because
$$\chi_{U_P^{**}}\,
\nabla\Bigl(\sum_{Q\in\Lambda(N)}\eta_Q\Bigr) = 0
\quad \mbox{ if $P\in\Lambda(N)\setminus \Lambda_1(N)$.}$$
Thus, denoting $\psi=\sum_{Q\in\Lambda(N)}\eta_Q$ to shorten notation,
\begin{align*}
\biggl|\sum_{Q \in \Lambda(N)} \!\!J_{2}(Q)\biggr| &\lesssim  \!\!
\sum_{P\in\Lambda_1(N)}
\int_{U_P^{**}} \bigl|A\nabla[(\partial_1 u_*)^2]\cdot \nabla\psi \, u \bigr| + \!\!\sum_{P\in\Lambda_1(N)}\int_{U_P^{**}} \bigl|A\nabla u\cdot \nabla\psi\, (\partial_1 u_*)^2\bigr|\\
&=: \sum_{P\in\Lambda_1(N)} \bigl(J_{21}(P) + J_{22}(P)\bigr).
\end{align*}
We  split the last sum as
$$
\sum_{P \in \Lambda_{11}(N)} 
(J_{21}(P) + J_{22}(P))+ \sum_{P \in \Lambda_{12}(N)} (J_{21}(P) + J_{22}(P)).$$
Combining again \eqref{eq:bound-u-3*}, \eqref{eq:bound-gradu-i2*} and \eqref{eq:grad2ubound}, it is easy to see that
$$ J_{21}(P) + J_{22}(P)\lesssim \ell(P)^n \lesssim \mathcal{H}^n(U_P^{***} \cap \partial \Omega^*_{\wt\tree(S)} ),$$
where in the last inequality we used that $ \partial \Omega^*_{\wt\tree(S)}$ is n-AD-regular (see \cite{HMM2}). Therefore, by the bounded overlap property of the $U_P^{***}$'s, we infer that
$$\sum_{P \in \Lambda_{11}(N)} (J_{21}(P) + J_{22}(P)) \lesssim \mathcal{H}^n( \partial \Omega^*_{\wt\tree(S)} ) \lesssim \mu(S).$$
We conclude the proof of \eqref{eq:packingJ2} and thus, of \eqref{eq:main-packing-L} and the proposition, by noting that if $P \in \Lambda_{12}(N)$, then, by the definition of $U^{**}_P$, it is clear that $\ell(P) \approx N^{-1} \ell(S)$ and hence, an argument similar to (but simpler than) the one used above shows that
$$\sum_{P \in \Lambda_{12}(N)} (J_{21}(P) + J_{22}(P))  \lesssim \mu(S).$$
\end{proof}

\vv
Now the proof of Lemma \ref{lem:packing-case1} is concluded.

\vv

\subsection{About the cubes of $\case(2,R)$}

We have the following fundamental result:
\begin{lemma}\label{lemfund}
Assume $\ve_0>0$ and $\kappa_0>0$ are small enough (in particular, $\ve_0\leq K_0^{-6}$) and $M>1$ big enough. 
Given $R\in\ttt$, we have that
$$Q\in\case(2,R) \Rightarrow Q\in\whsa_{\ve_0}.$$
\end{lemma}

The proof of this result is almost the same as the one of Lemma 5.10 from \cite{HLMN}, where the same implication is proved in the case that $L$ is the Laplacian, by using the properties of the harmonic Green function and its connection with harmonic measure. The same estimates that are used in the proof in \cite{HLMN}
also hold for the $L$-harmonic Green function and the associated elliptic measure. Since the required modifications are very minor\footnote{Note also that the above Lemma \ref{lem333} replaces Lemma 3.35 from \cite{HLMN}.
Also, by Lemma \ref{lem:grad-u-regul}, $\nabla w$ is $\alpha$-H\"older continuous and it is defined pointwise.}, we refer the
reader to the proof of Lemma 5.10 in \cite[Section 5.3]{HLMN} (which in turn is inspired by some of the techniques
in \cite{LV}). 

\vv

It remains to deal with the cubes of  $\case(3,R)$. To this end we will use a version of the Alt-Caffarelli-Friedmazn formula 
valid for elliptic operators. In the next section we introduce this formula.

\vv


\section{The ACF formula for elliptic operators}

In this section we will state and prove a more precise version of Theorem \ref{teoACF-elliptic} and prove some related technical results. First we need some additional 
notation: given an open subset of the unit sphere, $\Sigma\subset \partial B(0,1)\equiv \partial B_1$, we denote
$$\lambda_\Sigma = \inf_{f\in W^{1,2}_0(\Sigma),f\not\equiv 0} \frac{\ds\int_\Sigma |\nabla_{\partial B_1}f|^2\,d\sigma}{
\ds\int_\Sigma |f|^2\,d\sigma},$$
where $\sigma$ denotes the surface measure on $\partial B_1$ and $\nabla_{\partial B_1} f$ is the gradient of $f$ on $\partial B_1$. In other words, $\lambda_\Sigma$ is the principal eigenvalue of the spherical Laplacian on $\Sigma$. We also denote by 
$\gamma_\Sigma$ the positive root of the equation
$$\lambda_\Sigma = \gamma_\Sigma\,(\gamma_\Sigma + n -1).$$
In fact, $\gamma_\Sigma$ is the so-called {\it characteristic constant} of $\Sigma$ (see for e.g. p.36 in \cite{PSU}).

If $A \in \Lip_{\loc}(\om)$ one can write 
\begin{align}\label{eq:Asplit-sym}
\divv A \nabla u = \divv A_s \nabla u + \vec b \cdot \nabla u,
\end{align}
in the weak sense, where 
$$A_s:= \frac{1}{2} (A + A^*)\quad \text{and}\quad \vec b := \Big( \frac{1}{2} \sum_{k=1}^{n+1} [ \d_k a_{ki} - \d_k a_{ik}] \Big)_{i=1}^{n+1}.$$
Notice also that $\divv \vec b = 0$ locally in the weak sense. Therefore, any divergence form uniformly elliptic the operator with  $\Lip_{\loc}(\om)$ coefficients can be written in the form given in \eqref{eqwtl}.


\vv

\begin{theorem}\label{teoACFprecise}
Under the assumptions of Theorem \ref{teoACF-elliptic} for $L$ as in \eqref{eq:Asplit-sym}, for a.e. $r\in(0,R)$ we have
\begin{equation}\label{eqprec1}
\frac{J'(x,r)}{J(x,r)}\geq \frac2r\bigl(\gamma_1 + \gamma_2  - 2 \bigr) - c\,\frac{(1+K_r(x))\,w(x,r)}{r},
\end{equation}
where $\gamma_i$ is the characteristic constant of the open subset  $\Sigma_i\subset\partial B_1$ given by  $$ \Sigma_i=\bigl\{r^{-1}(y-x): y\in\partial B(x,r),\,u_i(y)>0\bigr\}.$$
We have
$\gamma_1 + \gamma_2\geq2$,
and if one of the domains 
$\Sigma_{i}$
differs 
from a hemisphere by an area of size $\ve$, then 
\begin{equation}\label{eqprec2}
\gamma_1+\gamma_2-2 \geq c\,\ve^2.
\end{equation}
\end{theorem}
\vv

Note that the condition that one of the domains and $\Sigma_i$ digresses from a hemisphere by an area of size $\ve$ is
equivalent to the fact that one of the domains  
$\Sigma_{i,r}=\partial B(0,r) \cap \{u_i>0\}$
digresses from a hemisphere by an area of size $\ve\,r^n$.

The preceding  theorem will be also proved in the Appendix 
\ref{secappendix}. Next we need a technical estimate.

\vv

\begin{lemma}\label{lemcc**}
 Under the assumptions of Theorem \ref{teoACF-elliptic} for $L$ as in \eqref{eq:Asplit-sym}, we have
\begin{equation}
\label{lemcc**eq}
\int_{B(x,r)}\frac{|\nabla u_i(y)|^2}{|y-x|^{n-1}}\,dy\lesssim \|u_i\|_{\infty,B(x,2r)}^2.
\end{equation}
\end{lemma}

\begin{proof}
We will show that, for an arbitrary $0<\delta<R/2$, 
\begin{equation}\label{eqgreen7530}
\int_{B(x,r)\setminus B(x,\delta)}\frac{|\nabla u_i(y)|^2}{|y-x|^{n-1}}\,dy\leq c\,\|u_i\|_{\infty,B(x,2r)}^2,
\end{equation}
uniformly on $\delta$. To this end, let  $\vphi_\delta$ be a radial function such that $\chi_{A(x,\delta,2r)} \leq  \vphi\leq
\chi_{A(x,\frac12\delta,3r)}$, with $|\nabla\vphi|\leq C/\delta$, where $A(x,r,R)$ denotes the annulus of inner and outer radii $r$ and $R$ respectively. Consider the $L$-Green  function $G_{2B}$ of $B(x,2r)$ and note that
\begin{equation}\label{eqgreen753}
G_{2B}(z,y) \approx \frac1{|z-y|^{n-1}}\approx \frac1{|x-y|^{n-1}}\quad\mbox{ for all $z\in B(x,\delta/4)$ and $y\in A(x,\delta/2,r)$,}
\end{equation}
by Lemma \ref{lemgreen*}. Thus, using the ellipticity of $A$, 
\begin{align}\label{eqparts19}
\int_{B(x,r)\setminus B_\delta}\frac{|\nabla u_i(y)|^2}{|y-x|^{n-1}}\,dy &\lesssim
\int  A(y)\nabla u_i(y)\cdot\nabla u_i(y)\,G_{2B}(z,y)\,\vphi_\delta(y)\,dy\\
& = \int  A(y)\nabla u_i(y) \cdot \nabla \bigl(u_i\,G_{2B}(z,\cdot)\,\vphi_\delta\bigr)(y)\,dy\nonumber\\
&\quad - \int A(y)\nabla u_i(y) \cdot \nabla_yG_{2B}(z,y) \,u_i(y)\,\vphi_\delta(y)\,dy\nonumber\\
&\quad - \int A(y)\nabla u_i(y) \cdot \nabla \vphi_\delta(y)\,u_i(y)\,G_{2B}(z,y)\,dy.\nonumber
\end{align}
Notice now that 
$\bigl(u_i\,G_{2B}(z,\cdot)\,\vphi_\delta\bigr)$ is a non-negative function from $W_0^{1,2}(\R^{n+1})$, and since $u_i$ is $L$-subharmonic, we deduce that
$$\int  A(y)\nabla u_i(y) \cdot \nabla \bigl(u_i\,G_{2B}(z,\cdot)\,\vphi_\delta\bigr)(y)\,dy\leq0.$$
Hence, we only have to estimate the last two integral on the right hand side of  \rf{eqparts19}.

Concerning the second integral on the right hand side of  \rf{eqparts19}, we have
\begin{align*}
2\int \!A(y)\nabla u_i(y) \cdot \nabla_yG_{2B}(z,y) \,u_i(y)\,\vphi_\delta(y)\,dy &= 
\int \!A(y)\nabla \bigl(u_i^2\,\vphi_\delta)(y) \cdot \nabla_yG_{2B}(z,y) \,dy \\
& \; - \int \!A(y)\nabla \vphi_\delta(y) \!\cdot\! \nabla_yG_{2B}(z,y) \,u_i(y)^2\,dy\\
&=: I_1(z) - I_2(z).
\end{align*}
By \rf{eqgreen*23}, for a.e. $z\in B(x,\delta/4)$, we have
$$
I_1(z)= \int A(y)\nabla \bigl(u_i^2\,\vphi_\delta\bigr)(y) \cdot \nabla_yG_{2B}(z,y) \,dy = 
 - \int_{\partial B(x,2r)} u_i^2\,\vphi_\delta\,d\omega^z + u_i(z)^2\,\vphi_\delta(z).
 $$
Thus, since $u_{i}^{2}\varphi_{\delta}\geq 0$, 
$$|I_1(z)|\leq 0+\|u_i^2\,\vphi_\delta\|_{\infty,B(x,2r)} = \|u_i\|_{\infty,B(x,2r)}^2
\quad \mbox{ for a.e. $z\in B(x,\delta/4)$,}$$
taking into account that $u_i^2\,\vphi_\delta$ vanishes on $B(x,\delta/4)$. On the other hand, regarding
$I_2(z)$, using \rf{eqgpo0}, we derive
$$|I_2(z)| \lesssim \frac1\delta \,\|u_i\|_{\infty,B(x,2r)}^2\int_{B(x,\delta)} |\nabla_yG_{2B}(z,y)|\,dy
\lesssim \|u_i\|_{\infty,B(x,2r)}^2.$$

Finally we turn our attention to the last integral on the right hand side of \rf{eqparts19}.
Using \rf{eqgreen753} and Caccioppoli's inequality, we obtain
\begin{align*}
 \left|\int \!A(y)\nabla u_i(y) \cdot \nabla \vphi_\delta(y)\,u_i(y)\,G_{2B}(z,y)\,dy\right|
& \lesssim \frac1\delta \int_{A(x,\delta/2,\delta)} \!|\nabla u_i(y)|\,u_i(y)\,G_{2B}(z,y)\,dy\\
&\lesssim \frac1{\delta^n} \int_{B(x,\delta)} |\nabla u_i(y)|\,u_i(y)\,dy\\
&\lesssim \frac1{\delta^n} \,\|u_i\|_{L^2(B(x,\delta)}\,\|\nabla u_i\|_{L^2(B(x,\delta)}\\
&\lesssim \frac1{\delta^{n+1}} \,\|u_i\|_{L^2(B(x,\delta)}^2 \lesssim  \,\|u_i\|_{\infty,B(x,2r)}^2.
\end{align*}
Together with \rf{eqparts19} and the estimates obtained for $I_1(z)$ and $I_2(z)$, this proves
\rf{eqgreen7530}, as wished.
\end{proof}

\vv

The following lemma should be compared to Lemma 4.4 from \cite{ACS} and
Corollary 12.4 from \cite{CS}. This will play an essential role to prove some connectivity results in the following section, and this is the way that our ACF formula will be used in the proof of Theorem \ref{teo1}.

\vv

{
\begin{lemma}\label{lemcoro}
Consider the elliptic operator $Lu=\divv A\nabla u$. Let $B(x,R)\subset \R^{n+1}$, and let $u_1,u_2\in
W^{1,2}(B(x,R))\cap C(B(x,R))$ be nonnegative $L$-subharmonic functions. Suppose that 
$A_s(x)= Id$, $u_1(x)=u_2(x)=0$, and $u_1\cdot u_2\equiv 0$. 
Suppose also that there exists a modulus of continuity $w_0:[0,\infty] \to [0,\infty]$ satisfying \eqref{eq:Dini cond} such that
\begin{equation}\label{eqACF1'}
u_i(y)\leq C_1 \,w_0\!\left(\frac{|y-x|}r\right) \, \|u_i\|_{\infty,B(x,r)}\quad \mbox{ for $i=1,2$ and all $0<r\leq R$, $y\in B(x,r)$.}
\end{equation}
Let $J(x,r)$ be as in Theorem \ref{teoACF-elliptic} and denote $w_R=w(x,R)$.
Suppose that for $0<r_1<r_2<R$ 
there exist $M>1$ and $\eta \in (0,1)$ such that 
$$\HH^{n+1}\bigl(A(x,r,Mr)\cap \{u_1=0\}\cap \{u_2=0\}\bigr) \geq \eta\,\HH^{n+1}\bigl(A(x,r,Mr)) \quad\mbox{for all $r\in (r_1,r_2)$.}$$
Then, if $w_R$ is small enough (depending on $C_{w_0}$, $C_1$, $\eta$ and $M$),
\begin{equation}\label{eqpol99}
\frac{J(x,r_2)}{J(x,r_1)}\geq c\,\Bigl(\frac{r_2}{r_1}\Bigr)^\rho,
\end{equation}
for some positive constants $c$ and $\rho$ depending only on  $M$, $n$ and $\eta$.
\end{lemma}

\begin{proof}
Denote 
$$Z=\{u_1=0\}\cap \{u_2=0\},$$
and let $I_G$ be the subset of those $s\in[r,Mr]$ such that
$$\HH^n(Z\cap \partial B(x,s)) \geq \frac\eta2\,\HH^n( \partial B(x,s)).$$
It is easy to check that 
\begin{equation}\label{eq1gg}
\HH^1(I_G)\geq  \eta\,M\,C_n\,r.
\end{equation}
Indeed, we have
\begin{align*}
\int_{[r,Mr]\setminus  I_G} \HH^n(Z\cap \partial B(x,s))\,ds & \leq \frac\eta2
\int_{[r,Mr]\setminus  I_G} \HH^n(\partial B(x,s))\,ds \\
& \leq \frac\eta2\,\HH^{n+1}(A(x,r,Mr))\\
& \leq \frac12 \HH^{n+1}(Z\cap A(x,r,Mr)) \\
&= \frac12 \int_{[r,Mr]} \HH^n(Z\cap \partial B(x,s))\,ds.
\end{align*}
Hence,
\begin{align*}
\int_{I_G} \HH^n(Z\cap \partial B(x,s))\,ds & \geq \frac12 \int_{[r,Mr]}\HH^n(Z\cap \partial B(x,s))\,ds \\
&= \frac12 \,\HH^{n+1}(Z\cap A(x,r,Mr))\\
& \geq  \eta\,(M^{n+1} -1)\,\HH^{n+1}(B(0,1))\,r^{n+1}.
\end{align*}
Therefore, the inequality \rf{eq1gg} follows from  
\begin{align*}
\HH^{n+1}(B(0,1))\,\eta\,(M^{n+1} -1)\,r^{n+1} &\leq \int_{I_G} \HH^n(Z\cap \partial B(x,s))\,ds\\
&\leq \HH^{n}(\d B(0,1)) \,M^n\,r^n\,\HH^1(I_G),
\end{align*}
 where we used that $M^{n+1} \geq 2$.
 
By \rf{eqprec1} and \rf{eqprec2} we know that, for $s\in I_G$, 
$$\frac{J'(x,s)}{J(x,s)}\geq \frac{c_4\,\eta^2}s  - \frac{c_5\,w_R}{s},$$
while $x\not\in I_G$ we just only know that
$$\frac{J'(x,s)}{J(x,s)}\geq -\frac{c_5\,w_R}{s}.$$
Then, by integrating in $[r,Mr]$  we get
\begin{align*}
\log \frac{J(x,Mr)}{J(x,r)} &\geq \int_{I_G} \frac{c_4\,\eta^2}s \,ds - \int_{[r,Mr]} \frac{c_5\,w_R}s\,ds \geq \frac{c_4\,\eta^2}{Mr} \, \eta\,M\,C_n\,r - c_5\,\log M \, w_R\\
& = c_4\,C_n\,\eta^3 - c_5\,\log M \, w_R,
\end{align*}
where $c_5$ is dimensional constant multiple of $1+C_0\,C_{w_0}$. Hence, if we choose $w_R$ to be for example
$$
w_R \leq \frac{c_4\,C_n\,\eta^3}{2 c_5\,\log M},
$$
then
$$
\log \frac{J(x,Mr)}{J(x,r)} \geq \frac{c_4\,C_n\,\eta^3}{2}=:\wt C(\eta),
$$
which further implies that
$$J(x,Mr)\geq e^{\wt C(\eta)}\,J(x,r)=: (1+\gamma)\,J(x,r),$$
for some constant $\gamma>0$ which depends only on $n$ and $\eta$. Iterating this estimate we obtain
$$J(x,M^kr) \geq (1+\gamma)^k\,J(x,r),$$
which implies \rf{eqpol99}.
\end{proof}
}

\vvv


\section{Connectivity arguments for the cubes in $\case(3,R)$}\label{sec6*}

In this section we assume again that the matrix $A$ satisfies \rf{eqelliptic1}, \rf{eqelliptic2}, and \rf{eqelliptic3}, and we consider the families of cubes  $\case(i,R)$, $0\leq i\leq3$, introduced in 
Subsection \ref{sub95}.

Our main result in this section is the following.

\begin{lemma}\label{lem6.1}
Let $A_0>1$ and let $\kappa_0,\tau_0>0$ be the constants defined in Subsection \ref{subtype}. For any $\tau,t>0$ and $a_0>1$, if $\kappa_0$, $A_0$, and $\tau_0$ are chosen suitably (in particular, $\kappa_0$ 
and $\tau_0$ small enough and $A_0$ big enough), then
$$\case(3,R)\subset \wts(A_0,\alpha,p,t,\tau)\quad \mbox{ for all $R\in\wt\ttt$},$$
for some $\alpha>0$ possibly depending on $\tau$, and for some $p>0$ depending only on the constants in 
the corona decomposition in Proposition \ref{propo**}. 
Further, the family
$$\bigcup_{R\in\wt\ttt}\case(3,R)$$
satisfies the compatibility condition with constant $a_0$. 
\end{lemma}

To prove the lemma we consider $R\in\wt\ttt$ and $Q\in\case(3,R)$. For $\lambda>0$, we denote
$$V_1^\lambda(Q) = \bigl\{x\in A_0B_Q:u(x)>\lambda\,\ell(Q)\bigr\}$$
and
$$V_2^\lambda(Q) = \bigl\{x\in A_0B_Q:u_*(x)>\lambda\,\ell(Q)\bigr\}.$$
Then, given a constant $\tau_1>0$ to be fixed below (depending on $\tau$ and $A_0$), we let $U_i(Q)$ be the union of the connected components of $V_i^{\tau_1^{1/2}}(Q)$ which intersect
$V_i^{2\tau_1^{1/2}}(Q)\cap 20B_Q$. Also  we let $U_i'(Q)$ be the union of the connected components of $V_i^{\tau_1}(Q)$ which intersect
$V_i^{2\tau_1}(Q)\cap 20B_Q$. Note that, for $\tau_1\leq1/10$,
\begin{equation}\label{eqinc11}
V_i^{2\tau_1^{1/2}}(Q)\cap 20B_Q\subset U_i(Q)\subset V_i^{\tau_1^{1/2}}(Q)\subset A_0 B_Q
\end{equation}
and  
\begin{equation}\label{eqinc12}
V_i^{2\tau_1}(Q)\cap 20B_Q\subset U_i'(Q)\subset V_i^{\tau_1}(Q)\subset A_0 B_Q.
\end{equation}
Also, it is clear that
$$U_i(Q)\subset U_i'(Q),$$
and since $Q\in\case(3,R)$,
$$U_1'(Q)\cap U_2'(Q)=\varnothing.$$

The rest of this section is devoted to show that the sets $U_i(Q)$ and $U_i'(Q)$ satisfy the properties required in the definition of $\wts(A_0,\alpha,p,t,\tau)$, and to prove the compatibility condition with constant $a_0$ for the family $\bigcup_{R\in\wt\ttt}\case(3,R)$.

\vv


\subsection{Auxiliary lemmas}\label{sec31}

\begin{lemma}\label{lem6262}
Let $R\in\wt\ttt$ and $Q\in\wt\tree(R)\setminus\case(0,R)$. Let  $\kappa_0$ be small enough, $\lambda\in (\kappa_0,1)$, and
$C_3>0$ be large enough (independent of $\lambda$). 
For $i=1,2$ and for each $x\in \supp\mu\cap A_0B_Q$,\begin{equation*}\label{eq1bis}
B(x,C_3\lambda\ell(Q))\cap V_i^{\lambda}(Q)\neq\varnothing.
\end{equation*}
\end{lemma}

\begin{proof}
Consider the case $i=1$.
Let $Q'\in\DD_\mu$ be a cube which contains $x$ with $C_3'\lambda\ell(Q)\leq\ell(Q')<2 C_3'\lambda\ell(Q)$, for a suitable $C_3'$. Since $Q\not\in\case(0,R)$, then
$$\omega^{p_{R_1}}(3Q')\approx \frac{\mu(Q')}{\mu(R_1)}.$$
Then, for a suitable bump function $\vphi_{Q'}$, using Caccioppoli's inequality, we obtain
\begin{align*}
\frac{\mu(Q')}{\mu(R_1)} &\approx
\omega^{p_{R_1}}(3Q')\leq \left|\int A^*(y)\,\nabla_y G(p_{R_1},y)\,\nabla\vphi_{Q'}(y)\,dy\right|\\
 &\lesssim  \frac1{\ell(Q')\,\mu(R_1)}\int_{5B_{Q'}} |\nabla u(y)|\,dy\lesssim
\frac{\ell(Q')^{n-1}}{\mu(R_1)}\,\sup_{10B_{Q'}}u.
\end{align*}
Therefore, 
$$\sup_{10B_{Q'}}u\geq c \,\ell(Q')\geq c \,C_3'\lambda \ell(Q)>\lambda\ell(Q).$$
So if $C_3'$ is big enough, then $10B_{Q'}$ intersects $V_1^\lambda(Q)$. So choosing appropriately $C_3\geq C_3'$ $B(x,C_3\lambda \ell(Q))\supset 10B_{Q'}$, and thus $B(x,C_3\lambda \ell(Q))\cap V_{1}^{\lambda}(Q)\neq \varnothing$ as well. 
\end{proof}

\vv


We will also need the following auxiliary result, which is essentially proved in Lemmas 3.14 and 4.24 of \cite{HLMN}:

\begin{lemma}\label{lemaux}
For $R\in\wt\ttt$ and $Q\in\wt\tree(R)$ with $\ell(Q)\leq c\,\ell(R)$ for some $0<c<1$ depending just on the 
parameters in the corona decomposition, we have
\begin{equation}\label{eq50}
u(x)\lesssim \ell(Q)\quad \mbox{ for all $x\in B(z_Q,2\ell(Q))$,}
\end{equation}
where $z_Q$ is the center of $Q$.
Also,
\begin{equation}\label{eq51}
|\nabla u(x)|\lesssim 1\quad \mbox{ for all $x\in B(z_Q,2\ell(Q))$ such that $\dist(x,Q)\gtrsim \ell(Q).$}
\end{equation}
Further, there exists $C>0$ and some ball $\wt B_Q\subset C\,B_Q\cap\Omega$ such that $r(\wt B_Q)\approx\ell(Q)$ and
\begin{equation}\label{eq52}
u(x)\approx\ell(Q) \quad \mbox{ for all $x\in \wt B_Q$.}
\end{equation}
The analogous estimates hold for $u_*$. All implicit constants depend on $A$, $\delta$, and the AD-regularity constants. 
\end{lemma}
\vv

\begin{lemma}\label{lemaux2}
Let  $\kappa_0$ be small enough, and $C_4\kappa_0\leq \lambda\leq 1$, with $C_4$ a large universal constant.
For $R\in\wt\ttt$, $Q\in\wt\tree(R)\setminus\case(0,R)$, and $i=1,2$, we have
\begin{equation}\label{eq2}
\dist(V_i^{\lambda}(Q),E)\gtrsim \lambda\,\ell(Q),
\end{equation}
where $E=\supp\mu$.
Also,
\begin{equation}\label{eq5}
|\nabla u(x)|\lesssim 1\quad \mbox{ for all $x\in V_1^{\lambda}(Q)$,}
\end{equation}
and
\begin{equation}\label{eq6}
|\nabla u_*(x)|\lesssim 1\quad \mbox{ for all $x\in V_2^{\lambda}(Q)$.}
\end{equation}
\end{lemma}

\begin{proof}
 The first estimate is an immediate consequence of \rf{eq50} and the definition of $V_i^{\lambda}(Q)$.
Indeed, for any small $\tilde c>0$, since $Q\not\in \case(0,R)$, all cubes $P$ with $P\cap 2A_0B_Q\neq\varnothing$ and
$\ell(P)\geq \tilde c\lambda\,\ell(Q)$ satisfy $P\in\tree (R)$ (assuming $C_4=C_4(\tilde c)$ big enough and $\kappa_0\ll A_0^{-1}$). 
By taking a covering of $ E\cap 2A_0B_Q$ by balls $B(x_P,2\ell(P))$ associated to cubes $P\in\tree(R)$ with
$\ell(P)\approx\tilde c\lambda\,\ell(Q)$, for $\tilde c$ small enough we deduce that 
$$u(x)\leq\lambda\,\ell(Q) \quad\mbox{if $x\in 2A_0B_Q$ and $\dist(x,E)\leq c\,\tilde c\lambda\,\ell(Q)$},$$
which implies that $\dist(V_i^{\lambda}(Q),E)\geq c\,\tilde c\lambda\,\ell(Q)$.

On the other hand, \rf{eq5} and \rf{eq6} follow from \rf{eq51} and \rf{eq2}.
\end{proof}

\vv

In the next lemma we show the connection between the constants $\tau_1$ and $\tau_0$.

\begin{lemma}\label{lemaux33}
Suppose that $\tau_0$ is chosen small enough in Section \ref{subtype}. For $R\in\wt\ttt$ and $Q\in\wt\tree(R)
\cap\case(3,R)$,
$$U_1'(Q)\cap \{x:u_*(x)>\tau_1\,\ell(Q)\} = \varnothing \quad \text{ and }\quad U_1(Q)\cap \{x:u_*(x)>\tau_1^{1/2}\,\ell(Q)\} = \varnothing.$$
Analogously,
$$U_2'(Q)\cap \{x:u(x)>\tau_1\,\ell(Q)\} = \varnothing \quad \text{ and }\quad U_2(Q)\cap \{x:u(x)>\tau_1^{1/2}\,\ell(Q)\} = \varnothing.$$
\end{lemma}

\begin{proof}
We show that
$$U_1'(Q)\cap \{x:u_*(x)>\tau_1\,\ell(Q)\} = \varnothing.$$
The arguments for the other statements are analogous. Let $V$ be one of the connected components of $U_1'(Q)$
and suppose that there is some $x\in V$ such that $u_*(x)>\tau_1\,\ell(Q)$. By Lemma \ref{lemaux2},  
$\dist(V,E)\gtrsim \tau_1\,\ell(Q)$, and so there is a Harnack chain of balls $B_i\subset \Omega$ with radii comparable to $\tau_1\,\ell(Q)$ which connects $x$ and another point $x'\in V\cap 20B_Q$. By an easy covering argument we can assume this Harnack chain to have a bounded number of balls (depending on $A_0$ and $\tau_1$).
Hence we deduce that 
$$u_*(x')\geq c(A_0,\tau_1)\,\ell(Q)>0.$$
Clearly, we also have $u(x')\geq \tau_1\,\ell(Q)$.
Thus, if $\tau_0\leq \min(\tau_1,c(A_0,\tau_1))$, then 
$$20 B_Q \cap \{x:u_*(x)>\tau_0\,\ell(Q)\}\cap \{x:u(x)>\tau_0\,\ell(Q)\} \neq\varnothing,$$
and $Q\not\in\case(3,R)$.
\end{proof}
\vv

From now we assume that the constant $\tau_0$ is small enough so that the conclusion of the preceding lemma holds.
\vv


\subsection{The connectivity of $U_i(Q)$ and $U_i'(Q)$}

The proof of the following lemma is inspired by the arguments in 
Theorem 4.8 in \cite{ACS} and Theorem 12.1
in \cite{CS}.

\begin{lemma}\label{lemconn}
Assume $A_0=A_0(\tau_1)$ big enough and $\kappa_0=\kappa_0(\tau_1)$ small enough.
For $R\in\wt\ttt$ and $Q\in\case(3,R)$, the sets $U_i(Q)$ and $U_i'(Q)$ are connected.
\end{lemma}

\begin{proof}
We just prove that $U_1'(Q)$ is connected. The arguments for the other sets $U_1(Q)$, $U_2(Q)$ and $U_2'(Q)$ 
are analogous. We intend to apply Lemma \ref{lemcoro} to the function $u$ on different components.

Suppose that $U_1'(Q)$ has two different components $U_a$ and $U_b$. We want to check that this cannot happen if $A_0$ is big enough.
Let $x_Q$ be the center of $Q$, and so of $B_Q$ (which, in particular, belongs to $\supp\mu$). Denote
$$J_{Q}(r) = J(x_{Q},r),$$
with $J(x_{Q},r)$ defined in \rf{eqACF2}.
We will prove the following:

\vv
$\bullet$ Claim 1: $J_{Q}(r)\lesssim 1$ for $\tau_1 \ell(Q)\leq r\leq A_0 \,\ell(Q)$.

\vv
$\bullet$ Claim 2: $J_{Q}(40\ell(Q)) \gtrsim \tau_1^{4(n+1)}$.

\vv  
$\bullet$ Claim 3: For any $x\in 10B_{Q}$ (in particular for $x=x_Q$) and $C_5\tau_1\ell(Q)\leq r\leq A_0\,\ell(Q)$ (where $C_5>1$ is some fixed constant),
$$\HH^{n+1}(B(x,r)\setminus (U_a\cup U_b))\gtrsim r^{n+1}.$$
\vv
{
Assume the three claims hold. From the fact that $Q\not\in\case(0,R)$, we know that
\begin{align*}
\sup \bigl\{|A(y_1)-A(y_2)|: y_1,y_2\in \kappa_0^{-1} B_Q,\, \delta_\Omega(y_i)\geq \kappa_0\,\ell(Q)
\bigr\} &\leq\theta_0,\\
\sup \bigl\{|y-x_Q| \, |\nabla A(y)|: y\in \kappa_0^{-1} B_Q,\, \delta_\Omega(y)\geq \kappa_0\,\ell(Q)
\bigr\} &\leq \frac{\theta_0}{\kappa_0},
\end{align*}
where $\theta_0 \ll \kappa_0$ will be chosen below. In particular, if $\kappa_0$ is chosen small enough (independent of $\theta_0$),
the above condition implies that 
\begin{align*}
\sup \bigl\{|A_s(y_1)-A_s(y_2)|: y_1,y_2\in U_a\cup U_b
\bigr\}
&\leq  \theta_0,\\
\sup \bigl\{|y-x_Q| \, |\vec b(y)|: y\in U_a\cup U_b
\bigr\} &\leq  \frac{ \theta_0}{\kappa_0},
\end{align*}
where $A_s$ and $\vec b$ are like in \rf{eq:Asplit-sym}. Fix $y_0\in U_a\cup U_b$ and 
suppose that
$$A_s(y_0)=Id.$$

Consider the matrix 
$$\wt A(y) = \left\{\begin{array}{ll} A^*(y)& \text{if $y\in U_a\cup U_b$,}\\
A_s^*(y_0)& \text{if $y\not\in U_a\cup U_b$.}
\end{array}\right.$$
In particular, $\wt A_s(x_Q)=Id$ and if we denote $\wt L u=\divv\wt A \nabla u$, then it follows that 
$$u_a:=(u-\tau_1\,\ell(Q))^+ \chi_{U_a}\quad  \textup{and} \quad u_b:=(u-\tau_1\,\ell(Q))^+ \chi_{U_b}$$
are $\wt L$-subharmonic in $\R^{n+1}$. Moreover, it is easy to see that they also satisfy \eqref{eqACF1'} since, by Lemma \ref{lemaux2}, for every $x\in U_a \cup U_b = U'_1(Q) \subset V^{\tau_1}_1(Q)$ it holds $|x-x_Q| \approx_{\tau_1} \ell(Q)$ while they vanish outside $U_a $ and $U_b$ respectively. Since they clearly have disjoint supports and vanish at $x_Q$, we can now apply Lemma \ref{lemcoro}  to $u_a$ and $u_b$ for $\wt L$, choosing $2(1+\kappa_0^{-1}) \theta_0=w_R$ (with $w_R$ given by Lemma \ref{lemcoro} and $M$ depending on the parameters of the corona decomposition) and obtain
$$\tau_1^{4(n+1)} \,\ell(Q)^{-\rho}\lesssim (40\ell(Q))^{-\rho}\,J_{Q}(40\ell(Q))  \lesssim (A_0\,\ell(Q))^{-\rho}\,J_{Q}(A_0\,\ell(Q))
\lesssim (A_0\,\ell(Q))^{-\rho}.$$
Thus,
$$A_0\lesssim  \tau_1^{-4(n+1)/\rho},$$
which fails if $A_0$ is taken big enough (depending on $\tau_1$).

If $A_s(y_0)\neq Id$, we use Corollary \ref{cor:A(x0)=id} and the comments following its statement, and arguing as before we obtain connectivity of $U_i$ and $U'_i$ in full generality. The details are left for the reader.}
Concerning the claims above, note that Claim 1 is an immediate consequence of Lemma \ref{lemcc**} and the estimate \rf{eq50}. 
Indeed, for $r\in [\kappa_{0}\ell(Q),A_{0}\ell(Q)]$, since $Q\not\in \case(0,R)$, if $x\in B(x_{Q},2r)$, then $x\in B_{T}$ for some $T\in  \wt\tree(R)$ and $\ell(T)\approx r$ if we pick $\kappa_0$ small enough. Hence, \rf{eq50} implies
\[
|u(x)|\lec \ell(T)\approx r\]
and similarly for $u_{*}$. This and \eqref{lemcc**eq} imply
\[
J_{Q}(r)
\lesssim r^{-4} \|u\|_{\infty,B(x_Q,2r)}^2\|u_{*}\|_{\infty,B(x_Q,2r)}^2
\lesssim 1.\]

Claim 3
follows from \rf{eq52}. Indeed, this estimate implies that for 
$C_4\tau_1\ell(Q)\leq r\leq A_0\,\ell(Q)$, $B(x_Q,r)$ contains another ball $B(y,cr)$ such that $u_*(z)\geq c'\, r \geq\tau_1\ell(Q)$ for all $z \in B(y,cr)$,
and so 
$B(y,cr)\subset V_2^{\tau_1}(Q)$, which by assumption is disjoint from $U_a$ and $U_b$ by Lemma \ref{lemaux33}.
Hence,
$$\HH^{n+1}(B(x_Q,r)\setminus (U_a\cup U_b))\geq \HH^{n+1}(B(y,cr))\approx r^{n+1}$$
with constant independent of $\tau_{1}$
To prove Claim 2, consider $z_0\in U_a\cap 20B_Q$ such that $u(z_0)> 2\tau_1 \ell(Q)$ (the existence of this point is guaranteed by
the definition of $U_a$). Since $|\nabla u|\lesssim 1$ on $U_a$, there exists some ball $B(z_0,c\tau_1\ell(Q))$ such that
$u(y)>3\tau_1\,\ell(Q)/2$ for all $y \in B(z_0,c\tau_1\ell(Q))$. We have
\begin{equation}\label{eqh2}
\int_{U_a\cap B(x_Q,40\ell(Q))}|\nabla u|\,dy\gtrsim\tau_1^{n+1}\ell(Q)^{n+1}.
\end{equation}
To check this, let $B$ be the largest ball centered at $z_{0}$ not intersecting $E$ and let $y_{0}\in E\cap \partial B$. Then 
\[
B(z_0,c\tau_1\ell(Q)) \subset B\subset B(x_{Q},40\ell(Q))\]
and $u(y_{0})=0$. 
 Then, by considering the convex hull $H\subset B$ of $B(z_0,c\tau_1\ell(Q))$ and
$y_0$ and integrating by spherical coordinates (with the origin in $y_0$), one can check that 
$$\int_{H\cap U_a} |\nabla u|\,dy \gtrsim \tau_1^{n+1}\,\ell(Q)^{n+1}.$$
We leave the details for the reader.

From \rf{eqh2}, by Cauchy-Schwarz, taking also into account that $H\subset B(x_Q,40\ell(Q))$, we get
\begin{align*}
\tau_1^{n+1}&\ell(Q)^{n+1} \lesssim\int_{U_a\cap B(x_Q,40\ell(Q))}|\nabla u|\,dy \\& \lesssim 
\left(\int_{U_a\cap B(x_Q,40\ell(Q))} \frac{|\grad u(y)|^{2}}{|y-x_Q|^{n-1}}dy\right)^{1/2}\,
\left(\int_{U_a\cap B(x_Q,40\ell(Q))} {|y-x_Q|^{n-1}}dy\right)^{1/2}\\
&
\lesssim \ell(Q)^n\,\left(\int_{ U_a \cap B(x_Q,40\ell(Q))} 
\frac{|\grad u(y)|^{2}}{|y-x_Q|^{n-1}}dy\right)^{1/2}.
\end{align*}
An analogous estimate holds replacing $U_a$ by $U_b$, and then it follows that $J(r)\gtrsim\tau_1^{4(n+1)}$.
\end{proof}
\vv
%
%
%

Up to now we have shown that the sets $U_i(Q)$ and $U_i'(Q)$ are connected and satisfy the property
(2) in the definition of $\wts$ (see Definition \ref{WTS}). Further, from \rf{eq2} and the definition of $U_i(Q)$ and $U_i'(Q)$,  it follows that $\dist(U_i(Q),E)\gtrsim
\tau_1^{1/2}\ell(Q)$, and since $U_i(Q)$ is open and connected, the property (4) holds with $\alpha=
C\,\tau_1^{1/2}$. On the other hand, the property (3) also follows easily from Lemma \ref{lem6262}.
Indeed, for each $x\in E\cap 10B_Q$ and $4\tau_1^{1/2}\leq\lambda<1$,
$$B(x,C_3\lambda\ell(Q))\cap V_i^{\lambda}(Q)\neq\varnothing,$$
and then using also Lemma \ref{lemaux}, we infer that that $B(x,2C_3\lambda\ell(Q))\cap V_i^{\lambda/2}(Q)$ contains some ball $\wh B$ with
radius $c\lambda\ell(Q)$. In particular, $\wh B\subset U_i(Q)$. 

Thus, it remains to show property (1) of WTS holds for the sets $U_{i}(Q)$.

\vv

\subsection{The property (1) for $\wts$}
Next we deal with the property (1) in the definition of $\wts$.
\begin{lemma}\label{lemfi}
Let $\tau>0$. If $z\in 10B_Q$ and $\dist(x,E)\geq \tau\,\ell(Q)$, then $z\in V_1^{2\tau_{1}^{1/2}}(Q)\cup V_2^{2\tau_{1}^{1/2}}(Q)$, assuming $\tau_1$ small
enough depending on $\tau$. In particular, $z\in U_1(Q)\cup U_2(Q)$ by \rf{eqinc11}.
\end{lemma}

The proof will be another consequence of the ACF formula.

\begin{proof}
Let $z_0\in 10B_Q$ be such that $R\equiv \dist(z_0,E)\geq \tau\,\ell(Q)$. Denote
$$B_0 = B(z_0,R-\tau\ell(Q)).$$
If $B_0\cap V_i^\tau(Q)\neq\varnothing$, then $z_0$ can be connected to $V_i^\tau(Q)$ by a Harnack chain of $N$ balls $B_j\subset B(z_0,R)\setminus E$, $j=1,\ldots,N$,
with radii $\tau\ell(Q)/10$, where $N\approx \tau^{-1}$. This implies that, for some constant  $c_6(\tau)>0$, 
$$u(z_0)\geq c_6(\tau)\,\ell(Q)>2\,\tau_1^{1/2}\ell(Q),$$
if $\tau_1$ is chosen small enough (depending on $\tau$), and so $z_0\in V_i^{2\tau_1^{1/2}}(Q)\cap 10B_Q\subset U_i(Q)$.
Thus in this case the statement in the lemma holds.

{
Assume now that $B_0\cap (V_1^\tau(Q)\cup V_2^\tau(Q))=\varnothing$. We will show that in this case $R\leq C\,\tau\ell(Q)$
for some suitable big constant $C$, which will complete the proof of the lemma (after renaming $\tau$). 

To this end, let $x_0\in E\cap 20B_Q$ be such that
$|x_0-z_0|=\dist(z_0,E)$ and for fixed $y_0\in V_1^\tau(Q)$, denote by $G_{B_0}(\cdot,z_0)$ the Green function associated with $A_s^*(y_0)$ in $B_0$ with pole at $z_0$. Consider now the matrix 
$$\wt A(y) = \left\{\begin{array}{ll} A^*(y)& \text{if $y\in V_1^\tau(Q)$,}\\
A_s^*(y_0)& \text{if $y\not\in V_1^\tau(Q)$}
\end{array}\right.$$
and denote $\wt L u=\divv\wt A \nabla u$. It follows that $(u-\tau\,\ell(Q))^+$ is $\wt L$-subharmonic.

Let $v$ be the function defined by:
$$v(y)= \left\{\begin{array}{ll} 
(u(y)-\tau\ell(Q))^+ & \mbox{if $y\in V_1^\tau(Q)$,}\\
r(B_0)\,G_{B_0}(y,z_0) & \mbox{if $y\in B_0$,}\\
0 & \mbox{otherwise.}
\end{array}
\right.
$$
It is clear that $v$ is $L^*$-subharmonic, non-negative, and continuous in $B(x_0,R/2)$.

Suppose now that $A_s(y_0)=Id$. We intend to apply the ACF formula with
$$J(r) = \left(\frac{1}{r^{2}} \int_{B(x_0,r)\cap V_1^\tau(Q)} \frac{|\grad v(y)|^{2}}{|y-x_0|^{n-1}}\,dy\right)\cdot \left(\frac{1}{r^{2}} \int_{B(x_0,r)\cap B_0} \frac{|\grad v(y)|^{2}}{|y-x_0|^{n-1}}\,dy\right),
$$
for $C\tau\ell(Q)\leq r\leq R/2$. 

 We have:

\vv
$\bullet$ Claim 1: $J(R/2)\lesssim 1$.

\vv
$\bullet$ Claim 2: $J(C\tau \ell(Q)) \gtrsim 1$, for some $C$ big enough, where $C_4$ appears in Lemma \ref{lemaux2}.


\vv  
$\bullet$ Claim 3: For $C_4\tau\ell(Q)\leq r\leq R/2$,
$$\HH^{n+1}(B(x,r)\setminus (V_1^\tau(Q)\cup B_0))\gtrsim r^{n+1}.$$
\vv

The proof of these claims is very similar to the analogous claims in the proof of Lemma \ref{lemconn}.
Indeed, for the first one we use Lemma \ref{lemcc**} and that $|\nabla v|\lesssim1$ in $B(x_0,R/2)$.
Claim 3 has exactly the same proof as the homonymous claim in the proof of Lemma \ref{lemconn}.
Concerning Claim 2, note that now we want the constant implicit in the estimate to be independent of $\tau$.

On the one hand, since $v=r(B_0)\,G_{B_0}(\cdot,z_0)$ in $B_0$, it follows that $|\nabla v|\approx 1$ in
$B(x_0,r)\cap B_0$ (since $G_{B_0}$ is associated to $\Delta$ in the ball $B_0$), and thus
$$
\frac{1}{r^{2}} \int_{B(x_0,r)\cap B_0} \frac{|\grad v(y)|^{2}}{|y-x_0|^{n-1}}\,dy\approx 1$$
for $C\tau \ell(Q)\leq r\leq R/2$.
To prove that
\begin{equation}\label{eqfh67}
\frac{1}{r^{2}} \int_{B(x_0,r)\cap V_1^\tau(Q)} \frac{|\grad v(y)|^{2}}{|y-x_0|^{n-1}}\,dy\gtrsim 1
\end{equation}
for $C\tau \ell(Q)\leq r\leq R/2$ we take into account that, by \rf{eq52}, there exists some ball
$B_r:= B(y_r,c\,r)\subset U_1(Q)\cap B(x_0,r)$ such that $u(y)\approx r$ in $B_r$.
Then the same argument described near \rf{eqh2} (note that now there is no dependence on $\tau$ 
in any estimate)
shows that
$$\int_{V_1^\tau(Q)\cap B(x_0,r)}|\nabla u|\,dy\gtrsim r^{n+1},$$
and then Cauchy-Schwarz yields \rf{eqfh67}, and this proves Claim 2.

Since $Q\not\in\case(0,R)$, we know that
that \begin{align*}
\sup \bigl\{|A_s(y_1)-A_s(y_2)|:y_1,y_2\in V_1^\tau(Q)
\bigr\}
&\leq 2 \theta_0\\
\sup \bigl\{|y-x_Q| \, |\vec b(y)|: y\in V_1^\tau(Q)
\bigr\} &\leq  \frac{2 \theta_0}{\kappa_0},
\end{align*}
where $\kappa_0$ is chosen small enough (independently of $\theta_0$).
Finally, choosing $\theta_0=w_R$,  we apply Lemma \ref{lemcoro} to deduce that
$$(C\,\tau\ell(Q))^{-\rho}\lesssim (C\,\tau\ell(Q))^{-\rho}\,J(C\,\tau\ell(Q))  \lesssim (\tfrac12 R)^{-\rho}\,J(\tfrac12 R)
\lesssim (\tfrac12 R)^{-\rho},$$
which implies that $R\leq C'\tau\ell(Q)$, as wished.

If $A(y_0)\neq Id$, then we just use Corollary \ref{cor:A(x0)=id} and argue analogously. The details are left for the reader again.}
\end{proof}
\vv


\subsection{The compatibility condition}

\begin{lemma}\label{lemcompat}
The family
$$\FF:= \bigcup_{R\in\wt\ttt}\case(3,R)$$
satisfies the compatibility condition with constant $a_0$. 
\end{lemma}

\begin{proof}
We have to show that for all $P,Q\in \FF$ such that $2^{-a_{0}}\ell(Q)\leq  \ell(P)\leq \ell(Q)$ it holds  $U_{i}(P)\cap 10B_{Q}\subset U_{i}'(Q)$.
To this end, suppose first that $P,Q\in\case(3,R)$ for some common $R\in\wt\ttt$.
From the inclusions in  \rf{eqinc11}, 
using that $2\tau_1\,\ell(Q)\geq \tau_1^{1/2}\,\ell(P)$ (for $\tau_1$ small enough),
we get
\begin{align*}
U_1(P)\cap 10B_Q &\subset V_1^{\tau_1^{1/2}}(P) \cap 10B_Q   \subset \{x\in 10B_Q: u(x)> \tau_1^{1/2}\,\ell(P)\}\\
&
\subset \{x\in 10B_Q: u(x)> 2\tau_1\,\ell(Q)\}\subset U_1'(Q),
\end{align*}
as wished.

If $P$ and $Q$ belong to different trees $\wt\tree(R_P)$ and $\wt\tree(R_Q)$, with $R_P,R_Q\in\wt\ttt$, then
$A_0B_P\cap A_0B_Q=\varnothing$, assuming $\kappa_0$ small enough (depending on $A_0$), and thus the 
inclusion $U_{i}(P)\cap 10B_{Q}\subset U_{i}'(Q)$ is trivial.
\end{proof}
\vv

This completes the proof of Lemma \ref{lem6.1}.
\vvv

\section{$\whsa$ versus $\batpp$}

Recall, given $R\in \wt\ttt$, if $Q\in\wt\tree(R)\cap \case(2,R)$, then $Q\in\whsa_{\ve_0}$. Also, for a ball $B$ centered on $E$, we denote
\[
b\beta_{\infty}(B)=r(B)^{-1}\inf_{P} \sup_{x\in B\cap E} \dist(x,P)\]
where the infimum is over all $n$-planes $P$. 

The following also holds.


\begin{lemma}\label{lem7.1}
Let $R\in \wt\ttt$ and $Q\in\wt\tree(R)\cap\case(2,R)$. 
Suppose that all cubes  $P\in\DD_\mu$ such that $P\subset 4B_Q$, $\ve_0\,\ell(Q)\leq \ell(P)\leq \ve_0^{1/2}\ell(Q)$
belong to $\wt\tree(R)\cap\bigl[\case(2,R)\cup\case(3,R)\bigr]$. Given $\ve>0$, if $\ve_0$, $\tau_0$, $\theta_0$ and $\kappa_0$ are small enough, then
either
\begin{itemize}
\item[(a)] $b\beta_\infty(10B_Q)\leq \ve$, or
\item[(b)] $Q\in\batpp_{\ve}$.
\end{itemize}
\end{lemma}

\begin{rem}\label{batppremark} The two alternatives above can be written together as $Q\in\batpp_{\ve}$, since in the case (a) above we may think that
the two parallel planes from the condition $\batpp_{\ve}$ coincide. However, for technical reasons we prefer the writing above.

Recall that in Lemma \ref{lemfund} we showed that the cubes $Q$ from $\tree(R)\cap\case(2,R)$
belong to $\whsa_{\ve_0}$. For the application of Proposition \ref{propowtf} we would like them to belong to
$\batpp_{\ve}$. Lemma \ref{lem7.1} takes care of this issue.
\end{rem}

\begin{proof}
Assume that $\ve_0\ll\ve^2$ and
suppose that $b\beta_\infty(10B_Q)> \ve$. Denote by $H(Q)$ the half-space in the definition of $\whsa_{\ve_0}$, and $L(Q)=\partial 
H(Q)$, so that
\begin{itemize}
\item $\dist(z,\supp\mu)\leq \ve_0\,\ell(Q)$ for every $z\in L(Q)\cap B(x_Q,\ve_0^{-2}\ell(Q))$,
\item $\dist(Q,L(Q))\leq K_0^{3/2}\,\ell(Q)$, and
\item $H(Q)\cap B(x_Q,\ve_0^{-2}\ell(Q))\cap \supp\mu =\varnothing$.
\end{itemize}
Recall that $K_0$ is some big constant independent of the choice of $A_0$, $\tau_0$, $\tau$, and other parameters.

Denote by $I$ the family of cubes $P\in\DD_\mu$ such that 
\begin{itemize}
\item $P\cap B\bigl(x_Q,K_0^2\ell(Q)\bigr)\neq\varnothing$, 
\item $\ve^{3/2}\,\ell(Q)\leq \ell(P)\leq 2\ve^{3/2}\ell(Q)$, and
\item $\dist(P,L(Q))\geq \frac12\ve \ell(Q)$.
\end{itemize}
Arguing as in \cite[Section 6 Claim 6.6]{HLMN}, it follows that
if every $P\in I$ belongs to $\whsa_{\ve^2}$ (and thus to $\whsa_{\ve_0}$), then $Q\in\batpp_\ve$.

Since $I\subset \case(2,R)\cup\case(3,R)$ and $\case(2,R)\subset \whsa_{\ve_0}$, to prove the lemma it is enough to show that there are no cubes from $I$ which belong to $\case(3,R)$.
So we suppose that there exists $P\in I \cap \case(3,R)$ and we will get a contradiction. Notice that such a cube $P$ must be contained in $B\bigl(x_Q,2K_0^2\ell(Q)\bigr)\setminus H(Q)$.

Our arguments will be based on the use of our variant of the ACF formula. We denote 
$$r_0 = 2\bigl(\dist(P,L(Q)) + \ell(P)\bigr)
\lec K_{0}^{2}\ell(Q).
$$

\vv
\noi {\bf Claim.} Given some fixed constant $M$ with $1\ll M\ll\ve_0^{-1}$ to be chosen below,
if $\ve_0$, $\tau_0$, $\theta_0$ and $\kappa_0$ are small enough, then
\begin{equation}\label{eqclam10}
\HH^{n+1}\bigl(\{x:u(x)> c\,\ell(P)\} \cap H(Q)^c\cap B(x_P,r)\bigr) \geq c\,r^{n+1} \quad \mbox{ if\, $r_0 \leq r\leq M\,
\ell(Q)$,}
\end{equation}
where $c>0$ is some fixed constant only depending on the parameters of the corona decomposition in 
Proposition \ref{propo**}.

To prove the claim, consider a cube $P'\in\DD_\mu$ with $P\subset P'$ so that $\ell(P')\approx
\dist(P,L(Q))$ and $B(x_{P'},2C_2\ell(P'))\subset H(Q)^c$, with $C_2$ as in Lemma \ref{lem:failure}.
This lemma ensures the existence of a point $\wt Y_{P'}\in B(x_{P'},C_2\ell(P'))$ such that
$u(\wt Y_{P'})>c_7\ell(P')$, for some fixed $c_7>0$. Consider the connected component of the
set 
$$\bigl\{x\in B(x_P,M\ell(Q)):u(x)>\tfrac12c_7\ell(P')\bigr\}$$ { that contains $\wt Y_{P'}$} and call it $V(P')$.  From the definition of $\whsa_{\ve_0}$ and the H\"older continuity of $u$, we have
$$u(x) \leq c\,\ve_0^\alpha\,\ell(Q)\ll \ell(P)\leq \ell(P')\quad \mbox{ for all $x\in  B(x_P,M\ell(Q))\cap L(Q)$}$$

(recall that $\alpha$ is the exponent regarding the
H\"older continuity of $u$ and $u^*$). Hence,
$$V(P')\subset H(Q)^c.$$

To apply the ACF formula we consider the 
functions
$$u_1= (u-\tfrac12c_7\ell(P'))\,\chi_{V(P')}$$
and 
$$u_2 = \dist(\cdot,L(Q))\,\chi_{H(Q)},$$
and the operator $\wt Lv = \divv\wt A\nabla v$, with $\wt A$ defined by
$$\wt A(y) = \left\{\begin{array}{ll} A^*(y)& \text{if $y\in V(P')$,}\\
A_s^*(y_0)& \text{if $y\not\in V(P')$,}
\end{array}\right.$$
where $y_0$ is some arbitrary fixed point from $V(P')$. Further, 
we will assume that $A_s(y_0)=Id$ (otherwise we change variables).

For $r_0 \leq r\leq M\,\ell(Q)$, we set
$$J(r) = \left(\frac{1}{r^{2}} \int_{B(x_P,r)} \frac{|\grad u_1(y)|^{2}}{|y-x_P|^{n-1}}\,dy\right)\cdot \left(\frac{1}{r^{2}} \int_{B(x_P,r)} \frac{|\grad u_2(y)|^{2}}{|y-x_P|^{n-1}}\,dy\right) =:J_1(r)\,J_2(r).
$$ 
We have
\begin{enumerate}
\item $J_2(r)\approx 1$ for $r_0 \leq r\leq M\,\ell(Q)$, and

\vv
\item $J_1(r_0) \gtrsim 1$.
 
\end{enumerate}

\vv

The first statement follows from the fact that $|\nabla u_2|=1$ on $H(Q)$, while the second is due to the 
existence of the point $\wt Y_{P'}\in  B(x_{P'},C_2\ell(P')) \subset B(x_{P},2C_2\ell(P'))$, with $u(\wt Y_{P'})\approx \ell(P')\approx r_0$,
and then one argues as in the proof of Claim 2 in the proof of Lemma \ref{lemfi}.

Since $P \not\in\case(0,R)$, we know that
that \begin{align*}
\sup \bigl\{|A_s(y_1)-A_s(y_2)|:y_1,y_2\in V(P')\cap B(x_{P},M\ell(Q))
\bigr\}
&\leq 2 \theta_0\\
\sup \bigl\{|y-x_Q| \, |\vec b(y)|: y\in  V(P')\cap B(x_{P},M\ell(Q))
\bigr\} &\leq  \frac{2 \theta_0}{\kappa_0},
\end{align*}
if $\kappa_0$ is chosen small enough, depending on $M$ (which in turn will be chosen below depending on $\ve$) but independently of $\theta_0$.
Then, by Theorem \ref{teoACF-elliptic} applied to $\wt L$ and the subsequent Remark \ref{rem100}, we derive
$$J(r_0)\leq J(r)\, \left(\frac r{r_0}\right)^{C_6\,\theta_0} \quad\mbox{ for\, $r_0 \leq r\leq M\,
\ell(Q)$}
.$$
From the statements (1) and (2) above we infer that
$$J_1(r) \geq c_8 \left(\frac r{r_0}\right)^{-C_6\,\theta_0} \geq  c_8\,M^{- C_6\,\theta_0}\quad\mbox{ for\, $r_0 \leq r\leq M\,
\ell(Q)$}
.$$
If we assume $\theta_0$ small enough (so that $\theta_0\leq C_6^{-1} \frac{\log 2}{\log M}$), we get
\begin{equation}
\label{j1r>1}
J_1(r) \geq \frac{c_8}2\approx 1,
\end{equation}
for $r$ as above.

For some $r_1\in (r_0,r)$ to be chosen in a moment, we split
\begin{align*}
J_1(r) & = \frac{1}{r^{2}} \int_{B(x_P,r_1)} \frac{|\grad u_1(y)|^{2}}{|y-x_P|^{n-1}}\,dy +
\frac{1}{r^{2}} \int_{B(x_P,r)\setminus B(x_P,r_1)} \frac{|\grad u_1(y)|^{2}}{|y-x_P|^{n-1}}\,dy\\
& 
=: J_{1,a}(r) + J_{1,b}(r).
\end{align*}
By Lemma \ref{lemcc**}, we have
$$J_{1,a}(r)\lesssim \frac1{r^2}\,\|u_1\|_{\infty,B(x,2r_1)}^2\lesssim \frac{r_1^2}{r^2}.$$
Thus, if we take $r_1 = c_9\,r$ for a sufficiently small constant $c_9$, we have $J_{1,a}(r)\ll 1$, and in particular, $J_{1,b}(r)\gec 1$. By \eqref{j1r>1}, (1) above, and the definition of $J(r)$, we also get 
$$J_{1,a}(r)\leq \frac12 J(r),$$
and we get
$$1 \lesssim  J_{1,b}(r) \lesssim \frac1{c_9^2\,r^{n+1}}\int_{B(x_P,r)\cap V(P')} |\grad u_1(y)|^2\,dy
\lesssim \frac1{r^{n+3}}\int_{B(x_P,r)\cap V(P')} |u_1(y)|^2\,dy,$$
applying also Caccioppoli's inequality in the last estimate. By Lemma \ref{lemaux}, $u_1\lesssim r$   in $B(x_P,r)$,
and so we deduce that $\HH^{n+1}(B(x_P,r)\cap V(P'))\gtrsim r^{n+1}$ and
our claim \rf{eqclam10} follows.

\vv

We are ready now to get our desired contradiction to the assumption that $P\not \in \case(3,R)$. To this end
we will apply again the ACF formula. In this case
we consider an open set $V_*(P)$, which is analogous to $V(P')$, but replacing $P'$ by $P$ and $u$ by $u_*$. That is, 
$V_*(P)$ is the connected component of the
set $$\bigl\{x\in B(x_P,M\ell(Q)):u_*(x)>\tfrac12c_7\ell(P)\bigr\}$$ which contains the point $\wt Y_{P}^*\in B(x_{P},C_2\ell(P))$ given by 
 Lemma \ref{lem:failure}. By the same arguments as above, we have
$$V_*(P)\subset H(Q)^c.$$

We consider the 
functions
$$u_{*,1}= (u_*-\tfrac12c_7\ell(P))\,\chi_{V_*(P)}$$
and 
$$u_{*,2} = u_2= \dist(\cdot,L(Q))\,\chi_{H(Q)},$$
and the operator $\wh Lv = \divv\wh A\,\nabla v$, with $\wh A$ defined by
$$\wh A(y) = \left\{\begin{array}{ll} A(y)& \text{if $y\in V_*(P)$,}\\
A_s(y_{*})& \text{if $y\not\in V_*(P)$,}
\end{array}\right.$$
where $y_{*}$ is some arbitrary fixed point from $V_*(P)$. Further, again
we assume that $A_s(y_{*})=Id$ (otherwise we change variables).

For $r_0 \leq r\leq M\,\ell(Q)$, we set
$$J^*(r) = \left(\frac{1}{r^{2}} \int_{B(x_P,r)} \frac{|\grad u_{*,1}(y)|^{2}}{|y-x_P|^{n-1}}\,dy\right)\cdot \left(\frac{1}{r^{2}} \int_{B(x_P,r)} \frac{|\grad u_{*,2}(y)|^{2}}{|y-x_P|^{n-1}}\,dy\right) =:J_1^*(r)\,J_2^*(r).
$$ 
Now we have:
\begin{enumerate}
\item $J_2^*(r)\approx 1$ \,for $r_0 \leq r\leq M\,\ell(Q)$,
\vv
\item $J_1^*(r)\lesssim 1$ \,for $r_0 \leq r\leq M\,\ell(Q)$,

\vv
\item $J_1^*(r_0) \gtrsim c(\ve_0)$, \,and

\vv
\item $\HH^{n+1}\bigl(B(x_P,r)\setminus (V_*(P)\cup H(Q)\bigr) \geq c\,r^{n+1}$\, if\, $r_0 \leq r\leq M\,
\ell(Q)$ and $\tau_0$ is assumed to be small enough.
\end{enumerate}

\vv


The first statement is again immediate because $|\nabla u_{*,2}|\equiv 1$ in $H(Q)$, and the second one follows from Lemma \ref{lemcc**}. Concerning the third statement, from the existence of the point $\wt Y_{P'}^*\in  B(x_{P},C_2\ell(P')) $, with $u(\wt Y_{P}^{*})\approx \ell(P)$,
arguing as in the proof of Claim 2 in the proof of Lemma \ref{lemfi}, we infer that 
$$\frac{1}{\ell(P)^{2}} \int_{B(x_P,2C_2\,\ell(P))} \frac{|\grad u_{*,1}(y)|^{2}}{|y-x_P|^{n-1}}\,dy\gtrsim 1.$$
Also recall that by the definition of $r_{0}$, for $\ve\ll C_2^{-1}$,
\[
r_{0} \geq \ve \ell(Q)+2\ve^{3/2}\ell(Q) >\frac{\ve}{2} \ell(Q)>\frac{\ve^{-1/2}}{2}\ell(P) \gg 2C_2\ell(P)\]
and the combination of these two estimates implies that 
$$J_1^*(r_0) \gtrsim \frac{\ell(P)^2}{r_0^2} \gtrsim \ps{\frac{\ve^{3/2}\,\ell(Q)}{K_0^2\, \ell(Q)} }^{2} = c(\ve)$$
Finally, to prove the statement (4) above, recall that, by \rf{eqclam10}, if\, $r_0 \leq r\leq M\,
\ell(Q)$,
$$
\HH^{n+1}\bigl(\{x:u(x)> c\,\ell(P)\} \cap H(Q)^c\cap B(x_P,r)\bigr) \geq c\,r^{n+1}.$$

We claim that $\{x:u(x)> c\,\ell(P)\}\cap V_*(P)=\varnothing$ if $\tau_0$ is small enough. Otherwise, by a Harnack chain argument (using the connectedness of $V_{*}(P)$ { and \eqref{eq2}})
 it easily follows that
$u(x)\geq c'(\ve)$ for all $x\in V_*(P)\cap 20B_P$, with $c'(\ve)>0$ depending on $\frac{M\ell(Q)}{\ell(P)}$ and thus on $\ve$ (recall that we assume $M\ll \ve_0^{-1}$, but we may have $\ve^{-1}\ll M$).
So if $\tau_0<\min\{c,c(\ve)\}$, then $P\not\in\case(3,R)$, which contradicts our initial assumption.

On the other hand, since $P \not\in\case(0,R)$, we know that
that \begin{align*}
\sup \bigl\{|A_s(y_1)-A_s(y_2)|:y_1,y_2\in V(P)\cap B(x_{P},M\ell(Q))
\bigr\}
&\leq 2 \theta_0\\
\sup \bigl\{|y-x_Q| \, |\vec b(y)|: y\in  V(P)\cap B(x_{P},M\ell(Q))
\bigr\} &\leq  \frac{2 \theta_0}{\kappa_0},
\end{align*}
if $\kappa_0$ is chosen small enough, depending on $\ve$ and $M$. but independently of $\theta_0$.

Now we have checked that all the assumptions in Lemma \ref{lemcoro} hold, and then we deduce that
$$
\left(\frac{M\,\ell(Q)}{r_0}\right)^\rho\leq  C\,\frac{J(x,M\,\ell(Q))}{J(x,r_0)} \lesssim C(\ve).
$$
for some positive constant $\rho$. This implies that $M\leq C'(\ve)$. Hence if we choose $M=2\,C'(\ve)$ we get our desired contradiction.
\end{proof}

\vv

\section{Uniform rectifiability} \label{secuniform}

\def\BATPP{\mathsf{BATPP}}
 In this section we assume that
$\mu$ is an arbitrary $n$-AD-regular measure in $\R^{n+1}$, and we set $E=\supp\mu$. Our objective consists
in proving Proposition \ref{propowtf}, which is the last ingredient of the proof of Theorem \ref{teo1}.

%

\vv
For convenience, we restate a more precise version of Proposition \ref{propowtf}

\def\H{\HH}
\begin{propo}\label{wtf}
Let $\mu$ be an $n$-AD-regular measure in $\R^{n+1}$.
Let $A_0,a_0,\alpha,p,t,\tau,\ve>0$. Let $\FF\subset \wts(A_0,\alpha,p,t,\tau)$ be some family of cubes satisfying a compatibility condition with constant $a_0$. Suppose that for all $S\in \DD_\mu$
\begin{equation}\label{Bsum}
\sum_{Q\subset S \atop Q\not\in \BATPP_{\ve} \cup \FF} \mu(Q) \lec \mu(S),
\end{equation}
with the implicit constant possibly depending on $A_0,a_0,\alpha,p,t,\tau,\ve$.
Suppose that the constants in the definition of $\wts$ are so that, for $p>0$ fixed, $t$ and $\tau$ are small enough. Assume also that the compatibility constant $a_0$ is big enough (possibly depending on $A_0,\alpha,p,t,\tau$), and that $\ve$ is small enough. Then $\mu$ is uniformly rectifiable.
\end{propo}

\vv

First we will use this lemma to complete the proof of Theorem \ref{teo1}.

\begin{proof}[\bf Proof of Theorem \ref{teo1}, (a) or (b) $\Rightarrow$ (c)] Assume either (a) or (b) of Theorem \ref{teo1}
 and let us check that the assumptions of Proposition \ref{wtf} hold. 
We know that $\mu$ admits a corona decomposition involving elliptic measure 
such as the one in Proposition \ref{propo1}, with the family $\wt\ttt\subset\DD_\mu$ satisfying a Carleson packing condition.

We consider the families of cubes $\case(i,R)$, for $R\in\ttt$ and $i=0,1,2,3$.
By Lemmas \ref{lemcase0} and \ref{lem:packing-case1}, for all $S\in\DD_\mu$ we have
$$\sum_{R\in\wt\ttt}
\sum_{\substack{Q\subset S:\\Q\in\case(0,R)\cup\case(1,R)}} \mu(Q)\leq C\,\mu(S).
$$

From this estimate and Lemmas \ref{lemcase0}, \ref{lemfund}, \ref{lem:packing-case1}, and \ref{lem7.1} (recalling Remark \ref{batppremark}), it follows that
$$\sum_{R\in\wt\ttt}\,
\sum_{\substack{Q\subset S:\\Q\in\case(2,R)\setminus \batpp_\ve}} \mu(Q)\leq C\,\mu(S).
$$


Further, by Lemma \ref{lem6.1}, for $A_0$ big enough and for some $p>0$ depending only on the constants in 
the corona decomposition in Proposition \ref{propo**}, the family
$$\FF := \bigcup_{R\in\wt\ttt}\case(3,R)$$
is contained in $\wts(A_0,\alpha,p,t,\tau)$ satisfies the compatibility condition with constant $a_0>1$,
with $\tau,t$ arbitrarily small and $a_0$ arbitrarily big if $\tau_0$ and $\kappa_0$ are small enough in the definitions of the 
families $\case(i,R)$. Hence the assumptions in Proposition \ref{wtf} hold and so $\mu$ is uniformly $n$-rectifiable.
\end{proof}
\vv

We now proceed with the proof of Lemma \ref{wtf}, which is modelled on arguments similar to the ones in \cite{DS2,DS3}). In particular, the condition of $Q\in \wts$ is similar to the $\mathsf{WTN}$ condition in \cite{DS2}. However, it is not easy to derive one property from the other, as far as we know. A further complication arises from the fact that the characterization of uniform rectifiability in Lemma \ref{wtf}involves two different collections of cubes, $\wts$ and $\BATPP_{\ve}$. 
\def\cB{\mathscr{B}}

The constants $\ve,\alpha,a,a_{0},\tau,\lambda,p,\upsilon_{1}$ and $c$ will appear in the proof and we will adjust their values as we go along. The reader can check that their dependencies are
\[
\ve\ll a\alpha, \;\; \max\{\tau,\ve\} \ll c \ll \frac{a}{\lambda}, \;\; 2^{-a_{0}}\approx a\ll p<1,\] 
\[
\max\{\upsilon_{1},\upsilon_{2},\tau,\ve,b\}\ll \kappa, \;\; \lambda^{-1} \ll \upsilon_{1}^{n} ,\]
\[
\upsilon_{1}\ll \upsilon_{2} \ll p\kappa < \kappa\]
and $\kappa$ is a universal constant we will introduce near the end of the section. We also assume all constants depend on the Ahlfors regularity of $E$ and $n$ and suppress this dependency from the notation.

Let 
\[
\cB=\DD_\mu\backslash (\BATPP_{\ve} \cup \wts).\]

\noi We define an auxiliary set first as follows:
\[
\wt{E}= E\cup \bigcup_{Q\in \cB} C_{Q},\]
where $C_{Q}$ is the union of the boundaries of all dyadic cubes of side length between $b\ell(Q)$ and $2b\ell(Q)$ that intersect $10B_{Q}$, where $b>0$ is a small number we will pick later. 

It is easily checked that this set is also $n$-regular by \eqref{Bsum}. Our goal now is to show the following.

\begin{lemma}\label{etilde}
The set $\wt{E}$ is uniformly rectifiable. 
\end{lemma}

It will then follow that $E$ is also uniformly rectifiable. One way to see this is that, if $\wt{E}$ is uniformly rectifiable, then by the results of \cite{DS1}, $\beta_{\wt{E}}(x,r)^{2} d\H^n(x)\frac{dr}{r}$ is a Carleson measure on $\wt{E}\times (0,\infty)$, where
\[
\beta_{\wt{E}}(x,r)^{2} := \inf_{P \text{ an $n$-plane}} \frac{1}{r^{n}}\int_{\wt{E}\cap B(x,r)} \ps{\frac{\dist(y,P)}{r}}^{2} d\H^{n}(y)\]
and so the restriction of $\beta_{\wt{E}}(x,r)^{2} d\H^n(x)\frac{dr}{r}$ to $E\times (0,\infty)$ is also Carleson since $E\subset \wt{E}$. For the same reason, $\beta_{{E}}(x,r)\leq \beta_{\wt{E}}(x,r)$, and thus $\beta_{{E}}(x,r)^{2} d\H^{n}(x)\frac{dr}{r}$ is a Carleson measure on $E\times (0,\infty)$, and this implies $E$ is uniformly rectifiable again by the results in \cite{DS1}. 

When proving Lemma \ref{etilde}, recall that $\mu$ is a measure supported on $E$, while on $\wt E$ we will consider the measure $\HH^n|_{\wt E}$.

For $Q\in \DD_\mu\cap \BATPP_{\ve}$, we will set $H_{1}(Q)$ and $H_{2}(Q)$ to be the disjoint open half-spaces whose boundaries are $P_{1}(Q)$ and $P_{2}(Q)$ respectively, and $M(Q)$ the open region between $P_{1}(Q)$ and $P_{2}(Q)$. Let $H_{1}'(Q)$ be the connected component of $10B_{Q}\backslash (P_{1}(Q)\cup P_{2}(Q))_{\ve \ell(Q)}$ contained in $H_{1}(Q)$, where we will fix $\ve>0$ later, and define $H_{2}'(Q)$ and $M'(Q)$ similarly.

For $Q\in \DD_\mu$, we will define the {\it components} of $Q$ to be
\begin{enumerate}
\item the three connected components $H_{1}'(Q),H_{2}'(Q)$, or $M'(Q)$ if $Q\in \BATPP_{\ve}$.
\item  $U_{1}(Q)$ and $U_{2}(Q)$ if $Q\in \wts$, or
\item  the connected components of $(C_{Q})^{c}$ if $Q\in \cB$. 
\end{enumerate}

Note that 
\begin{equation}
\dist([x,y],E) \leq \ve\ell(Q) \mbox{ if $x,y$ are in different components of $Q\in \BATPP_{\ve}$}.
\label{xywhsa}
\end{equation} 
This follows since we can find $z\in [x,y]\cap \bigl(P_{1}(Q)\cup P_{2}(Q)\bigr)$ and then $\dist(z,E)\leq \ve\ell(Q)$ by \eqref{batpp1}. 

Similarly, 
\begin{equation}
\dist([x,y],E) \leq \tau \ell(Q) \mbox{ if $x,y$ are in different components of $Q\in \wts$}.
\label{xyG}
\end{equation} 
Indeed, let $x\in U_{1}^{}(Q)$ and $y\in U_{2}^{}(Q)$, and let $z\in[x,y]\setminus (U_1(Q)\cup U_2(Q))$. By the property (1) in the definition of $\wts$  there is $z'\in E$ such that $|z-z'|\leq \tau\,\ell(Q)$. 

Thus, in any case, if $x,y$ are in different components of $Q$, then 
\begin{equation}\label{ac'}
\dist([x,y],E)\leq  \max\{\tau,\ve\}\ell(Q).
\end{equation}

We say that two balls $B_{1}$ and $B_{2}$ of radii at least $\upsilon_{1}\ell(Q)$ are {\it separated} for $Q\in \DD_\mu$ if 
\begin{enumerate}
\item $2B_{1}\cup 2B_{2}\subset 10B_{Q}\backslash \wt{E}$ and 
\item $B_{1}$ and $B_{2}$ are contained in different components of $Q$.
\end{enumerate}
We say that $B_{1}$ and $B_{2}$ are separated if they are separated for some cube $Q$ or are in different components of $(C_Q)^c$. 

\def\co{\textrm{co}}
For a separated pair $\cC=(B_{1},B_{2})$, set 
\[
G(B_{1},B_{2}) = \co (B_{1}\cup B_{2})\cap \wt{E}, 
\]
where $\co(F)$ stands for the convex hull of $F\subset\R^{n+1}$.  
Let $P(\cC)$ be a plane orthogonal to $[x_{B_{1}},x_{B_{2}}]$, where $x_{B_i}$ are the centers of $B_i$. Let $\pi_{\cC}$ be the orthogonal projection onto this plane. For $x\in P(\cC)$, define
\[
M_{\cC}(x) = \sup_{r>0} r^{-n} \H^{n}(\pi_{\cC}^{-1}(B(x,r))\cap \wt{E}\cap 20B_{Q}).\]
Let $\lambda>0$ to be chosen (depending on $\upsilon_{1}$) so that the set 
\[
F_{\cC}=\{x\in \pi_{\cC}(B_{1}): M_{\cC}(x)\leq \lambda\}.\]
satisfies
\begin{equation}\label{3/4}
|F_{\cC}|\geq  \frac{3}{4}|\pi_{\cC}(B_{1})|.
\end{equation}
This $\lambda>0$ exists since the usual proof of the Hardy-Littlewood maximal inequality implies
\[
|\{x\in \pi_{\cC}(10B_{Q}): M_{\cC}(x)>\lambda \}| \lec \frac{1}{\lambda},\]
with the implicit constant depending on $n$ and the Ahlfors regularity constant of $\wt{E}$. 
\vv

\begin{lemma}
There is $\upsilon_{1}>0$ so that the following holds. Let $\cC=(B_{1},B_{2})$ be a separated pair of balls for $Q\in\DD_\mu$, $x\in B_{1}$ and $y\in B_{2}$ so that $\pi_{\cC}(x)=\pi_{\cC}(y)\in F_{\cC}$ and $r(B_{1})=r(B_{2}) \geq \upsilon_{1}\ell(Q)$. Then $[x,y]\cap \wt{E}\neq\varnothing$. 
\end{lemma}
\def\bad{\textrm{Bad}}
\begin{proof}
Note that if there is $Q'\in \cB$ with $\ell(Q')\leq \ell(Q)$ so that $10B_{Q'}\cap  [x,y]\neq\varnothing$, then we are done since, for $b>0$ small enough (depending on $\upsilon_{1}$), $[x,y]\cap \wt{E}\neq\varnothing$. Hence, we can assume without loss of generality that $Q'\not\in \cB$ whenever $\ell(Q')\leq \ell(Q)$ and $10B_{Q'}\cap [x,y]\neq\varnothing$. 

Let $x_{1}^{1}=x$ and $x_{2}^{1}=y$.  We will inductively find points $x_{i}^{j}$, $i=1,2$ and $j=1,2,...$ and cubes 
$Q_{j}\in\DD_\mu$ such that if $I_{j}=[x_{1}^{j},x_{2}^{j}]$, and $B_{i}^{j}= B(x_{i}^{j},c\ell(Q_{j}))$, then 
\begin{enumerate}
\item  $\ell(Q_{j})\approx \diam I_{j}$ and $I_{j}\subset 10B_{Q_{j}}$,
\item $B_{1}^{j}$ and $B_{2}^{j}$ are separated for $Q_{j}$,
\item $I_{j+1}\subset I_{j}$,
\item $\ell(Q_{j+1})<\frac{1}{2}\ell(Q_{j})$ and $\dist(I_{j},E)\leq\frac{3}{4} |I_{j-1}|$,
\end{enumerate}
Note that since $x_{i}^{j}$ are in two components of $Q_{j}$, then \eqref{ac'} implies $\dist(I_{j},E)<\ell(Q_{j})$, and so the last condition implies that $\bigcap_j I_{j}$ is a point in $E$, which implies the lemma. 

We now proceed with finding the points $x_{i}^{j}$. 
We first set $Q_{1} = Q$, $x_{1}^{1}=x$, and $x_{2}^{1}=y$. 

Now suppose we have chosen $x_{i}^{j}$ for $i=1,2$ and some $j$. Assume $I_{j}=[x_{1}^{j},x_{2}^{j}]$ and $Q_{j}$ are such that $|I_{j}|\approx \ell(Q_{j})$, $I_{j}\subset 2B_{Q_{j}}$, and 
\[
B_{i}^{j}=B(x_{i}^{j},c\ell(Q_{j}))\subset  2B_{i}^{j}\subseteq 2B_{Q_{j}}.\]
Let $z_{j}$ be the closest point in $I_{j}$ to $E$. We can assume without loss of generality that $z_{j}=0$, and also that 
\[
|x_{1}^{j}-z_{j}|\geq \frac{|x_{1}^{j}-x_{2}^{j}|}{2}\gec \ell(Q_{j}).\]

\noi {\bf Claim:} Let $a\in (0,1/2)$. For $c$ small enough (depending on $\lambda$, $n$, and the Ahlfors regularity of $\wt{E}$), we can find
\[
y_{j}\in [ax^{j}_{1},ax^{j}_{1}/2]\] 
so that $B(y_{j},2c\ell(Q_{j}))\subset \wt{E}^{c}$. 

Indeed, suppose instead that every $y\in [ax_{j}^{1},ax_{j}^{1}/2]$ had $\dist(y,\wt{E}^{c})\leq 2c\ell(Q_{j})$. Let $B^{1},...,B^{k}$ be balls with radii $4c\ell(Q_{j})$, centers in $[ax_{j}^{1},ax_{j}^{1}/2]$, and so that every point on the interval is in at least one ball but no more than two. Then 
\[
k\approx \frac{|ax_{j}^{1} - ax_{j}^{1}/2|}{4c\ell(Q_{j})}\approx \frac{a}{c}\]
and since $x\in F_{\cC}$,
\begin{align*}
\lambda \geq M_{\cC}(x)
& \geq (4c\ell(Q_{j}))^{-n} \H^{n}(\pi_{\cC}^{-1}(B(x,4c\ell(Q_{j}))\cap \wt{E}\cap 10 B_{Q})) \\
& \geq \frac{1}{2} (4c\ell(Q_{j}))^{-n} \sum_{i=1}^{k} \H^{n}(B^{i}\cap \wt{E})\\
& \gec (4c\ell(Q_{j}))^{-n} k (4c\ell(Q_{j}))^{n} \approx k \approx \frac{a}{c}.\end{align*}
Thus, for $c\ll \frac{a}{\lambda}$, we get a contradiction. This proves the claim.

Note that for $\ve,\tau < c$ small enough,
\[
\dist(B(y_{j},c\ell(Q_{j})), \wt{E})
\geq c\ell(Q_{j})>\max\{\ve,\tau\} \ell(Q_{j})\]
and so $B(y_{j},c\ell(Q_{j}))$ is contained in a component of $Q_{j}$. By convexity, we also known that $B(y_{j},c\ell(Q_{j}))\subset 10 B_{Q_{j}}$. Hence, since $(B(x_{1}^{j},c\ell(Q_{j})),B(x_{2}^{j},c\ell(Q_{j})))$ are separated for $Q_{j}$ by the induction hypothesis, we know that either $(B(x_{1}^{j},c\ell(Q_{j})), B(y_{j},c\ell(Q_{j})))$ or $(B(x_{2}^{j},c\ell(Q_{j})),B(y_{j},c\ell(Q_{j})))$ form another separated pair for $Q_{j}$. 
Without loss of generality, we can assume it is $(B(x_{1}^{j},c\ell(Q_{j})), B(y_{j},c\ell(Q_{j})))$, the rest of the proof is the same if instead $(B(x_{2}^{j},c\ell(Q_{j})),B(y_{j},c\ell(Q_{j})))$ form a separated pair. 

We let $x_{1}^{j+1}= x_{1}^{j}$ and $x_{2}^{j+1}=y_{j}$. We will now verify the conditions in our induction claim for these points. 

Since $y_{j}$ and $x_{1}^{j}$ are in different components of $Q_{j}$, \eqref{ac'} implies we can find $z_{j+1}\in [x_{1}^{j},y_{j}]$ so that 
\begin{equation}\label{zj+1}
\dist(z_{j+1},E)< \isif{ \tau \ell(Q_{j})  & Q_{j}\in \wts \\ \ve\ell(Q_{j}) & Q_{j}\in \BATPP_{\ve}}.
\end{equation}
 Let $z_{j+1}'\in E$ be closest to $z_{j+1}$ and let $Q_{j+1}$ be the smallest cube containing $z_{j+1}'$ so that $2B_{Q_{j+1}}\supset [y_{j},x_{1}^{j}]$. Then since $z_{j+1}$ is close to $Q_{j+1}$, it is not hard to show that
 \begin{equation}\label{aqj}
|I_{j+1}| \approx  \ell(Q_{j+1})\approx a \ell(Q_{j}).
 \end{equation}
 Hence, for $a$ small enough, there is some $a_{0}\in \bN$ so that $a\approx 2^{-a_{0}}$ and 
\begin{equation}\label{qjqj+1qj}
\ell(Q_{j})2^{-a_{0}}\leq  \ell(Q_{j+1})<\min\Bigl\{ \frac{p}{4} , \frac{1}{20}\ell(Q_{j})\Bigr\},
\end{equation}
where $p$ is the corkscrew constant from the definition of $\wts$. Since $a<1/2$, we have $|I_{j+1}|\leq \frac{3}{4} |I_{j}|$. This and \eqref{qjqj+1qj} imply condition (4). Clearly, (3) holds since by construction we have picked $I_{j+1}\subset I_{j}$, and (1) follows from \eqref{aqj}. It now suffices to show that the balls $(B(x_{1}^{j},c\ell(Q_{j+1})), B(y_{j},c\ell(Q_{j+1})))$ form a separated pair for $Q_{j+1}$. 

Since $x_{1}^{j},y_{j}\in 2B_{Q_{j+1}}$, we have for $c>0$ small that
\[
B(x_{1}^{j},2c\ell(Q_{j}))\cup B(y_{j},2c\ell(Q_{j}))\subset 10 B_{Q_{j+1}}.\]
So now it remains to show that they are in two different components of $Q_{j+1}$. We will also need the following estimate later:
\begin{equation}\label{10B}
10B_{Q_{j+1}}\subset 10 B_{Q_{j}}.
\end{equation}
To see this, note that if $x\in 10B_{Q_{j+1}}$, then since $x,z_{j+1}'\in 10 B_{Q_{j+1}}$,
\begin{align*}
|x-x_{Q_{j}}|
& \leq |x-z_{j+1}'|+|z_{j+1}'-z_{j+1}|+|z_{j+1}-x_{Q_{j}}|\\
& \leq \diam 10 B_{Q_{j+1}}+\max\{\tau,\ve\}\ell(Q_{j})+2r(B_{Q_{j}})\\
& <20r(B_{Q_{j+1}})+\max\{\tau,\ve\}c_{10}^{-1} r(B_{Q_{j}})+2r(B_{Q_{j}})\\
& \stackrel{\eqref{qjqj+1qj}}{<}(1+\max\{\tau,\ve\}c_{10}^{-1} +2)r(B_{Q_{j}})<10r(B_{Q_{j}})
\end{align*}
which implies \eqref{10B}.


For $\tau,\ve\ll c$, 
\[
\dist(\{x_{1}^{j},y_{j}\},E)
\geq c\ell(Q_{j}) >c\ell(Q_{j+1})\gg \max\{\tau,\ve\} \ell(Q_{j+1})\]
and thus, \eqref{10B}, Lemma \ref{lemfi}, and our definition of components imply $B(x_{1}^{j},c\ell(Q_{j+1}))$  and $B(y_{j},c\ell(Q_{j+1}))$ must be contained in components of $Q_{j+1}$. It now suffices to show that they are not in the same component of $Q_{j+1}$. 

Recall that by assumption. $Q_{j+1}\not\in \cB$ since $10B_{Q_{j+1}}\cap [x,y]\neq\varnothing$ (since $10B_{Q_{j}}\supset 2B_{Q_{j+1}}\supset I_{j+1}$). 

\vv


\noi{\bf Case 1:} If $Q_{j+1}\in \BATPP_{\ve}$, suppose $y_{j}$ and $x_{1}^{j}$ are in the same component of $Q_{j+1}$. Since $y_{j}$ and $x_{1}^{j}$ lie in different components of $Q_{j}\in \BATPP_{\ve}{\cup \wts}$ by assumption, $\dist(z_{j+1},E)<\max\{\tau,\ve\} \ell(Q_{j})$ by \eqref{zj+1}. But if $x_{1}^{j}$ and $y_{j}$ lie in $H_{1}'(Q_{j+1})$, for example, then one of those points is closer to $P_{1}(Q)\cup P_{2}(Q)$ than $z_{j+1}$. Hence, for $\tau, \ve\ll c$ small,
\begin{align*}
\dist(z_{j+1},E)
& \stackrel{\eqref{batpp1}}{\geq}  \dist(z_{j+1}, P_{1}(Q_{j+1})\cup P_{2}(Q_{j+1}))-\ve \ell(Q_{j+1})\\
& \geq \dist(\{x_{1}^{j},y_{j}\},P_{1}(Q_{j+1})\cup P_{2}(Q_{j+1}))-\ve \ell(Q_{j+1})\\ 
& \stackrel{\eqref{batpp2}}{\geq}  \dist(\{x_{1}^{j},y_{j}\},E)-2\ve \ell(Q_{j+1})\\ 
& \geq c\ell(Q_{j+1})-2\ve \ell(Q_{j+1}) 
\geq \frac{c}{2}\ell(Q_{j+1})
> \max\{\ve,\tau\}\,\ell({ Q_{j}})\\
& \stackrel{\eqref{zj+1}}{>}\dist(z_{j+1},E)
\end{align*}
which is a contradiction, and the same would happen if $x_{1}^{j}$ and $y_{j}$ both lied in $H_{2}'(Q_{j+1})$ or $M'(Q_{j+1})$.

\vv
\noi{\bf Case 2:} If $Q_{j},Q_{j+1}\in \wts$ but  $(B(x_{1}^{j},c\ell(Q_{j+1})), B(y_{j},c\ell(Q_{j+1})))$  are not separated for $Q_{j+1}$, this implies that these balls are contained in the same $U_{i}^{}(Q_{j+1})$ set. Without loss of generality, assume this set is $U_{1}^{}(Q_{j+1})$. 

But then by the compatibility condition for cubes in $\wts$,
\[
B(x_{1}^{j},c\ell(Q_{j+1})) \cup B(y_{j},c\ell(Q_{j+1})) 
\subset U_{1}^{}(Q_{j+1})
\subset U_{1}'(Q_{j})\]
and since $U_{1}^{}(Q_{j+1})$ is connected and $ B(y_{j},c\ell(Q_{j+1})) \subset U_{2}^{}(Q_{j})\subset U_{2}'(Q_{j})$, this implies $U_{1}'(Q_{j})\cap U_{2}'(Q_{j})\neq\varnothing$, which contradicts $Q_{j}\in \wts$. 

\vv
\noi{\bf Case 3:} Suppose $Q_{j}\in \BATPP_{\ve}$ and $Q_{j+1}\in \wts$.
%

By hypothesis, $x_{1}^{j}$ and $y_{j}$ are in different components of $Q_{j}$, say they are to either side of the plane $P_{1}(Q_{j})$. If $x_{1}^{j}$ and $y_{j}$ are in the same component of $Q_{j+1}$, say $U_{1}(Q_{j+1})$, then there is a curve $\Gamma$ connecting $x_{1}^{j}$ and $y_{j}$ inside $U_{1}(Q_{j+1})\subset A_{0} B_{Q_{j+1}}\subset A_{0}B_{Q_{j}}$. Then there must be $z\in \Gamma\cap P_{1}(Q_{j})\cap A_{0}B_{Q_{j}}$, but then by property (4) of the $\wts$ condition,
\[
\alpha \ell(Q_{j+1})
\lec \dist(\Gamma, E)
\leq \dist(z,E) 
\lec \ve \,\ell(Q_{j})
\lec \ve \,a^{-1}\ell(Q_{j+1}),\]
which is a contradiction for $\ve\ll a \,\alpha$.

%
%
%
%
\end{proof}

\vv
In the next lemma we will consider the family of cubes associated with the measure $\tilde{\mu} := \HH^n|_{\wt E}$. We denote this family 
by $\DD_{\tilde{\mu}}$.

\begin{lemma}\label{lem8.4}
Let $\theta\in \bS^{n}$ and $\cB_{\theta}$ be the set of cubes from $\DD_{\tilde{\mu}}$ for which there are balls $B_{i}=B_{i}(Q)$, $i=1,2,3$, mutually disjoint such that
\begin{enumerate}
\item $2B_{i}\subset \wt{E}^{c}$,
\item $B_{2}$ is separated from $B_{1}$ and { $B_{3}$},
\item $r(B_{1})=r(B_{2})=r(B_{3})=\upsilon_{2}\ell(Q)$, and
\item the centers of the balls $B_{i}$ are on a line parallel with $\theta$ with $B_{2}$ between $B_{1}$ and $B_{3}$.
\end{enumerate}
Then for all $R\subset \wt{E}$,
\[
\sum_{Q\subset R \atop Q\in \cB_{\theta}} \tilde{\mu}(Q) \lec_{\upsilon_{2}} \tilde{\mu}(R).\]
\end{lemma}

\begin{proof}
Without loss of generality, we can assume $\theta=e_{n+1}$. For each $Q\in \cB_{\theta}$, the center of $B_{1}(Q)$ lies above $B_{2}(Q)$ whose center lies above $B_{3}(Q)$. For $Q\in \cB_{\theta}$, let 
\[
G_{Q}=\{x\in \co(B_{1}(Q),B_{2}(Q))\cap \wt{E}: \exists y\in \co(B_{2}(Q),B_{3}(Q))\cap \wt{E} \mbox{ s.t. } (x,y)\subset \wt{E}^{c}\}.\] 
Observe that $x$ is above $y$.\\

\noi{\bf Claim:} For $Q\in \cB_{\theta}$
\begin{equation}\label{mugq}
\tilde{\mu} (G_{Q}) \gec \tilde{\mu} (Q).
\end{equation}

Indeed, writing $B_{i}=B_{i}(Q)$ for the moment, we just have to note that
\[
 \pi_{\R^{n}}(B_{2}) \supset \pi_{\R^{n}} (G_{Q}) 
\supset F_{(B_{1},B_{2})}\cap F_{(B_{2},B_{3})}\]
and by \eqref{3/4},
\[
|F_{(B_{i},B_{2})}|\geq \frac{3}{4} |\pi_{\R^{n}}(B_{2})|.\]
Thus, by the pigeonhole principle, 
\[
\tilde{\mu} (G_{Q})\geq | \pi_{\R^{n}} (G_{Q})|\geq   |F_{(B_{1},B_{2})}\cap F_{(B_{2},B_{3})}|\geq \frac{1}{4} |\pi_{\R^{n}}(B_{2})| \approx \ell(Q)^{n} \approx \tilde{\mu} (Q)\]
which proves \eqref{mugq}.\\

\noi{\bf Claim:} The sets $G_{Q}$ have bounded overlap, i.e.
\begin{equation}\label{overlap}
\sum_{Q\in \cB_{\theta}} \chi_{G_{Q}} \lec 1.
\end{equation}
Clearly this implies the lemma. Suppose $Q,R\in \cB_{\theta}$ are such that $\ell(Q)<\delta \ell(R)$ for some $\delta>0$ to be picked shortly and $G_{Q}\cap G_{R}\neq\varnothing$. Let $x\in G_{Q}\cap G_{R}$. Then there is $y\in \wt{E}\cap \co (B_{1}(Q)\cup B_{3}(Q))$ so that $(x,y)\subset \wt{E}^{c}$. For that same $x$, there is $z \in \wt{E}\cap \co (B_{1}(R)\cup B_{3}(R))$ for which $(x,z)\subset \wt{E}^{c}$. In each case, $y$ and $z$ are below $x$, but $y\neq z$ since, because $2B_{2}(Q)\subset \wt{E}^{c}$ and $(x,y)$ passes through $B_{2}(Q)$,
\[
|x-y| \geq r(B_{2}(Q)) = \upsilon_{2} \ell(Q) \]
whereas  for $\delta>0$ small enough
\begin{align*}
|x-z|
& \leq \diam \co (B_{1}(R)\cup B_{3}(R))\leq \diam 10B_{R} 
=20 c_{10} \ell(R) 
<20c_{10} \delta \ell(Q) \\
& < \upsilon_{2}\ell(Q) \leq |x-y|.
\end{align*}
And thus $z\in (x,y)\cap \wt{E}$, which is a contradiction. Thus, if $Q,R\in \cB$ and $\ell(R)<\delta\ell(Q)$, $G_{Q}\cap G_{R}=\varnothing$, and this implies \eqref{overlap}.

\end{proof}

For $\kappa>0$, let $\ls_{\kappa}$ denote the set of cubes $Q\in \DD_{\tilde{\mu}}$ such that for all $x,y\in 2B_{Q}\cap \wt E$,  $$\dist(2x-y,\wt E)\leq \kappa \ell(Q).$$ 
We say that $\wt E$ satisfies the {\it local symmetry condition} if, for all $R\in \DD_{\tilde{\mu}}$,
\[
\sum_{ Q\subset R \atop Q\not\in \ls_{\kappa}} \tilde{\mu} (Q)\lec \tilde{\mu} (R).\]
In \cite{DS1}, it is shown that the local symmetry condition, with $\kappa>0$ small enough, is equivalent
to the uniform rectifiability of $\wt E$. We will use this criterion to prove Proposition \ref{wtf}.

\begin{lemma}
For $\upsilon_{2}>0$ small enough in the statement of Lemma \ref{lem8.4},  if $Q\in \DD_{\wt{E}}$ and $Q\not\in \ls_{\kappa}$, then $Q\in \cB_{\theta}$ for some $\theta\in \bS^{n}$. 
\end{lemma}

\begin{proof}
Let $Q\not\in \ls_{\kappa}$ and let $x,y\in 2B_{Q}\cap \wt E$ so that  $\dist(2x-y,\wt E)\geq \kappa \ell(Q)$.  Let $B_{2}=B(x,\frac{\kappa}{100}\ell(Q))$, $B_{1}= B(y,\frac{\kappa}{100}\ell(Q))$ and $B_{3}= B(2x-y,\kappa\ell(Q)/2)$. For $\tau,\ve\ll \kappa$ small enough, we know that $B_{3}$ is contained in some component of $Q$.\\

\noi{\bf Claim:} For $\upsilon_{2}>0$ small enough, we can find a ball $\wt{B}_{i}\subset B_{i}\backslash \wt{E}$ colinear and of radius $\upsilon_{2}\ell(Q)>0$  so that $B_{1}$ is separated from $B_{2}$ and $B_{2}$ is separated from $B_{3}$ for $\upsilon_{1}< \upsilon_{2}$.

If $\dist(x,E)\geq \frac{\kappa}{200} \ell(Q)$, then $x\in C_{R}$ for some $R\subset E$, $R\in \cB$. Moreover, $\ell(R)\gec \kappa\ell(Q)$. For $\upsilon_{1},b\ll \kappa $, we can ensure that we can find balls $\wt{B}_{i}$ satisfying the conclusions of the lemma. 

If $\dist(x,E)<\frac{\kappa}{200}\ell(Q)$, let $x'\in E$ be so that 
\[
B(x',\frac{\kappa}{200}\ell(Q))\subset B_{2}.\]

Let $R\in\DD_\mu$ have $\ell(R)=\ell(Q)$. 

\begin{enumerate}
\item If $R\in \BATPP_{\ve}$, note that for $\tau,\ve\ll \kappa$, $B_{3}$ is contained in a component of $R$.  For $\upsilon_{2}>0$ small enough, we can pick $\wt{B}_{2}\subset B_{2}\backslash \wt{E}$ of radius $\upsilon_{2}\ell(Q)$  in a different component of $R$ than $B_{3}$ and then pick $\wt{B}_{1}\subset B_{1}\backslash \wt{E}$ in a different component of $R$ to $B_{2}$ of the same radius. Let $\wt{B}_{3}=B(2x_{\wt{B}_{2}}-x_{\wt B_{1}},\upsilon_{2}\ell(Q))$. Then for $\upsilon_{2}< \kappa/3$,
\[
\wt{B}_{3}
\subset B(2x-y,4\upsilon_{2}\ell(Q))\subset B_{3}\]
and so  $\wt{B}_{3}$ is in a different component from $\wt{B}_{2}$ as well. Hence, for $\upsilon_{1}<\upsilon_{2}$, $(B_{1},B_{2})$ and $(B_{2},B_{3})$ are separated pairs for $R$. 
\item If $R\in { \wts}$, again, $B(2x-y,\kappa\ell(Q)/2)$ is contained in a component of $R$, say it is in $U_{1}(R)$. Since $U_{1}(R)$ and $U_{2}(R)$ have corkscrews by Definition \ref{WTS} (3), for $i=1,2$ we can find balls $\wt{B}_{i}\subset { (B_{i}\backslash \wt{E})}\cap U_{i}(R)$ of radius $\frac{p \kappa}{100}\ell(Q)=\frac{p \kappa}{100}\ell(R)$ and if we define $\wt{B}_{3}$ as before, we will still have $\wt{B}_{3}\subset B_{3}\subset U_{1}(R)$. For $\upsilon_{1}<\frac{\rho\kappa}{100}$, we again have $(B_{1},B_{2})$ and $(B_{2},B_{3})$ are separated pairs for $R$. 
\item If $R\in \cB$, then we can pick a dyadic cube of side length $b\ell(R)$ containing $x'$ whose boundary is in $C_{R}$. Again, it is not hard to find colinear corkscrew balls in $B_{i}$ that are all in different components of $(C_{R})^{c}$ and separated for $\upsilon_{1}$ small enough.
\end{enumerate}

\end{proof}

%
%

Now let $\cB_{\theta}' $ be the set of cubes in $\DD_{\tilde{\mu}}$ for which there are balls $B_{1}(Q),B_{2}(Q),B_{3}(Q)$ mutually disjoint such that
\begin{enumerate}
\item $2B_{i}(Q)\subset \wt{E}^{c}$,
\item $B_{2}$ is separated from $B_{1}$ and $B_{2}$,
\item $r_{B_{1}}=r_{B_{2}}=r_{B_{3}}=\frac{\upsilon_{2}}{2}\ell(Q)$
\item the centers of the $B_{i}$ are on a line parallel with $\theta$ with $B_{2}$ between $B_{1}$ and $B_{3}$.
\end{enumerate}
Let $\{\theta_{i}\}_{i=1}^{N(\delta)}$ be a maximal $\delta$-separated set in $\bS^{n}$. If $\delta$ is small enough, then for every $\theta\in \bS^{n}$, there is $i$ so that  $\cB_{\theta}\subset \cB_{\theta_{i}}'$. Hence, for $R\subset \wt{E}$,
\begin{align*}
\sum_{Q\subset R \atop Q\not\in \ls_{\kappa}} \tilde{\mu} (Q)
& \leq \sum_{i=1}^{N(\delta)} \sum_{Q\subset R\atop Q\in \cB_{\theta_{i}}'} \tilde{\mu} (Q)\lec \tilde{\mu} (R).
\end{align*}

Thus, $\wt{E}$ satisfies the local symmetry with constant $\kappa$ condition, which, by the results in \cite{DS1}, implies that $\wt{E}$ is uniformly rectifiable. 
\vv




\appendix
\section{} \label{secappendix}

\subsection{Proof of Theorem \ref{teoACF-elliptic}}
Without loss of generality, we assume $x=0$, and we denote $B_r=B(0,r)$. and $J(r) = J(0,r)$.
For $i=1,2$, we also set
$$J_i(r) = \frac{1}{r^{2}} \int_{B_r} \frac{|\grad u_i(y)|^{2}}{|y|^{n-1}}\,dy.$$
From Lemma \ref{lemcc**} and the fact that $u_i\in W^{1,2}(B(0,R))$ 
 it follows that $J_i(r)$ (and thus $J(r)$) is 
absolutely continuous. Further, for a.e. $0<r\leq R$ we have
$$J_i'(r) = \frac{1}{r^{2}} \int_{\partial B_r} \frac{|\grad u_i(y)|^{2}}{|y|^{n-1}}d\sigma(y)
- \frac2r\,J_i(r)
= \frac{1}{r^{n+1}} \int_{\partial B_r} |\grad u_i(y)|^{2}\,d\sigma(y)
- \frac2r\,J_i(r)
.$$
Therefore,
\begin{align}\label{eqalkg22}
\frac{J'(r)}{J(r)} &
= (\log J(r))' = (\log J_{1}(r)+\log J_{2}(r))' 
 = \frac{J_1'(r)}{J_1(r)} + \frac{J_2'(r)}{J_2(r)}\\
& = \frac{\ds  \frac{1}{r^{n+1}}\int_{\partial B_r} |\grad u_1(y)|^{2}\,d\sigma(y)}{J_1(r)} + 
\frac{\ds  \frac{1}{r^{n+1}}\int_{\partial B_r} |\grad u_2(y)|^{2}\,d\sigma(y)}{J_2(r)} - \frac4r.\nonumber
\end{align}

We write
\begin{align*}
\Delta(u_i^2) &= 2 \,|\nabla u_i|^2 + 2\,u_i\Delta u_i\\
& = 2\,|\nabla u_i|^2 - 2\,u_i\,\wt L u_i + 2\,u_i \,\vec b \cdot \nabla u_i + 2\, d \,u_i^2 - 2\,u_i\, \divv (\vec e \,u_i) + 2\,u_i\,\divv D\,\nabla u_i,
\end{align*}
where $D=Id-A$ and the above identities should be understood in the sense of distributions. Since $\wt L  u_i\leq0$, we derive
$$|\nabla u_i|^2 \leq \frac12 \Delta(u_i^2) + \,u_i \,(-\vec b \cdot \nabla u_i - d \,u_i^2 + \divv (\vec e \,u_i) -  u_i\,\divv D\,\nabla u_i).$$
Therefore, if $u_i$ is sufficiently smooth, we can write
\begin{align*}
r^2\,J_i(r) &= \int_{B_r} \frac{|\grad u_i(y)|^{2}}{|y|^{n-1}}\,dy\\
& \leq
 \frac12\int_{B_r} \frac{\Delta(u_i^2)(y)}{|y|^{n-1}}\,dy +
 \int_{B_r} \frac{u_i\,(-\vec b \cdot \nabla u_i - d\, u_i + \divv (\vec e u_i) - \divv D\,\nabla u_i)}{|y|^{n-1}}\,dy.
\end{align*}
By Green's theorem we have
\begin{align*}
\frac12 \int_{B_r} \frac{\Delta(u_i^2)(y)}{|y|^{n-1}}dy & = \frac12 \int_{\partial B_r} \partial_r(u_i^2) \,\frac1{|y|^{n-1}}\,
d\sigma - \frac12\int_{\partial B_r} \partial_r\frac1{|y|^{n-1}} \,u_i^2\,d\sigma\\
& = \frac1{r^{n-1}} \int_{\partial B_r} \Bigl(u_i\,\partial_r u_i + \frac{n-1}{2\,r}\,u_i^2\Bigr)\,d\sigma,
\end{align*}
and thus
\begin{align*}
r^2\,J_i(r)\leq \frac1{r^{n-1}} \int_{\partial B_r} &u_i\,\partial_r u_i\,d\sigma + \frac{n-1}{2\,r^n} \int_{\partial B_r} u_i^2\,d\sigma\\
&+ \int_{B_r} \frac{ u_i\,(-\vec b \cdot \nabla u_i - d\, u_i + \divv (\vec e u_i) - \divv D\,\nabla u_i)}{|y|^{n-1}}\,dy.
\end{align*}

If $u_i$ is not smooth enough, then $\Delta (u^2)$ is defined in the sense of distributions and we cannot argue as above.
In this case, for each $\delta>0$ we consider a radial $C^\infty$ function $\vphi_\delta$ such that $\chi_{A(0,\delta,r)} \leq \vphi_\delta
\leq \chi_{A(0,\delta/2,r+\delta)}$ and we write
\begin{align*}
r^2\,J_i(r) & =\lim_{\delta\to0+}\int \frac{|\grad u_i(y)|^{2}}{|y|^{n-1}}\,\vphi_\delta(y)\,dy \\
& \leq
 \limsup_{\delta\to0+}\frac12\int \frac{\Delta(u_i^2)}{|y|^{n-1}}\,\vphi_\delta\,dy + \limsup_{\delta\to0+}
 \int\frac{u_i\,(-\vec b \cdot \nabla u_i - d\, u_i + \divv (\vec e u_i) - \divv D\,\nabla u_i)}{|y|^{n-1}}\,\vphi_\delta\,dy.
\end{align*}
One can check that
\begin{equation}\label{eqclaimlimit}
\lim_{\delta\to 0+}\frac12\int \frac{\Delta(u_i^2)(y)}{|y|^{n-1}}\,\vphi_\delta(y)\,dy  = \frac1{r^{n-1}} \int_{\partial B_r}
\Bigl(u_i\,\partial_r u_i + \frac{n-1}{2\,r}\,u_i^2\Bigr)\,d\sigma
\end{equation}
for a.e. $r$. We defer the arguments for this fact to the end of this proof. Then we get
\begin{align*}
r^2\,J_i(r)\leq \frac1{r^{n-1}} \int_{\partial B_r}& \Bigl(u_i\,\partial_r u_i + \frac{n-1}{2\,r}\,u_i^2\Bigr)\,d\sigma \\
&+ \limsup_{\delta\to 0 +}
 \int \frac{u_i\,(-\vec b \cdot \nabla u_i - d\, u_i + \divv (\vec e u_i) - \divv D\,\nabla u_i)}{|y|^{n-1}}\,\vphi_\delta\,dy.
 \end{align*}

For $\gamma_i>0$ to be chosen below, we take now into account that
$$\int_{\partial B_r} u_i\,\partial_r u_i \,d\sigma\leq \frac12 \biggl(\frac{\gamma_i} r
\int_{\partial B_r} u_i^2 \,d\sigma + \frac r{\gamma_i} \int_{\partial B_r} (\partial_r u_i)^2 \,d\sigma\biggr),$$
and then, denoting 
$$\Sigma_{r,i} = \partial B_r\cap \{u_i >0\}$$
and
\begin{equation}\label{eqir1}
I_i(r) =  \limsup_{\delta\to 0 +}
 \int \frac{u_i\,(-\vec b \cdot \nabla u_i - d\, u_i + \divv (\vec e u_i) - \divv D\,\nabla u_i)}{|y|^{n-1}}\,\vphi_\delta\,dy,
\end{equation}
we deduce that
$$r^2\,J_i(r)\leq \frac1{2\,r^{n-2}} \int_{\Sigma_{r,i}} \Bigl(\frac 1{\gamma_i} (\partial_r u_i)^2 + \frac{\gamma_i + n-1}{r^2}\,u_i^2\Bigr)\,d\sigma + I_i(r).$$
Note now that
\begin{equation}\label{eqlambdai}
r^2\int_{\Sigma_{r,i}} |\nabla_{\partial B_r} u_i|^2 \,d\sigma\geq \lambda_i \int_{\Sigma_{r,i}} |u_i|^2\,d\sigma,
\end{equation}
where $\nabla_{\partial B_r}$ denotes the tangential gradient on the sphere $\partial B_r$ and $\lambda_i$ is the principal eigenvalue of the spherical Laplacian on the set 
$$\Sigma_{i}^{(r)} := r^{-1}\,\Sigma_{r,i}\subset \partial B_1.$$
Then we infer that
$$r^2\,J_i(r)\leq \frac1{2\,r^{n-2}} \int_{\partial B_r} \Bigl(\frac 1{\gamma_i} (\partial_r u_i)^2 + \frac{\gamma_i + n-1}{\lambda_i}\,|\nabla_{\partial B_r} u_i|^2\Bigr)\,d\sigma + I_i(r).$$
We choose $\gamma_i\geq0$ so that
\begin{equation}\label{eqgamma0}
\lambda_i = \gamma_i(\gamma_i + n-1),
\end{equation}
that is, $\gamma_i$ is the characteristic of $\Sigma_i^{(r)}$. We obtain then
\begin{equation}\label{eqcc482}
J_i(r)\leq \frac1{2\,\gamma_i\,r^{n}} \int_{\partial B_r} \!\bigl( (\partial_r u_i)^2 + |\nabla_{\partial B_r} u_i|^2\bigr)d\sigma + I_i(r) = \frac1{2\,\gamma_i\,r^{n}} \int_{\partial B_r}  |\nabla u_i|^2\,d\sigma + 
\frac{|I_i(r)|}{r^2}.
\end{equation}

Plugging the preceding estimate into \rf{eqalkg22}, we obtain
\begin{align}\label{eqgt5}
\frac{J'(r)}{J(r)} & \geq 
\frac{\ds \frac1r \int_{\Sigma_{1,r}} |\grad u_1|^{2}\,d\sigma}{\ds
\frac1{2\,\gamma_1} \int_{\Sigma_{1,r}}\!  |\nabla u_1|^2\,d\sigma + r^{n-2}\,|I_1(r)|} + 
\frac{\ds \frac1r \int_{\Sigma_{2,r}} |\grad u_2|^{2}\,d\sigma}{\ds
\frac1{2\,\gamma_2} \int_{\Sigma_{2,r}}  \!|\nabla u_2|^2\,d\sigma + r^{n-2}\,|I_2(r)|}
- \frac4r.
\end{align}
We can rewrite this inequality as follows
$$
\frac{J'(r)}{J(r)}  \geq \frac{2\,\gamma_1}r + \frac{2\,\gamma_2}r - \frac4r - E_1(r) - E_2(r),$$
where each error term $E_i(r)$ is defined by
$$
E_i(r)  =   \frac{2\,\gamma_i}r - \frac{\ds \frac1r \int_{\Sigma_{i,r}} |\grad u_i|^{2}\,d\sigma}{\ds
\frac1{2\,\gamma_i} \int_{\Sigma_{i,r}} \! |\nabla u_i|^2\,d\sigma + r^{n-2}\,|I_i(r)|}= 
\frac{2\,\gamma_i\,r^{n-3}\,|I_i(r)|}{\ds \frac1{2\,\gamma_i} \int_{\Sigma_{i,r}}  |\nabla u_i|^2\,d\sigma + r^{n-2}\,|I_i(r)|}.
$$
By the Friedland-Hayman inequality (see \cite[Chapter 12]{CS}), it turns out that
$$\gamma_1 + \gamma_2 \geq 2,$$
and thus
\begin{equation}\label{eqerri0}
\frac{J'(r)}{J(r)}  \geq  - E_1(r) - E_2(r).
\end{equation}
Hence, to prove \rf{eqww2} we have to estimate the error terms $E_i(r)$.

Plugging the estimate \rf{eqcc482} into the definition of $E_i(r)$, we obtain
\begin{equation}\label{eqerri}
E_i(r)  \leq  \frac{2\,\gamma_i\,r^{n-3}\,|I_i(r)|}{r^n\,J_i(r)} = \frac{2\,\gamma_i\,|I_i(r)|}{\ds r\,\int_{B_r} \frac{|\grad u_i(y)|^{2}}{|y|^{n-1}}\,dy}.
\end{equation}
Under the assumption that
\begin{equation}\label{eqfac998}
\frac1{r^n}\int_{\partial B_r} |\grad u_i(y)|^{2}\,d\sigma(y)\leq \frac5{r^2}\int_{B_r} \frac{|\grad u_i(y)|^{2}}{|y|^{n-1}}\,dy
\quad \mbox{ for $i=1$ and $i=2$,}
\end{equation}
we will show below that, for a.e.\ $r$,
\begin{equation}\label{eqfac10}
|I_i(r)| \leq C\,(1+\gamma_i^{-1})\,(1+K_r)\,w(0,r)\int_{B_r} \frac{|\grad u_i(y)|^{2}}{|y|^{n-1}}\,dy \quad \mbox{ for $i=1$ and $i=2$.}
\end{equation}

Let us see
how the estimate \rf{eqww2} follows from the preceding estimate. We distinguish several cases.
In the first one, we suppose that
\begin{equation*}\label{eqfac99}
\frac1{r^n}\int_{\partial B_r} |\grad u_i(y)|^{2}\,d\sigma(y)> \frac5{r^2}\int_{B_r} \frac{|\grad u_i(y)|^{2}}{|y|^{n-1}}\,dy
\end{equation*}
either for $i=1$ or $i=2$. 
In either situation, from \rf{eqalkg22} we infer that
\begin{equation}\label{eqfac3}
\frac{J'(r)}{J(r)}  \geq \frac1r,
\end{equation}
and thus \rf{eqww2} holds.

In the second case we suppose that \rf{eqfac998} holds
 and that $\gamma_j\geq 5$ for some $j$.
Then, from \rf{eqgt5} we derive
$$\frac{J'(r)}{J(r)}  \geq 
\frac{\ds \frac1r \int_{\Sigma_{j,r}} |\grad u_j|^{2}\,d\sigma}{\ds
\frac1{10} \int_{\Sigma_{j,r}}\!  |\nabla u_j|^2\,d\sigma + r^{n-2}\,|I_j(r)|} 
- \frac4r.$$
If $\frac1{10} \int_{\Sigma_{j,r}}  |\nabla u_j|^2\,d\sigma \geq r^{n-2}\,|I_j(r)|$,
 we deduce that
 \begin{equation}\label{eqfac4}
\frac{J'(r)}{J(r)}  \geq \frac5r-\frac4r = \frac1r,
\end{equation}
Otherwise, using also \rf{eqcc482}, we get
\begin{align*}
\frac{J'(r)}{J(r)} & \geq\frac1{
2\, r^{n-1}\,|I_j(r)|} 
 \int_{\Sigma_{j,r}} |\grad u_j(y)|^{2}\,d\sigma(y)- \frac4r  \geq \gamma_j \biggl(\frac{r\,J_j(r)}{|I_j(r)|} - \frac1r\biggr) -\frac4r 
\end{align*} 
By \rf{eqfac10} we have
$$
|I_j(r)| \lesssim (1+\gamma_i^{-1})\,(1+K_r)\,w(0,r) \, r^2\,J_j(r) \lesssim (1+K_r)\,w(0,r)\, r^2\,J_j(r),
$$
and so
$$\frac{J'(r)}{J(r)} \geq \gamma_j \biggl(\frac {c_{11}}{(1+K_r)\,w(0,r)\,r} - \frac1r\biggr) -\frac4r.$$
Hence, if $(1+K_r)\,w(0,r)\leq \frac12c_{11}$, since $\gamma_j\geq5$, this yields
 \begin{equation}\label{eqfac4'}
\frac{J'(r)}{J(r)} \geq 5 \biggl(\frac 2r - \frac1r\biggr) -\frac4r=\frac1r.
\end{equation}
In case that $(1+K_r)\,w(0,r)\geq \frac12c_{11}$, we just use use the trivial estimate $\frac{J'(r)}{J(r)}\geq -\frac4r$
(which follows from \rf{eqalkg22}), and we get
 \begin{equation}\label{eqfac4''}
\frac{J'(r)}{J(r)}\geq -\frac 8{c_{11}r} \,(1+K_r)\,w(0,r).
\end{equation}

In the third case we assume that \rf{eqfac998} holds and that
$\gamma_i\leq 5$ for $i=1$ and $i=2$. Then, from \rf{eqerri0}, \rf{eqerri}, and \rf{eqfac10}
 we obtain
\begin{align}\label{eqfac100}
\frac{J'(r)}{J(r)} & \geq  -  \,\sum_{i=1}^2
 \frac{2\,\gamma_i\,|I_i(r)|}{\ds r\,\int_{B_r} \frac{|\grad u_i(y)|^{2}}{|y|^{n-1}}\,dy}\\
&\geq  -  C\sum_{i=1}^2
 \frac{\gamma_i\,(1+\gamma_i^{-1})\,(1+K_r)\,w(0,r) }{r}\nonumber\\
&\geq  -  C\,\frac{(1+K_r)\,w(0,r)}{r}. \nonumber
\end{align}
Together with \rf{eqfac3}, \rf{eqfac4}, \rf{eqfac4'}, and \rf{eqfac4''}, this yields
\begin{equation}\label{eqfac11}
\frac{J'(r)}{J(r)}  \geq \frac1r\,\min\biggl(1, \,\,-C\,(1+K_r)\,w(0,r)\biggr) = -\frac Cr\,(1+K_r)\,w(0,r),
\end{equation}
which yields \rf{eqww2}.
\vv

\noi {\bf Proof of \rf{eqfac10}.} Recall that $I_i(r)$ is defined by the limit in \rf{eqir1}. For any $\delta>0$ we have
\begin{align}\label{eqir20}
 \int \frac{u_i\,\divv (D\,\nabla u_i)}{|y|^{n-1}}\,\vphi_\delta\,dy & = 
\int \divv\biggl(\frac{u_i\,D\,\nabla u_i}{|y|^{n-1}}\biggr)\,\vphi_\delta\,dy - \int \frac{\nabla u_i\,D\,\nabla u_i}{|y|^{n-1}}\,\vphi_\delta\,dy\\
& \quad
- \int u_i\,\nabla \frac1{|y|^{n-1}}\,D\,\nabla u_i\,\vphi_\delta\,dy \nonumber\\
& =: S_a(\delta) + S_b(\delta) + S_c(\delta).\nonumber
\end{align}
We first estimate the integral $S_a$. Integrating by parts and taking into account that 
\begin{equation}
\label{D<w}
|D(y)|\leq w(0,r+\delta)\lesssim1\mbox{ in }B(0,r+\delta)
\end{equation} 
(because $D=A-Id=A-A(0)$), $|\nabla \vphi_\delta|\lesssim 1/\delta$, 
and that $\supp\nabla\vphi_\delta\subset A(0,\delta/2,\delta)\cup A(0,r,r+\delta)$, we get
\begin{align}\label{eqasj43}
|S_a(\delta)| &=  \left|\int \frac{u_i\,D\,\nabla u_i}{|y|^{n-1}}\nabla\vphi_\delta\,dy \right|\\
& \lesssim \frac{1}{\delta^n}\,\|u_i\|_{\infty,B_{\delta}}\int_{B_\delta} |\nabla u_i|\,dy
+
\frac{w(0,r+\delta)}{r^{n-1}\,\delta}\int_{B_{r+\delta}\setminus B_r} u_i\,|\nabla u_i|\,dy.\nonumber
\end{align}
To estimate the first term on the right hand side we use Caccioppoli's inequality\footnote{Following the standard proof, one can show the validity of Caccioppoli's inequality in this generality, which for radii $r \leq 1$ reads exactly as the one for operators with no lower order terms.}:
\begin{align*}
\frac{1}{\delta^n}\,\|u_i\|_{\infty,B_{\delta}}\int_{B_\delta} |\nabla u_i|\,dy
& \lesssim \delta\,\|u_i\|_{\infty,B_{\delta}}\left(\avint_{B_\delta} |\nabla u_i|^2\,dy\right)^{1/2}\\
&
\lesssim \|u_i\|_{\infty,B_{\delta}}\left(\avint_{B_{2\delta}} |u_i|^2\,dy\right)^{1/2}\leq \|u_i\|_{\infty,B_{2\delta}}^2,
\end{align*}
which tends to $0$ as $\delta\to0$.
Concerning the last term on the right hand side of \rf{eqasj43}, by the Lebesgue differentiation theorem, it is easy to check that
$$\lim_{\delta\to0+}\frac{1}\delta\,\int_{B(0,r+\delta)\setminus B(0,r)} u_i\,|\nabla u_i|\,dy = c_n \int_{\partial B_r} u_i\,|\nabla u_i|\,d\sigma 
\quad\mbox{ for a.e. $r\in (0,R)$.}$$
Hence, for a.e.\ $r$, using \rf{eqlambdai},
the assumption \rf{eqfac998} and the fact that $\lambda_i>\gamma_i^2$ (by \rf{eqgamma0}),
\begin{align*}
\limsup_{\delta\to0+}|S_a(\delta)| & \lesssim 
\frac{w(0,r)}{r^{n-1}}\,\left(\int_{\partial B_r} u_i^2\,d\sigma\right)^{1/2}\,
\left(\int_{\partial B_r} |\nabla u_i|^{2}\,d\sigma\right)^{1/2}\\
& \lesssim \frac{w(0,r)}{r^{n-2}\,\lambda_i^{1/2}}\,\int_{\partial B_r} |\nabla u_i|^2\,d\sigma\\
&\leq \frac{w(0,r)}{\gamma_i}\int_{B_r} \frac{|\nabla u_i|^2}{|y|^{n-1}}\,dy.
\end{align*}

Regarding the term $S_b(\delta)$ in \rf{eqir20}, we have by \eqref{D<w}
$$|S_b(\delta)| = \left|\int \frac{\nabla u_i\,D\,\nabla u_i}{|y|^{n-1}}\,\vphi_\delta\,dy\right|
\leq w(0,r+\delta)\int_{B_{r+\delta}}\frac{|\nabla u_i|^2}{|y|^{n-1}}\,dy.
$$
Therefore,
$$\limsup_{\delta\to0+}|S_b(\delta)| \leq w(0,r)\int_{B_r}\frac{|\nabla u_i|^2}{|y|^{n-1}}\,dy.$$

Moreover, we turn our attention to the term $S_c(\delta)$.
By Cauchy-Schwarz, we obtain
\begin{align*}
|S_c(\delta)| & = \left|\int u_i\,\nabla \frac1{|y|^{n-1}}\,D\,\nabla u_i\,\vphi_\delta\,dy\right|
\lesssim  \int_{B_{r+\delta}} \frac{w(0,|y|)\,u_i\,|\nabla u_i|}{|y|^{n}}\,dy\\
&
\leq  \left(\int_{B_{r+\delta}}\frac{w(0,|y|)^2\,u_i(y)^2}{|y|^{n+1}}\,dy\right)^{1/2}
\left(\int_{B_{r+\delta}}\frac{|\nabla u_i(y)|^2}{|y|^{n-1}}\,dy\right)^{1/2}.
\end{align*}
So we have
\begin{align*}
\limsup_{\delta\to0+}|S_c(\delta)| & \lesssim 
\left(\int_{B_{r}}\frac{w(0,|y|)^2\,u_i(y)^2}{|y|^{n+1}}\,dy\right)^{1/2}
\left(\int_{B_r}\frac{|\nabla u_i(y)|^2}{|y|^{n-1}}\,dy\right)^{1/2}\\
& \leq K_r\,w(0,r)\int_{B_r} \frac{|\nabla u_i(y)|^2}{|y|^{n-1}}\,dy.
\end{align*}
Arguing as in $S_c(\delta)$ we can show that
\begin{align*} 
\limsup_{\delta\to0+} \Big| \int \frac{u_i\,\vec b \cdot \nabla u_i}{|y|^{n-1}}\,\vphi_\delta\,dy \Big|&\lesssim K_r\,w(0,r)\int_{B_r} \frac{|\nabla u_i(y)|^2}{|y|^{n-1}}\,dy,\\
\limsup_{\delta\to0+} \Big| \int \frac{d \,u_i^2}{|y|^{n-1}}\,\vphi_\delta\,dy \Big|&\lesssim K_r\,w(0,r)\int_{B_r} \frac{|\nabla u_i(y)|^2}{|y|^{n-1}}\,dy,
\end{align*}
and also,
$$
\limsup_{\delta\to0+} \Big| \int \frac{u_i\,\divv(\vec e u_i)}{|y|^{n-1}}\,\vphi_\delta\,dy \Big|\lesssim K_r\,w(0,r)\int_{B_r} \frac{|\nabla u_i(y)|^2}{|y|^{n-1}}\,dy.
$$

Combining the latter inequalities with  the estimates for $S_a(\delta)$, $S_b(\delta)$ and $S_c(\delta)$, and noting that $1+\gamma_i^{-1}+K_r  \leq (1+\gamma_i^{-1})(1+K_r)$, we obtain  \rf{eqfac10}. This concludes the proof of the theorem, modulo the arguments for the identity \rf{eqclaimlimit}.
\vv

\noi{\bf Proof of \rf{eqclaimlimit}.} Recall that we have to show that, for a.e.\ $r$,
\begin{equation}\label{eqclaimlimit'}
\lim_{\delta\to 0}\frac12\int \frac{\Delta(u_i^2)(y)}{|y|^{n-1}}\,\vphi_\delta(y)\,dy  = \frac1{r^{n-1}} \int_{\partial B_r}
\Bigl(u_i\,\partial_r u_i + \frac{n-1}{2\,r}\,u_i^2\Bigr)\,d\sigma
\end{equation}
We have (recalling $|\cdot |^{1-n}$ is harmonic away from $0\not\in \supp \vphi_{\delta}$)
\begin{align*}
\frac12\int \frac{\Delta(u_i^2)(y)}{|y|^{n-1}}\,\vphi_\delta(y)\,dy & =
\frac12\int u_i^2\, \Delta\Bigl(\vphi_\delta\frac1{|y|^{n-1}}\Bigr)\,dy\\
& = \int u_i^2\, \nabla\vphi_\delta\,\cdot \nabla \frac1{|y|^{n-1}}\,dy +
\frac12\int u_i^2\, \Delta\vphi_\delta\,\frac1{|y|^{n-1}}\,dy.
\end{align*}

We write
$\vphi_\delta = \psi_\delta - \wt\psi_\delta,$
where $\psi_\delta$ and  $\wt\psi_\delta$ are radial $C^\infty$ functions such that $\chi_{B(0,r)} \leq \psi_\delta \leq \chi_{B(0,r+\delta)}$ and $\chi_{B(0,\delta/2)} \leq \wt\psi_\delta \leq \chi_{B(0,\delta)}$.
Note that 
$$\left|\int u_i^2\, \nabla\wt\psi_\delta\,\cdot \nabla \frac1{|y|^{n-1}}\,dy\right| 
\lesssim \|u_i\|_{\infty,B_\delta}^2\,\frac1\delta\,\frac1{\delta^n} \, |B_\delta|
\approx  \|u_i\|_{\infty,B_\delta}^2\to 0\quad \mbox{ as $\delta\to0$.}$$
By analogous estimates, 
$$\frac12\int u_i^2\, \Delta\wt\psi_\delta\,\frac1{|y|^{n-1}}\,dy\to 0\quad \mbox{ as $\delta\to0$.}$$
Therefore, \rf{eqclaimlimit'} is equivalent to 
\begin{equation}\label{eqclaimlimit''}
 \int u_i^2\, \nabla\psi_\delta\,\cdot \nabla \frac1{|y|^{n-1}}\,dy +
\frac12\int u_i^2\, \Delta\psi_\delta\,\frac1{|y|^{n-1}}\,dy
 \to \frac1{r^{n-1}} \int_{\partial B_r}
\Bigl(u_i\,\partial_r u_i + \frac{n-1}{2\,r}\,u_i^2\Bigr)\,d\sigma
\end{equation}
as $\delta\to0+$.

We consider first the first integral on the left hand side of \rf{eqclaimlimit''}. Since
$\psi_\delta$ is radial, we have
$$ \int u_i^2\, \nabla\psi_\delta\,\cdot \nabla \frac1{|y|^{n-1}}\,dy =
 \int u_i^2\, \partial_r\psi_\delta\,\,\partial_r\frac1{|y|^{n-1}}\,dy =
- \int u_i^2\, \partial_r\psi_\delta \frac{n-1}{|y|^n} \,dy.$$
We claim that given any function $f\in L^1_{loc}(\R^{n+1})$,
\begin{equation}\label{eqcla88}
\lim_{\delta\to0+}
\int f\, \partial_r\psi_\delta\,dy = \int_{\partial B_r} f\,d\sigma\quad
\mbox{for a.e.\ $r$.}
\end{equation}
To check this, taking into account that $\partial_r\psi_\delta$ is radial, we write
\begin{align*}
\left|\int f\, \partial_r\psi_\delta\,dy - \int_{\partial B_r} f\,d\sigma\right|
& = \left|\int_r^{r+\delta} \left(\int_{\partial B_s} f\,d\sigma - \int_{\partial B_r} f\,d\sigma\right)\,
(-\partial_r\psi_\delta(s))\,ds\right|\\
& \leq \frac c\delta
\int_r^{r+\delta} \left|\int_{\partial B_s} f\,d\sigma - \int_{\partial B_r} f\,d\sigma\right|\,ds,
\end{align*}
which tends to $0$ for a.e.\ $r$ as $\delta\to 0$, by the Lebesgue differentiation theorem
applied to the function $s\mapsto\int_{\partial B_s} f\,d\sigma$, $s\in(0,\infty)$.
Thus the claim \rf{eqcla88} is proved and we infer that
\begin{equation}\label{eqcla98}
 \int u_i^2\, \nabla\psi_\delta\,\cdot \nabla \frac1{|y|^{n-1}}\,dy\to 
\frac{n-1}{r^n} \int_{\partial B_r} u_i^2\,d\sigma \quad\mbox{ for a.e.\ $r$}.
\end{equation}

Next we turn our attention to the second term on the left hand side of \rf{eqclaimlimit''}. 
By Taylor's formula applied to the function $s\mapsto \frac1{s^{n-1}}$, it follows that
$$\frac1{|y|^{n-1}} = \frac1{r^{n-1}} - \frac{n-1}{r^n}\,(|y|-r)+ O((|y|-r)^2).$$
Thus we have
\begin{equation*}
\frac12\int u_i^2\, \Delta\psi_\delta\,\frac1{|y|^{n-1}}\,dy
= 
\frac12\int u_i^2\, \Delta\psi_\delta\,\biggl(\frac1{r^{n-1}} - \frac{n-1}{r^n}\,(|y|-r)+ O(\delta^2)\biggr)\,dy.
\end{equation*}
Using that $|\Delta\psi_\delta|\lesssim 1/\delta^2$ and that $\Delta\psi_\delta$ is supported on 
$\overline{ B_{r+\delta}}\setminus B_r$, it follows easily that 
$$O(\delta^2)\int u_i^2\, \Delta\psi_\delta\,dy\to 0 \quad \mbox{ as $\delta\to0$.}$$
Hence we only have to deal with the integral
$$\frac12\int u_i^2\, \Delta\psi_\delta\,\biggl(\frac1{r^{n-1}} - \frac{n-1}{r^n}\,(|y|-r)\biggr)\,dy.$$
Since $u_i^2\,\left(\frac1{r^{n-1}} - \frac{n-1}{r^n}\,(|y|-r)\right)$ belongs to $W^{1,2}(B_R)$, 
the preceding integral equals
\begin{align*}
-\frac12\int \nabla \psi_\delta & \cdot \nabla\biggl(u_i^2\,\biggl(\frac1{r^{n-1}} - \frac{n-1}{r^n}\,(|y|-r)\biggr)\biggr)\,dy \\
& = 
-\int  u_i\,\,\partial_r u_i\,\partial_r \psi_\delta\,\biggl(\frac1{r^{n-1}} - \frac{n-1}{r^n}\,(|y|-r)\biggr)\,dy + 
\frac{n-1}{2\,r^n}\int u_i^2\,\partial_r \psi_\delta  \,dy.
\end{align*}
Using that $|\nabla \psi_\delta|\lesssim 1/\delta$ and that this is supported on 
$\overline{ B_{r+\delta}}\setminus B_r$, it is immediate to check that 
$$\int  u_i\,\,\partial_r u_i\,\partial_r \psi_\delta\,\frac{n-1}{r^n}\,(|y|-r)\,dy
\to 0 \quad \mbox{ as $\delta\to0$.}$$
Therefore,
$$\lim_{\delta\to 0} \frac12\!\int u_i^2\, \Delta\psi_\delta\,\frac1{|y|^{n-1}}\,dy
= \lim_{\delta\to 0} \frac{-1}{r^{n-1}}\int  u_i\,\,\partial_r u_i\,\partial_r \psi_\delta \,dy + 
\lim_{\delta\to 0}
\frac{n-1}{2\,r^n}\!\int u_i^2\,\partial_r \psi_\delta  \,dy.$$
Applying now \rf{eqcla88}, it follows that, for a.e.\ $r$,
$$\lim_{\delta\to 0} \frac12\int u_i^2\, \Delta\psi_\delta\,\frac1{|y|^{n-1}}\,dy
=  \frac1{r^{n-1}}\int_{\partial B_r}  u_i\,\,\partial u_r\,d\sigma - 
\frac{n-1}{2\,r^n}\int_{\partial B_r}  u_i^2 \,d\sigma.$$
Together with \rf{eqcla98}, this proves \rf{eqclaimlimit''} and concludes the proof of \rf{eqclaimlimit}.

\vvv

We now turn our attention to the second part of Theorem \ref{teoACF-elliptic}. We have to show that, for $i=1,2$,
\begin{equation}\label{eqprec29}
\int_{B(x,r)}\frac{u_i(y)^2}{|y-x|^{n+1}}\,dy\lesssim\int_{B(x,r)} \frac{|\grad u_i(y)|^{2}}{|y-x|^{n-1}}dy.
\end{equation}
To this end,
 we assume that $x=0$ and we use the fact that, for a.e. $y\in B_r$,
$$u_i(y)^2 \leq \int_0^{1}|\nabla (u_i^2)(ty)|\,|y|\,dt = 2\int_0^{|y|}u_i\Bigl(t\frac{y}{|y|}\Bigr)\,
\Bigl|\nabla u_i\Bigl(t\frac y{|y|}\Bigr)\Bigr|\,dt.
$$
Therefore, we get
\begin{align}\label{eqtruc22}
\int_{B_{r}} \frac{u_i(y)^2}{|y|^{n+1}}\,dy  & \leq 2 
\int_{B_{r}} \frac1{|y|^{n+1}}\,\int_0^{|y|}u_i\Bigl(t\frac{y}{|y|}\Bigr)\,
\Bigl|\nabla u_i\Bigl(t\frac y{|y|}\Bigr)\Bigr|\,dt\,dy\\
& = 
2 
 \int_{y'\in\partial B_1}\int_0^r \frac{1}{s}\,
\int_0^{s}u_i(t\,y')\,
|\nabla u_i(t\,y')|\,dt\,ds\,d\sigma(y')\nonumber\\
& = 2 
 \int_{y'\in\partial B_1}\int_0^r u_i(t\,y')\,
|\nabla u_i(t\,y')|\,\left(\int_t^r 
 \frac{1}{s}\,ds\right)
\,dt\,d\sigma(y')\nonumber\\
&  = 2  \int_{y'\in\partial B_1}\int_0^r u_i(t\,y')\,
|\nabla u_i(t\,y')|\,\log\frac rt\,dt\,d\sigma(y')\nonumber\\
& = 2  \int_{B_r} \frac{u_i(y)\, |\nabla u_i(y)|}{|y|^n}\,\log \frac r{|y|}\,dy\nonumber\\
&\leq 2\left(\int_{B_r} \frac{|\grad u_i(y)|^{2}}{|y|^{n-1}}\,dy\right)^{1/2}
\left(\int_{B_r} \frac{u_i(y)^{2}}{|y|^{n+1}}\,\log^2 \!\frac r{|y|}\,dy\,\right)^{1/2}
.\nonumber
\end{align}
To estimate the last integral we split it and we use the assumption \rf{eqACF1} on $B_{r/2}$:
\begin{align*}
\int_{B_r} \frac{u_i(y)^{2}}{|y|^{n+1}}\,\log^2 \!\frac r{|y|}\,dy & = 
\int_{B_r\setminus B_{r/2}} \frac{u_i(y)^{2}}{|y|^{n+1}}\,\log^2 \!\frac r{|y|}\,dy + \int_{B_{r/2}} \frac{u_i(y)^{2}}{|y|^{n+1}}\,\log^2 \!\frac r{|y|}\,dy\\
& \lesssim \int_{B_r\setminus B_{r/2}} \frac{u_i(y)^{2}}{|y|^{n+1}}\,dy + \|u_i\|_{\infty,B_{r/2}}^2\int_{B_{r/2}}  w_0 \Big(\frac{|y|}r\Big)^{2}
\frac{1}{|y|^{n+1}}\,\log^2 \!\frac r{|y|}dy\\
& \lesssim \int_{B_r} \frac{u_i(y)^{2}}{|y|^{n+1}}\,dy + C_1\,C_{w_0} \|u_i\|_{\infty,B_{r/2}}^2,
\end{align*}
where in the last inequality we just changed variables. Since $u_i$ is $L$-subharmonic, we have
$$\|u_i\|_{\infty,B_{r/2}} \lesssim \,\avint_{B_r} u_i(y)\,dy \leq \left(\,\avint_{B_r} u_i(y)^2\,dy\right)^{1/2} \leq \left(\int_{B_r} \frac{u_i(y)^2}{|y|^{n+1}}\,dy\right)^{1/2},$$
and thus
$$\int_{B_r} \frac{u_i(y)^{2}}{|y|^{n+1}}\,\log^2 \!\frac r{|y|}\,dy\lesssim \int_{B_r} \frac{u_i(y)^2}{|y|^{n+1}}\,dy.$$
Plugging this estimate into \rf{eqtruc22}, we obtain \rf{eqprec29}.
\fiproof

\vvv


\subsection{Proof of Theorem \ref{teoACFprecise}}

The inequality \rf{eqprec1} follows just by inspecting the proof of Theorem \ref{teoACF-elliptic}.
The fact that $\gamma_1 + \gamma_2\geq2$ is just the Friedland-Hayman inequality mentioned above, and the 
estimate \rf{eqprec2} is shown in Corollary 12.4 of \cite{CS}.
\fiproof
\vv

\vv


\end{document}